\documentclass{article}

\usepackage[english]{babel}

\usepackage[letterpaper,top=2cm,bottom=2cm,left=3cm,right=3cm,marginparwidth=1.75cm]{geometry}

\usepackage{graphicx}
\usepackage[colorlinks=true, allcolors=blue]{hyperref}
\usepackage{amsmath, amssymb, amsthm, amsfonts}
\usepackage{fullpage}
\usepackage{hyperref}
\usepackage[shortlabels]{enumitem}
\usepackage{mathtools}
\usepackage[capitalise]{cleveref}
\usepackage[title]{appendix}
\usepackage{thmtools, thm-restate}
\usepackage{tikz}
\usetikzlibrary{positioning, calc}
\usepackage{caption}

\theoremstyle{plain}
\newtheorem{thm}{Theorem}[section]
\newtheorem{prop}[thm]{Proposition}
\newtheorem{lemma}[thm]{Lemma}
\newtheorem{conj}[thm]{Conjecture}
\newtheorem{cor}[thm]{Corollary}

\newtheorem{assumption}[thm]{Assumption}
\newtheorem{problem}[thm]{Problem}
\theoremstyle{remark}

\theoremstyle{definition}
\newtheorem{definition}[thm]{Definition}

\DeclarePairedDelimiter\ceil{\lceil}{\rceil}
\DeclarePairedDelimiter\floor{\lfloor}{\rfloor}

\newcommand{\NN}{{\mathbb N}}

\newcommand{\EE}{{\mathbb E}}

\newcommand{\E}{\mathbb{E}}

\newcommand{\sm}{\setminus}

\newcommand{\cG}{\mathcal{G}}

\newcommand{\cL}{\mathcal{L}}

\DeclarePairedDelimiter\norm{\lVert}{\rVert}%
\DeclarePairedDelimiter\abs{\lvert}{\rvert}%
\makeatletter
\let\oldabs\abs
\def\abs{\@ifstar{\oldabs}{\oldabs*}}
\let\oldnorm\norm
\def\norm{\@ifstar{\oldnorm}{\oldnorm*}}
\makeatother

\DeclareMathOperator{\bip}{bip}
\DeclareMathOperator{\pc}{pc}

\title{Polynomial-to-exponential transition in $3$-uniform Ramsey numbers}
\author{Ruben Ascoli\thanks{School of Mathematics, Georgia Institute of Technology, Atlanta, GA 30332. Email: rascoli3@gatech.edu.} \and
Xiaoyu He\thanks{School of Mathematics, Georgia Institute of Technology, Atlanta, GA 30332. Email: xhe399@gatech.edu.} \and
Hung-Hsun Hans Yu\thanks{Department of Mathematics, Princeton University, Princeton, NJ 08544. Email: hansonyu@princeton.edu. }}
\date{\today}

\begin{document}

\maketitle

\begin{abstract}
Let $r_k(s, e; t)$ denote the smallest $N$ such that any red/blue edge coloring of the complete $k$-uniform hypergraph on $N$ vertices contains either $e$ red edges among some $s$ vertices, or a blue clique of size $t$. Erd\H os and Hajnal introduced the study of this Ramsey number in 1972 and conjectured that for fixed $s>k\geq 3$, there is a well defined value $h_k(s)$ such that $r_k(s, h_k(s)-1; t)$ is polynomial in $t$, while $r_k(s, h_k(s); t)$ is exponential in a power of $t$. Erd\H os later offered \$500 for a proof.
Conlon, Fox, and Sudakov proved the conjecture for $k=3$ and $3$-adically special values of $s$, and Mubayi and Razborov proved it for $s > k \geq 4$. We prove the conjecture for $k=3$ and all $s$, settling all remaining cases of the problem. %\xh{I am not sure if we want to keep the commented out lines after this.} % We also provide a new, shorter proof of the cases $k \geq 7$. Our proofs appeal to recent results on Ramsey numbers of tightly connected hypergraphs to reduce the problem to maximizing the number of edges in a $k$-uniform hypergraph whose tight components are $k$-partite.
%We do this by solving the following novel extremal problem exactly for all $n$: what is the maximum number of edges in an $n$-vertex $3$-uniform hypergraph such that all tight components of $H$ are tripartite?
%We do this by showing that for every $n$, the balanced iterated blowup of an edge has the maximum number of edges among $3$-uniform hypergraphs on $n$ vertices with all tight components tripartite. 
%\ra{TODO: decide on terminology -- balanced iterated blowup of an edge versus balanced complete iterated $k$-partite $k$-graph} %\hy{I think I would combine the current version with the commented-out version: We do this by considering the following novel Tur\'an-type problem exactly for all $n$: what is the maximum number of edges in an $n$-vertex $3$-uniform hypergraph such that all tight components of $H$ are tripartite?
%We show that balanced iterated blowup of an edge maximizes the number of edges, which may be of independent interest (or something like this---it depends on how hard we want to sell our result as a Tur\'an-type result).} \xh{I like Hans' wording modified slightly: 
We do this by solving a novel Tur\'an-type problem: what is the maximum number of edges in an $n$-vertex $3$-uniform hypergraph in which all tight components are tripartite?
We show that the balanced iterated blowup of an edge is an exact extremizer for this problem for all $n$.
%}
\end{abstract}

\section{Introduction}
The Ramsey number $r_k(s, t)$ is the smallest integer $N$ such that any red/blue coloring of the edges of $K_N^{(k)}$, the complete $k$-uniform hypergraph (or $k$-graph for short), contains either a red copy of $K_s^{(k)}$ or a blue copy of $K_t^{(k)}$.
The graph case $k=2$ has been the subject of intensive study and is one of the most active areas in modern combinatorics. Recent breakthroughs include the celebrated exponential improvement on the upper bound for the ``diagonal" Ramsey number $r_2(t, t)$ by Campos, Griffiths, Morris, and Sahasrabudhe \cite{CGMS}, as well as the establishment of the asymptotics (up to polylogarithmic factors) of the ``off-diagonal" Ramsey number $r_2(4, t)$ by Mattheus and Verstra\"ete \cite{MV24}. Despite the massive body of work, there remain polynomial gaps between lower and upper bounds on $r_2(s, t)$ for all fixed $s \geq 5$. 

Even less is known about hypergraph Ramsey numbers, the numbers $r_k(s, t)$ for $k\geq 3$. Towards the off-diagonal case of this problem, Erd\H os and Hajnal \cite{EH72} introduced the following more general definition in 1972.

\begin{definition} The Ramsey number $r_k(s, e; t)$ is the smallest integer $N$ such that any red/blue coloring of $K_N^{(k)}$ contains either $e$ red edges among some $s$ vertices, or a blue clique of size $t$. In particular, $r_k(s, \binom{s}{k}; t) = r_k(s, t)$, so this definition includes the classical Ramsey numbers.
\end{definition}

Erd\H os and Hajnal conjectured that for each $k\geq 3$ and for each fixed $s$, there is a transition point in $e$ where $r_k(s, e; t)$ jumps from being polynomial in $t$ to being exponential in a power of $t$. %They proved some small cases of this statement using the probabilistic method and connected the problem to other conjectures in hypergraph Ramsey theory.
To state precisely their conjecture, define the function $g_k(n)$ to be the maximum number of edges in an $n$-vertex iterated blowup of an edge. It satisfies $g_k(n) = 0$ for $n < k$, and for $n\geq k$, $$g_k(n) = \max_{s_1+\ldots+s_k = n} \left(\prod_{j=1}^k s_j + \sum_{j=1}^k g_k(s_j)\right),$$ where the maximum is over compositions of $n$ into $k$ positive integer parts. 
%\hy{I think the equation here is not that helpful. I'd either only say that they have a precise conjectural $g_k(n)$ or say that it is the maximum number of edges in an iterated blowup of an edge.} \xh{How is this?}
%\begin{definition} The family of \emph{iterated $k$-partite hypergraphs} is the minimal collection $\mathcal F_k$ of $k$-graphs which contains the following elements. The empty $k$-graphs on $1, 2, \ldots, k-1$ vertices are in $\mathcal F_k$, and the single edge is also in $\mathcal F_k$. If $H\in\mathcal F_k$, then the hypergraph obtained by replacing a vertex $v$ of $H$ with any hypergraph $G\in \mathcal F_k$ is also in $\mathcal F_k$. Here by replacing we mean that the new vertex set is $(V(H)\sm\{v\})\cup V(G)$, and the new edge set includes all edges of $H$ that do not use $v$ as well as the edge $(e\sm\{v\})\cup\{u\}$ for each edge $e\in E(H)$ that does use $v$ and for each vertex $u\in V(G)$.
%\end{definition}
%\begin{definition}[Equivalent definition of iterated $k$-partite hypergraphs] We refer to a complete $k$-partite $k$-graph as the \emph{blowup of an edge}. An \emph{iterated blowup of an edge} is formed from the blowup of an edge by (optionally) replacing each part of the blowup with another iterated blowup of an edge. We say that a $k$-graph is an \emph{iterated $k$-partite hypergraph} if it is contained in the iterated blowup of an edge.\end{definition}
%\ra{The first definition above is paraphrased from \cite{CFS}, while the second from \cite{CFG+Y}. Which one should we go with?}

Erd\H os and Hajnal \cite{EH72} showed that $r_k(s, g_k(s); t) = t^{O(1)}$, and they conjectured that requiring even one more red edge among $s$ vertices would cause the Ramsey number to grow exponentially in a power of $t$.
\begin{conj}[Erd\H os and Hajnal \cite{EH72}]\label{conj:main}
For all $s > k \geq 3$, we have $r_k(s, g_k(s)+1; t) = 2^{t^{\Omega(1)}}$.
\end{conj}

Erd\H os reiterated the problem in 1990 \cite[Equation (28)]{Erd90} and offered \$500 for its solution.\footnote{In his paper \cite{Erd90}, Erd\H os discussed this problem along with several related conjectures, ambiguously offering \$500 ``for a proof or disproof of these conjectures." To our knowledge, the \$500 prize became attached to this specific problem in Chung's collection of Erd\H os problems \cite{Chung97}.} However, 
little progress was made on the problem until 2010, when Conlon, Fox, and Sudakov~\cite{CFS} proved \cref{conj:main} for infinitely many $s$ when $k = 3$.
It turns out that the cases that Conlon, Fox and Sudakov solved are precisely when $s$ is of the form $3^n$, $3^m+3^n$ or $3^m-3^n$ (see \cref{prop:sum-and-diff-of-exp} for a proof).
In 2021, Mubayi and Razborov~\cite{MR21} proved \cref{conj:main} for all $s>k\geq 4$ by reducing the conjecture to a certain combinatorial optimization problem about inducibility of rainbow tournaments. Their techniques become stronger as the uniformity $k$ grows larger, and the case $k=4$ required the use of computer verification. For the unsolved case $k=3$, they noted that ``many crucial calculations in this paper completely fall apart."

Our main theorem establishes the polynomial-to-exponential transition of $r_3(s, e; t)$ as conjectured by Erd\H os and Hajnal, completing the proof of \cref{conj:main}.

\begin{thm}\label{thm:main} For all $s>3$, we have $r_3(s, g_3(s)+1; t) = 2^{\Omega(t^{2/3})}$. 
\end{thm}

We also provide a short new proof of \cref{conj:main} for uniformities $k\geq 7.$
Explicitly, we show that $r_k(s, g_k(s)+1; t) = 2^{\Omega(t^{2/k})}$ for all $s>k\geq 7$.

A key ingredient to our proof both for $k=3$ and for $k\geq 7$ comes from recent work on Ramsey numbers of \emph{tightly connected hypergraphs} by the second and third authors and Conlon, Fox, Gunby, Mubayi, Suk, and Verstra\"ete~\cite{CFG+Y}, on which we elaborate in \cref{sec:tight components}. We reduce \cref{thm:main} to a Tur\'an-type problem about tightly connected hypergraphs, which have recently been the object of intense study. Breakthrough results of Balogh and Luo~\cite{BL23} and Lidick\'y, Mattes, and Pfender~\cite{LMP24} established the Turán density of the $3$-uniform tight cycle minus an edge, while Kam\v{c}ev, Letzter, and Pokrovskiy~\cite{KLP23} and Bodn\'ar, Le\'on, Liu, and Pikhurko~\cite{BLLP} determined the Tur\'an density of the tight cycle. 

\begin{definition}
For $t < k$, we say that a $k$-graph is \emph{$t$-tightly connected} if it has no isolated vertices and any two edges $e$ and $f$ can be joined by a sequence of edges $e = e_0, e_1, \ldots, e_s = f$ such that $|e_{i-1}\cap e_i| \geq t$ for each $i\in[s]$. Define a \emph{$t$-tight component} of $H$ to be a maximal subgraph of $H$ which is $t$-tightly connected. Since $t$-tight components form equivalence classes, the edges of a $k$-graph can be partitioned into $t$-tight components. For $3$-graphs, we shorten the terms $2$-tightly connected and $2$-tight components to just tightly connected and tight components.
\end{definition}

The main innovation in our proof is to reduce \cref{thm:main} to the following Tur\'an-type result, which we believe to be of independent interest.

\begin{thm}\label{thm:iter-k-partite}
For $k=3$ and $k\geq 7$, the maximum number of edges in an $n$-vertex $k$-graph $H$ with the property that all 2-tight components of $H$ are $k$-partite is $g_k(n)$.
\end{thm}

Equivalently, if $\mathcal{T}_k$ is the family of all $2$-tightly connected non-$k$-partite $k$-graphs, then \cref{thm:iter-k-partite} states that $\textnormal{ex}(n, \mathcal{T}_k) = g_k(n)$ exactly, where the extremal number $\textnormal{ex}(n,\mathcal{T}_k)$ is the maximum number of edges in a $k$-graph that contains no element of $\mathcal{T}_k$ as a subgraph. To our knowledge, this is one of the first nontrivial hypergraph problems to be solved exactly where the extremizer is an iterated blowup. The interested reader should compare this to the results of Balogh and Luo~\cite{BL23} and Lidick\'y, Mattes, and Pfender~\cite{LMP24}, which imply that for a certain $H \in \mathcal{T}_3$, $\textnormal{ex}(n, H) = (1+o(1)) g_3(n)$.

While we believe \cref{thm:iter-k-partite} holds for $k=4, 5,$ and $6$ as well, our proof techniques do not readily extend to these uniformities. In \cref{sec:tight components}, we show that \cref{thm:main} (and the corresponding result for $k\geq 7$) is a consequence of \cref{thm:iter-k-partite}, and outline the remainder of the paper.

\section{Proof Outline}\label{sec:tight components}

%\xh{Belongs here: reduction to and statement of Thm 1.5 as a lemma.}
%\xh{Statements and proofs of Lemma 3.4-3.6 and Lemma 4.3 (restated without contradiction assumptions)}
%\ra{Frame as 3 key tools: reduction to Tur\'an-type theorem; tournaments; and degree sequences.}
Our proof of \cref{thm:main} has three main components. The first is the reduction to the Tur\'an-type result \cref{thm:iter-k-partite}. %The second is an analysis of the requirement that all tight components of a $3$-graph be tripartite. \hy{Is this supposed to refer to degree sequence? It is so vague that I do not know which component this is referring to. If so, would the following be better? } \ra{I'm ok with that change} \xh{I don't like either, how about "The second is an analysis of the possible degree sequences in a graph coloring where each color class is tripartite."} 
The second is an analysis of the possible degree sequences in a graph coloring where each color class is tripartite.
The third, and most substantial, is the study of cyclic triangles in a tournament obtained from randomly orienting the $2$-shadow of a $3$-graph. In this section, we explain in detail the reduction to \cref{thm:iter-k-partite} and then provide some preliminary lemmas for the other two proof components.

\subsection{Reduction to the Tur\'an-type problem}% and proof outline}

Questions about whether certain off-diagonal hypergraph Ramsey numbers are polynomial or superpolynomial have received significant attention in recent years \cite{CFG+noY, CFG+Y, CFH+stars}.  
An important family of hypergraphs in this context is the class of \emph{iterated $k$-partite $k$-graphs}. 

\begin{definition} We refer to a complete $k$-partite $k$-graph as the \emph{blowup of an edge}. An \emph{iterated blowup of an edge} is formed from the blowup of an edge by (optionally) replacing each part of the blowup with another iterated blowup of an edge. We say that a $k$-graph is \emph{iterated $k$-partite} if it is contained in the iterated blowup of an edge.\end{definition}

The function $g_k(n)$ is the maximum number of edges in an $n$-vertex iterated $k$-partite $k$-graph. That is, $g_k(n)$ is the number of edges in the balanced iterated blowup of an edge, which can be formed by partitioning the $n$ vertices into $k$ balanced parts and including all transversal edges, and then recursing within each part.\footnote{Erd\H os and Hajnal \cite{EH72} stated that it is easy to see that the maximum in the definition of $g(n)$ is given by letting the $s_j$ be as balanced as possible, i.e. they are all either $\floor{n/k}$ or $\ceil{n/k}$. Mubayi and Razborov~\cite{MR21} proved this formally.}

Observe that all $2$-tight components of an iterated $k$-partite $k$-graph are $k$-partite. Thus, \cref{thm:iter-k-partite} states that for $k=3$ and $k\geq 7$, the balanced iterated blowup of an edge is the densest $k$-graph with this property.

Rather than looking for any $e$ edges among $s$ vertices, one can fix a specific $k$-graph $H$ and define $r_k(H, t)$ to be the minimum $N$ such that any red/blue coloring of $K_N^{(k)}$ contains either a red copy of $H$ or a blue copy of $K_{t}^{(k)}$. To show that $r_k(s, g_k(s); t) = t^{O(1)}$, Erd\H os and Hajnal \cite{EH72} showed that if $H$ is an iterated $k$-partite $k$-graph, then $r_k(H, t)$ is polynomial in $t$. 
In \cite{CFG+Y}, the second and third authors and Conlon, Fox, Gunby, Mubayi, Suk, and Verstra\"ete asked if the converse of this statement also holds: if $H$ is not iterated $k$-partite, then $r_k(H, t)$ exhibits superpolynomial growth in $t$. They proved that this is the case when $H$ is 2-tightly connected, and for $k=3$ they also proved it in the case that $H$ is the union of two tightly connected subgraphs.

To obtain \cref{thm:main} from \cref{thm:iter-k-partite}, we use the main theorem from \cite{CFG+Y}.
\begin{thm}[{\cite[Theorem 2.1]{CFG+Y}}]\label{thm:blackbox}
For every $k\geq 3$ and positive integer $N$, there is a red/blue edge coloring of $K_N^{(k)}$ such that any red 2-tightly connected subgraph is $k$-partite and the largest blue clique has size $O((\log N)^{k/2})$.
\end{thm}
We remark that this theorem was only stated for $k=3$, but the same proof goes through for all higher uniformities as well, as noted in the conclusion of \cite{CFG+Y}. 

%Define $f_k(N, s, e)$ to be the largest $t$ such that any red/blue coloring of $K_N^{(k)}$ contains either $e$ red edges among some $s$ vertices, or a blue clique of size $t$. %Erd\H os and Hajnal showed that $f_k(N, s, g_k(s)) = N^{\Omega(1)}$, while \cref{conj:main} states that $f_k(N, s, g_k(s)+1) = (\log N)^{O(1)}$. 
%\cref{thm:main} (and the corresponding result for $k\geq 7$) is equivalent to the statement $f_k(N, s, g_k(s)+1) = O((\log N)^{k/2})$. We show this now.

\begin{proof}[Proof of \cref{thm:main} and corresponding statement for $k\geq 7$ given \cref{thm:iter-k-partite}]
Let $M$ be the maximum positive integer such that there exists a red/blue coloring $\chi$ of $K_M^{(k)}$ with no red non-$k$-partite $2$-tightly connected subgraph and no blue $t$-clique. \cref{thm:blackbox} implies that $M \geq 2^{\Omega(t^{2/k})}$. Thus, it suffices to show that $r_k(s, g_k(s)+1; t) > M$. This follows quickly from \cref{thm:iter-k-partite}, since the same coloring $\chi$ is guaranteed to have at most $g_k(s)$ red edges among any $s$ vertices, and no blue $t$-clique.%$g_k(s)+1$ red edges among some set $U$ of $s$ vertices, then letting $H\subseteq K_N^{(k)}$ be the red induced subgraph on $U$, \cref{thm:iter-k-partite} implies that $H$ must have a $2$-tight component which is not $k$-partite.% Thus, when $N = r_k(s, g_k(s)+1; t)$, any red/blue coloring of $K_{N}^{(k)}$ has either this red non-$k$-partite $2$-tight component or a blue clique of size $t$, showing that $M \leq N$, as desired. %Let $N = r_k(s, g_k(s)+1; t)$, so a red/blue coloring of $K_N^{(k)}$ must have either $g_k(s)+1$ edges among $s$ vertices or a blue clique of size $t$. If it has a blue clique then by \cref{thm:iter-k-partite}, it has either 
%Take the red/blue coloring of $K_N^{(k)}$ guaranteed by \cref{thm:blackbox}; any red 2-tightly connected subgraph is $k$-partite and the largest blue clique has size $O((\log N)^{k/2})$. We claim that for any $s$, this coloring of $K_N^{(k)}$ does not have $g_k(s)+1$ red edges among any set of $s$ vertices. Suppose the claim is false, and let $H\subseteq K_N^{(k)}$ be the subgraph induced by these $v$ vertices and $g_k(s)+1$ red edges. By \cref{thm:iter-k-partite}, some 2-tight component of $H$ is not $k$-partite, a contradiction to the assumption that any red 2-tightly connected subgraph of $K_N^{(k)}$ is $k$-partite.
%
%We conclude that $f_k(N, s, g_k(s)+1) = O((\log N)^{k/2})$, as desired.
\end{proof}

In a similar vein to Mubayi and Razborov \cite{MR21}, \cref{thm:iter-k-partite} can be further reduced to an inducibility problem regarding edge coloring %($2$-uniform) 
graphs. Define a \emph{$k$-partite coloring} of a graph to be an edge coloring of the graph in which the color classes are all $k$-partite.
Define the \emph{(2-)shadow} of a hypergraph $H$, denoted $\partial H$, to be the (2-uniform) %\xh{I think we only need this parenthetical $0$ or $1$ time} 
graph with vertex set $V(H)$ and edge set consisting of pairs $u, v \in V(H)$ which lie in some edge of $H$ together. Observe that if $(T_i)_{i\geq 1}$ are the 2-tight components of $H$, then their shadows $\partial T_i$ are edge-disjoint. Thus, \cref{thm:iter-k-partite} follows immediately from the following theorem.

\begin{thm}\label{thm:monochromatic-k-cliques}
For $k=3$ and $k\geq 7$, any $k$-partite coloring of the complete graph $K_n$ has at most $g_k(n)$ monochromatic copies of $K_k$.
\end{thm}

The majority of this paper is dedicated to prove \cref{thm:monochromatic-k-cliques}, especially for $k=3$.
In both the cases $k=3$ and $k\geq 7$, this is done by showing the structural result that in any counterexample to \cref{thm:monochromatic-k-cliques} there is a spanning complete $k$-partite color. Inducting within the parts of this subgraph arrives at a contradiction.
For $k\geq 7$, this is achieved using the Loomis-Whitney inequality and convexity. 
For $k=3$, we need to work significantly harder to prove this structural statement, as the convexity is much weaker in lower uniformity. We remark that for $k=3$ and sufficiently large $n$, one could prove this structural result using the stability results of Balogh and Luo~\cite{BL23}; we elaborate more on their work in the conclusion. 
%\hy{There are undefined objects floating around here. Do we define them?} \xh{We should spend some more time defining tight cycles, orientable hypergraphs, etc. I actually prefer to move everything in this paragraph except for the first sentence into the concluding remarks.}

The rest of this section and Sections \ref{sec:degree sequences} and \ref{sec:tournaments} are dedicated to proving \cref{thm:monochromatic-k-cliques} for $k=3$, while in Section \ref{sec:higher uniformity} (which is much less technical and can be read independently) we prove the theorem for $k\geq 7$. %provide a new proof of \cref{conj:main} for $k\geq 7$.

\subsection{Degree sequences}
%We begin with some elementary statements regarding edge colorings whose color classes are tripartite. %While the tournament approach gets stronger for larger $n$, this more direct approach is necessary to prove \cref{thm:monochromatic-k-cliques} for some small values of $n$. \xh{Is this paragraph necessary? I think by this section we can just dive into math and not have so much exposition. I prefer a single sentence "We begin with some elementary statements regarding edge colorings whose color classes are tripartite."}
%\hy{I think some intuition will still be useful. Otherwise this might feel like too many definitions in a row out of nowhere}
In this section, we prove some elementary properties of degree sequences in tripartite colorings (that is, edge colorings where the color classes are tripartite). 
\begin{definition}
Let $V = V(K_n)$. Let $\chi:E(K_n)\to \NN$ be a tripartite coloring.
For any $v\in V$ and color $c$, define the $c$ neighborhood of $v$, denoted $N_c(v)$, to be the set of neighbors $u$ of $v$ with $\chi(uv)=c$. 
Since each color class is tripartite, $N_c(v)$ induces a bipartite subgraph in color $c$. We denote $\cL_v$ to be the \emph{link} of $v$, i.e. the (colorless) bipartite graph which is the union over $c$ of these disjoint bipartite subgraphs.

For $i\geq 1$, let $c_i(v)$ be the $i^\text{th}$ most common color among edges incident to $v$ (with ties broken arbitrarily). We call $c_1(v)$ the \emph{primary color} of $v$ and $c_2(v)$ the 
\emph{secondary color} of $v$. Let $d^i(v) = |N_{c_i(v)}(v)|$. (In particular, $d^1(v)$ is the number of edges incident to $v$ in its primary color.% Different vertices can have different primary colors.
) Call the sequence $(d^1(v), d^2(v), \ldots)$ the \emph{degree sequence} of $v$. This is a finite sequence of positive integers whose length is the number of distinct colors appearing on edges incident to $v$.  Note that $\sum_{i\geq 1} d^i(v) = n-1$ for all $v\in V$. 
\end{definition}

The central observation is \cref{lem:upper-bound-d1} below, which states that in a tripartite coloring of $K_n$, on average the primary color of each vertex appears on at most $\frac 23 n$ of its incident edges. This statement is trivial when every vertex has the same primary color, but we do not have the luxury to make this assumption. %\xh{How is this?}

%\hy{Does is make sense to make the beginning of this paragraph a stand-alone definition?} \ra{Yeah, I support this. Maybe also include the previous paragraph in the definition?} \xh{How about: move these two entire paragraphs into a definition environment and stick it between the first and second sentences of the first paragraph of this section?} 

\begin{lemma}\label{lem:upper-bound-d1}
Given a tripartite coloring $\chi$ of $K_n$, we have
    \[
    \sum_{v\in V} d^1(v) \le \frac{2n^2}{3}.
    \]
\end{lemma}
\begin{proof}
Create an auxiliary digraph $D^1$ on the same vertex set $V$ where we place a directed edge from $u$ to $v$ if $\chi(uv) = c_1(u)$. (We allow antiparallel edges $u\rightarrow v$ and $v\rightarrow u$.) It suffices to show that $e(D^1) \le \frac{2}{3}n^2$.

Partition $V$ into subsets $U_i$ where $u \in U_i$ if $c_1(u) = i$. Observe that internal to each $U_i$ we have $e(D^1[U_i]) \le \frac{2}{3}|U_i|^2$. This is because there are at most $2$ antiparallel edges on each color $i$ edge and no directed edges on each non-color-$i$ edge, and the color class $i$ is tripartite.

Between different $U_i$ and $U_j$, there cannot be antiparallel edges, so $e(D^1(U_i, U_j)) \le |U_i||U_j|$. We obtain in total
    \[
    e(D^1) \le \sum_i \frac{2}{3}|U_i|^2 + \sum_{i<j} |U_i||U_j| \le \frac{2}{3}\left(\sum_i|U_i|\right)^2 = \frac 23 n^2.\qedhere
    \]
    %since $\sum_i|U_i| = n$. 
    %\xh{Add an extra inequality $\le \frac{2}{3} (\sum U_i)^2 \le$ in the calculation for clarity.}
\end{proof}

The proof above lets us read off a slight improvement to the lemma if there are at least two distinct $U_i$'s.

\begin{lemma}\label{lem:improved-bound-d1}
    Given a tripartite coloring $\chi$ of $K_n$, if the primary colors of the vertices are not all the same, then
    \[
    \sum_{v\in V} d^1(v) %\le \frac{2n^2}{3} - \frac{1}{3} \sum_{i<j} |U_i||U_j| 
    \le \frac{2n^2}{3} - \frac{1}{3} (n-1).
    \]
\end{lemma}
\begin{proof}
Following the proof of \cref{lem:upper-bound-d1}, in the last inequality we obtain $$e(D^1) \leq \frac{2n^2}{3} - \frac{1}{3} \sum_{i<j} |U_i||U_j| 
\leq \frac{2n^2}{3} - \frac{1}{3} (n-1)$$ since $|U_i|, |U_j| \geq 1$ for some $i\neq j$ and $\sum_i |U_i| = n$.
\end{proof}

We also need a generalization of \cref{lem:upper-bound-d1} to multiple colors.
Although we only need the case $t=2$, we still state and prove the lemma in generality as the proof stays the same, and the lemma may be of independent interest.%\ra{Actually, we only use the following lemma with $k=2$. Decide later whether to keep it in full generality.} \hy{I think the full generality is pretty cute and worth keeping :) the proof does not simplify that much for k=2 anyways}

\begin{lemma} \label{lem:upper-bound-d1-d2}
    Given a tripartite coloring $\chi$ of $K_n$, for any $t\ge 1$,
    \[
    \sum_{v\in V} d^1(v) + \cdots + d^t(v) \le (1-3^{-t})n^2.
    \]
\end{lemma}
\begin{proof}
%\ra{This proof has been edited somewhat, please reread} \xh{I checked it again and added a few details}
    Draw an auxiliary digraph $D^t$ where $u\rightarrow v$ if $\chi(uv)\in\{c_1(v), \ldots, c_t(v)\}$. Clearly $e(D^t) = \sum_{v\in V} d^1(v) + \cdots + d^t(v)$.
    
    For each color class, label the parts of the tripartition as $0$, $1$, and $2$. Encode each vertex $v$ as a string $s(v)$ in $\{0,1,2,\star\}^*$ with %at most 
    $t$ non-$\star$ coordinates as follows%\hy{I think we need to make sure that this is exactly $t$---see below}\ra{I see how that's convenient. But then maybe we should spell out that if a vertex sees fewer than $t$ colors then just choose some additional arbitrary colors and put $0$ in those coordinates?} \hy{sure I think this works} \ra{Ok I put this in}
    , where the coordinates correspond to colors used by $\chi$.  The non-$\star$ coordinates are the colors $\{c_1(v), \ldots, c_t(v)\}$% (or all colors incident to $v$ if there are fewer than $t$)
    , and the entry at each of these coordinates is the label for the part of the tripartition of that color that contains $v$. If $v$ sees fewer than $t$ colors, we choose arbitrary additional colors and consider $v$ to be in part $0$ of those colors so that $s(v)$ has exactly $t$ non-$\star$ coordinates.

    Let $S$ be a uniformly random string in $\{0, 1,2\}^*$, and say that $s'$ is \textit{compatible} with $S$ if all of its non-$\star$ coordinates agree with $S$. We bound in two ways the number $P$ of ordered pairs $(u,v)$ where $s(u)$ is compatible with $S$ and $s(v)$ is not.

    First, we claim that $\E[P]\ge 3^{-t}e(D^t)$ by linearity of expectation. Indeed, for a given pair $(u,v)$, the probability $s(u)$ is compatible with $S$ is $3^{-t}$, and conditioned on $s(u)$ being compatible with $S$, the probability that $s(v)$ is incompatible with $S$ is $1$ if $u\to v$ and $v\to u$ both lie in $D^t$, and at least $2/3$ if exactly one lie in $D^t$. Adding up over both $(u,v)$ and $(v,u)$ gives the desired lower bound.

    On the other hand, $\E[P] = \E[(n-X)X]$, where $X$ is the number of compatible strings, and $\E[X] = 3^{-t}n$. It follows by convexity that $\E[P] \le 3^{-t}(1-3^{-t})n^2$. Combining these two inequalities gives $e(D^t)\le (1-3^{-t})n^2$, as desired.
\end{proof}

\subsection{From colorings to tournaments}
In order to tackle Erd\H os and Hajnal's conjecture, Conlon, Fox, and Sudakov~\cite{CFS} used the function $T(n)$, defined as the maximum number of cyclic triangles in an $n$-vertex tournament, which satisfies the following explicit formula: %If the outdegrees of the vertices in a tournament are $d_1, \ldots, d_n$, then the tournament has exactly $\binom{n}{3}-\sum_{i=1}^n\binom{d_i}{2}$ cyclic triangles. %To maximize this quantity, we set the $d_i$ to be as nearly equal as possible: if $n$ is odd, we let $d_i = (n-1)/2$ for all $i$, and if $n$ is even, we let half the $d_i$ be $n/2$ and the other half be $(n-2)/2$. A tournament with these outdegrees always exists, 
%This quantity is maximized when $d_i$ are as equal as possible, giving 
$$T(n) = \begin{cases} \frac{n(n^2-1)}{24}&\text{ if $n$ is odd}\\\frac{n(n^2-4)}{24}&\text{ if $n$ is even}.\end{cases}$$ 
Given a tripartite coloring $\chi$ of $K_n$, for each color class, arbitrarily label the three parts as $1, 2, 3$ and orient all edges of that color in the directions $1\to 2\to 3\to 1$. In this way, every monochromatic triangle becomes a cyclic triangle, so in particular the number of monochromatic triangles under $\chi$ is at most $T(n)$. Conlon, Fox, and Sudakov defined the function $$d(n) = T(n) - g_3(n)$$ and proved an explicit recurrence for $d(n)$.

\begin{lemma}[{\cite[Lemma 5.4]{CFS}}]\label{lem:d(n) recursion}
We have $d(1)=d(2)=d(3)=0$, and for any positive integer $x$,
\begin{align*}
d(6x-2) &= 2d(2x-1)+d(2x)\\
d(6x-1) &= d(2x-1)+2d(2x)+x\\
d(6x) &= 3d(2x)\\
d(6x+1) &= 2d(2x)+d(2x+1)+x\\
d(6x+2) &= d(2x)+2d(2x+1)\\
d(6x+3) &=3d(2x+1).
\end{align*}
\end{lemma}

Using this recursion, they showed that $d(n) = O(n\log n)$. They also observed that $d(n) = 0$ for ``nice'' values of $n$, proving \cref{conj:main} in the case that $s$ is one of these values. We characterize the numbers that are ``nice'' in this sense exactly.%They called these integers ``nice," and they listed the nice numbers up to $100$. In fact, we can characterize the nice numbers exactly.

\begin{prop}\label{prop:sum-and-diff-of-exp}
We have $d(n) = 0$ if and only if $n$ is a power of $3$ or the sum or difference of two powers of $3$.
\end{prop}
\begin{proof} 
We deduce directly from \cref{lem:d(n) recursion} that $d(n)=0$ if and only if there is no integer $t\geq 0$ such that $\floor{n/3^t}$ or $\ceil{n/3^t}$ is at least $5$ and congruent to $\pm1\pmod 6$. Write $n$ in balanced ternary form $n = 3^{a_1}\pm 3^{a_2}\pm\cdots$ where $a_1 > a_2 > \cdots \geq 0$. If the representation stops after at most $2$ terms, then by checking each $t\in\{0, 1, \ldots, a_1-1\}$ we see that no such $t$ exists. Otherwise, if $t=a_3$,  then either $\floor{n/3^t}$ or $\ceil{n/3^t}$ is $3^{a_1-a_3}\pm 3^{a2-a3}\pm 1$. Since $3^{a_1-a_3}\pm 3^{a_2-a_3}$ is divisible by $6$, either $\floor{n/3^t}$ or $\ceil{n/3^t}$ is $\pm 1\pmod 6$. 
\end{proof}

When $n$ is not nice, the function $d(n)$ will still be essential for our proof, as we now explain.
%For our purposes, the importance of the function $d(n)$ even for non-nice $n$ begins with the simple observation that since an optimal tripartite coloring has at least $g_3(n)$ monochromatic triangles, in any orientation of the edges, at most $d(n)$ cyclic triangles are allowed to be non-monochromatic. To explore the implications of this, we need several more definitions and notation. %\xh{I do not like all of the discussion in this section up to here. I think we can reduce it to two paragraphs or less, we do not need to say so much exposition by this point. On the other hand it needs to be stated explicitly why $T(n)$ is relevant at all to our problem.}
Throughout this section, let $\chi$ be a tripartite coloring of $K_n$. Recall that $N_c(v)$ is the color-$c$ neighborhood of $v$, and it induces a bipartite subgraph in color $c$.
%\hy{The name ``Tournament setup'' feels imprecise, though I don't know what is a better name.} \xh{How about ``The random orientation trick"?} \ra{Maybe we just don't put a name for these definitions?} \hy{yeah maybe just don't give a name.}
\begin{definition}\label{def:tournament setup}
Let $e^{(3)}(\chi)$ denote the number of monochromatic triangles in $\chi$. %\xh{let's have a shorter notation, maybe $\tau(\chi)$ or $e^{(3)}(\chi)$?}
For any color $c$ and vertex $v$, let $c(v)\in\{0,1,2\}$ be the label of the part $v$ belongs to in the tripartition of color $c$.

Let $N_{c,1}(v), N_{c,2}(v)$ be the two parts of $N_c(v)$ where $c(N_{c,j}(v)) \equiv c(v)+j\pmod 3$ for every $j\in[2]$. (Here we use the shorthand $c(S)$ for $c(u)$ where $u$ is any vertex in $S$; this shorthand is valid whenever all $u\in S$ have the same $c(u)$.)

A triangle $xyz$ is called \emph{$2$-precyclic} %\xh{Can we update this name? Maybe 2-bad and 3-bad are better and consistent with the $B_2$ and $B_3$ notations} 
if, possibly after permutation, $\chi(xy)=\chi(xz)\neq \chi(yz)$ and $\chi(xy)(y)\neq \chi(xy)(z)$.
A triangle $xyz$ is called \emph{$3$-precyclic} if its three edges are all of different colors.
A triangle is called \emph{precyclic} if it is $2$-precyclic or $3$-precylic.
Let $\pc_2$ be the number of $2$-precyclic triangles, and let $\pc_3$ be the number of $3$-precyclic triangles. Note that $\pc_2$ and $\pc_3$ are implicitly functions of $\chi$.%(The subscripts show the number of distinct colors in the triangles.)

%For any set of vertices $S$, any vertex $u\not\in S$ and any color $c$, let $b^{(c)}_S(u)$ be the number of edges from $u$ to $S$ not of color $c$. We drop the superscript if $c$ is clear from context. \ra{More convenient to define this later}

Let $\delta_c(v) = \abs{\abs{N_{c,1}(v)}-\abs{N_{c,2}(v)}}$ be the imbalance in the bipartition of $N_c(v)$.
%Let $M=\max\abs{N_{c,j}(v)}$, where the max is taken over vertices $v$, colors $c$, and $j\in [2]$. \ra{More convenient to define this later}
%The following two parameters will be specified later.
%Let $\Delta$ be an upper bound on $\delta_c(v)$ for all colors $c$ and all vertices $v$.
%Let $B$ be an upper bound for the quantity $2\pc_2+\pc_3.$

%For any nonnegative integers $\Delta$ and $B$, let $\cG_{\Delta,B}$ be the set of graphs $G$ with at least $g_3(n)+1$ monochromatic triangles, each vertex in at least $g_3(n)-g_3(n-1)+1$ monochromatic triangles, $2\pc_2+\pc_3\leq B$ and $\delta_c(v)\leq \Delta$ for any color $c$ and vertex $v$.
\end{definition}

For any extremal coloring $\chi$, $e^{(3)}(\chi)$ lies in the short interval $[g_3(n), T(n)]$ of length $d(n)$. The following key lemma shows that in this situation, the number of precyclic triangles is small. This lemma forms the main guide-rail in our approach for~\cref{thm:monochromatic-k-cliques}.

\begin{lemma}\label{lemma:basic-facts-general} Given a tripartite coloring $\chi$ of $K_n$, we have:
    \begin{enumerate}
        \item $2\pc_2+\pc_3\leq 4(T(n) - e^{(3)}(\chi))$.
        \item $\sum_{v\in V}\sum_c\delta_c(v)^2 + 4\pc_2+2\pc_3\leq \frac{n(n+1)(n-1)}{3} - 8e^{(3)}(\chi).$
    \end{enumerate}    
\end{lemma}
\begin{proof}
    %Direct the edges of $K_n$ so that if $u\to v$, then $\chi(uv)(v)\equiv \chi(uv)(u)\mod 3$. \xh{Use $\pmod$, also is there a plus one missing here?.}
    %Now independently for each $c$ with probability $\frac{1}{2}$, flip the orientation of all edges of color $c$. 
    %Let $T$ be the resulting random tournament where we forget the color. \xh{I think it's more transparent to rewrite the above three sentences as: Given $\chi$, define a random tournament $T$ on the same vertex set as follows. For each color $c$ pick a i.i.d. uniform random $y(c) \in \{1,2\}$ and direct edge $u \to v$ if $\chi(uv) = c$ and $c(u) \equiv c(v) + y(c) \pmod 3$.}
    Given $\chi$, define a random tournament $T$ on the same vertex set as follows. For each color $c$, independently pick a uniformly random $y(c) \in \{1,2\}$, and for each edge $uv$ with $\chi(uv) = c$, direct $u \to v$ if $c(u) \equiv c(v) + y(c) \pmod 3$.
    Let $H$ be the orientable hypergraph generated by $T$; that is, $H$ is a $3$-graph whose edges are the cyclic triangles of $T$. Edges of $H$ come from three possible types of triangles in $\chi$:
    \begin{itemize} 
    \item Each monochromatic triangle is always a cyclic triangle in $T$ and hence gives an edge in $H$.
    \item Each $2$-precyclic triangle becomes a cyclic triangle in $T$ with probability $\frac{1}{2}$.
    \item Each $3$-precyclic triangle becomes a cyclic triangle in $T$ with probability $\frac{1}{4}$.
    \end{itemize}
    As there are $e^{(3)}(\chi)$ monochromatic triangles in the graph, we obtain
    \[e^{(3)}(\chi)+\frac{1}{2}\pc_2+\frac{1}{4}\pc_3 = \EE[e(H)]\leq T(n).\]
    Rearranging, we get the first inequality.

    Let $V$ be the shared vertex set of $K_n$, $T$, and $H$. For each $v\in V$, let $\delta(v)$ be the difference between its in-degree and its out-degree in $T$.
    Then we know that $\delta(v) = \abs{\sum_c\pm \delta_c(v)}$ where each sign is independently and uniformly chosen.
    Note that each triangle in $T$ is either cyclic or contains two oriented vees (an oriented vee is a triple $uvw$ with the orientation of $vu$ being the same as the orientation of $vw$).
    By double-counting, we have   
    \[2\left(\binom{n}{3}-e(H)\right) = \sum_{v\in V}\binom{\frac{n-1+\delta(v)}{2}}{2}+\binom{\frac{n-1-\delta(v)}{2}}{2}.\]
    The right hand side can be rewritten as 
    \[\sum_{v\in V}\frac{1}{4}(n-1)^2+\frac{1}{4}\delta(v)^2-\frac{1}{2}(n-1) = \frac{1}{4}\left(n(n-1)(n-3)-\sum_{v\in V}\delta(v)^2\right)\]
    and so
    \[\sum_{v\in V}\delta(v)^2\leq \frac{4}{3}n(n-1)(n-2)-n(n-1)(n-3)-8e(H) = \frac{1}{3}n(n+1)(n-1)-8e(H).\]
    This holds surely, so it holds in expectation. Observe that
    \[\EE[\delta(v)^2]= \EE\left[\left(\sum_{c}\pm\delta_c(v)\right)^2\right] = \EE\left[\left(\sum_{c_1}\sum_{c_2} \pm\delta_{c_1}(v)\cdot\pm\delta_{c_2}(v)\right)\right] = \sum_{c}\delta_c(v)^2,\] where we used linearity of expectation in the last step. % to expand out the terms and then noticed that for $c_1\neq c_2$ the expectation of the term is $0$ since it is equally likely to be $+\delta_{c_1}(v)\delta_{c_2}(v)$ as $-\delta_{c_1}(v)\delta_{c_2}(v)$. \xh{I think we can safely skip this explanation, just say by linearity of expectation.} 
    Hence, we get 
    \[\sum_{v\in V}\sum_{c}\delta_c(v)^2\leq \frac{n(n+1)(n-1)}{3}-8\EE[e(H)] = \frac{n(n+1)(n-1)}{3}-8e^{(3)}(\chi)-4\pc_2-2\pc_3.\]
    We then rearrange to get the second inequality.
\end{proof}

%\cref{lemma:basic-facts-general} helps to upper bound the numbers of precyclic triangles and also the $\delta_c(v)$'s.
%Intuitively, the upper bound on the precyclic triangles helps us build a monochromatic almost-complete tripartite subgraph, while the upper bound on $\delta_c(v)$'s shows that the constructed tripartite subgraph is approximately balanced.
%This intuition will be carried out in \cref{sec:tournaments}.
For the optimal tripartite coloring $\chi$, $e^{(3)}(\chi)$ will be very close to $T(n)$, so \cref{lemma:basic-facts-general} gives us strong quantitative control over $\pc_2$, $\pc_3$, and the average value of $\delta_c(v)$.

%\xh{I prefer to write instead: For the extremal coloring $\chi$, $\tau(\chi)$ will be very close to $T(n)$, so \cref{lemma:basic-facts-general} gives us strong quantitative control over $\pc_2$, $\pc_3$, and $\delta_c(v)$.}

\section{Induction setup and small \texorpdfstring{$n$}{n}}\label{sec:degree sequences}
%We now focus on the $3$-uniform setting. %Recall from the introduction that to prove our main theorem, it suffices to show that in any coloring of the edges of the complete graph $K_n$ such that each color class is tripartite, there are at most $g_3(n)$ monochromatic triangles. 
%The proof of \cref{thm:monochromatic-k-cliques} for $k=3$ proceeds by induction on $n$. For the base cases $n = 0, 1, 2$, the statement is obvious as there no triangles and $g_3(n)=0$. Now we want to show that the statement holds for some $n\geq 3$, and we assume it to be true for all smaller $n$. We suppose that \cref{thm:monochromatic-k-cliques} is false, and our goal is to find a contradiction to that assumption. \xh{Maybe it is sufficient to say: We begin the proof of \cref{thm:monochromatic-k-cliques} by making the following assumption for contradiction.}

In this section, we set up the inductive proof of \cref{thm:monochromatic-k-cliques}. We begin by making the following assumption for contradiction.
\begin{assumption}\label{asm:contradiction}
\cref{thm:monochromatic-k-cliques} is false for $k=3$. Let $n\geq3$ denote the minimum $n$ for which it fails, and let $\chi$ be a tripartite coloring of $K_n$ with at least $g_3(n)+1$ monochromatic triangles.
\end{assumption}

As the tools from the previous section vary greatly in strength depending on the size and $3$-adic properties of $n$, we were unable to obtain a contradiction to \cref{asm:contradiction} by a single unified method. 
Instead, we single out the cases $n=13,14,16,17$ and deal with them in this section.
For the remaining values of $n\geq 3$ we use another strategy, which we carry out in \cref{sec:tournaments}. %-\ref{sec:main lemmas}.
The exact details for that approach vary for three different regimes of $n$, as we explain in Appendix \ref{sec:verifying assumptions}.

As an immediate consequence of \cref{asm:contradiction}, we may assume every vertex lies in many monochromatic triangles.

%\ra{todo: make all lemmas self-contained. e.g. for the following lemma: Suppose $n$ is minimal such that \cref{thm:monochromatic-k-cliques} is false for $k=3$, and let $\chi$ be a coloring of the edges of $K_n$ with at least $g_3(n)+1$ triangles. Then, ...}
%\ra{Alternatively: be very clear in the paragraphs above about what $n$ and $\chi$ are and state that all lemmas in this section refer to those particular $n$ and $\chi$.}

\begin{lemma}\label{lem:each-vertex-triangles}
Under the coloring $\chi$ from \cref{asm:contradiction}, each vertex $v\in V$ lies in at least $g_3(n) - g_3(n-1) + 1$ monochromatic triangles, and furthermore $\sum_{i\geq 1} \bip(d^i(v)) \geq g_3(n) - g_3(n-1) + 1$, where $\bip(x)\coloneqq \floor{\frac{x}{2}}\ceil{\frac{x}{2}}$.
\end{lemma}
\begin{proof} 
Suppose $v$ is in at most $g_3(n) - g_3(n-1)$ triangles. Remove $v$ and restrict $\chi$ to the remaining $n-1$ vertices. The color classes are still tripartite, and there are at least $g_3(n-1)+1$ monochromatic triangles, violating the minimality of $n$ in \cref{asm:contradiction}. This gives the first part of the lemma.

Next, observe that since each color class $c$ is tripartite, each color neighborhood $N_c(v)$ induces a bipartite graph in color $c$ whose edges are in bijection with triangles containing $v$. Thus, the number of color-$c$ triangles containing $v$ is at most $\bip(|N_c(v)|)$. Summing over $c$ gives the second part of the lemma.
\end{proof}

\cref{lem:each-vertex-triangles} already significantly restricts the possible degree sequences in $\chi$. We make an additional minor simplification.
%To simplify the casework, we can impose one more restriction on allowable degree sequences as follows.

\begin{lemma}\label{lem:no-two-ones}
If \cref{asm:contradiction} holds, then there exists such a $\chi$ such that every vertex $v$ has at most one $1$ in its degree sequence.
\end{lemma}
\begin{proof}
    %Assign an arbitrary order to the collection of all colors used by $\chi$. Consider the following procedure.
    Starting from any $\chi$, consider the following procedure. Whenever there is $v$ that has exactly one incident edge of each of two colors $c_1 < c_2$, recolor the edge of color $c_2$ by $c_1$. Since $v$ has degree two in $c_1$ after this operation, the coloring remains tripartite.
    Since the sum of all colors in $\chi$ decreases after each step, this process eventually terminates. The resulting tripartite coloring $\chi'$ has no vertex with at least two $1$'s in its degree sequence. Finally, the number of monochromatic triangles can only increase throughout this process, so the end result $\chi'$ satisfies \cref{asm:contradiction} and the additional property desired.
    %Furthermore, since vertices of deg$\chi'$ is still tripartite. Indeed, suppose that at some step an edge incident to $v$ is recolored to color $c$. Before, $v$ had only one edge incident to $c$, so now it has only two; simply let $v$ be in a tripartition class of $c$ which is not the class of either of its two color-$c$ neighbors.
    %
    %We conclude that if some edge coloring $\chi$ of $K_n$ exists with at least $g_3(n)+1$ monochromatic triangles, then some coloring exists with at least $g_3(n)+1$ monochromatic triangles and with the property that the sequence $(d^1(v), d^2(v), \ldots)$ has at most one $1$ for every $v\in V$.
\end{proof}
We now obtain a contradiction to \cref{asm:contradiction} by hand for small values of $n$ not covered by our later arguments.

\begin{prop}\label{prop: small n} The minimum $n$ of \cref{asm:contradiction} is not $13, 14, 16,$ or $17$.
\end{prop}
\begin{proof}
We use the following table of values for $g_3(n)$, computed directly from its definition.
\[\begin{array}{c|c}
    n & g_3(n) \\
    \hline
    12 & 70 \\
    13 & 88 \\
    14 & 110 \\
    15 & 137 \\
    16 & 166 \\
    17 & 200 \\
\end{array}\]
In each case below, take $\chi$ to be the coloring guaranteed by \cref{asm:contradiction} and \cref{lem:no-two-ones}.% we can assume each vertex has at most one $1$ in its degree sequence. 
\paragraph{Case $n=13$.} By \cref{lem:each-vertex-triangles}, each vertex is contained in at least $g_3(13) - g_3(12) + 1 = 19$ monochromatic triangles under $\chi$. The only possible degree sequences satisfying \cref{lem:each-vertex-triangles} and \cref{lem:no-two-ones} are $(8,4), (9, 3),$ $(9,2,1),$ $(10,2),$ $(11,1),$ and $(12)$. By \cref{lem:upper-bound-d1}, we have $\sum_{v} d^1(v) \leq \floor{\frac{2}{3} (13)^2} = 112$, so at least $5$ vertices have degree sequence $(8, 4)$.

Suppose first that all 13 vertices have the same primary color $1$, say. We claim the tripartition for color $1$ has part sizes $(5,4,4)$. Indeed, any vertex in a part of size at least $6$ satisfies $d^1\leq 7$, contradiction. The only other tripartition of $13$ is $(5,5,3)$, but in this case, since $d^1\geq 8$ for all vertices, the tripartition must be complete in color $1$. But then all vertices have $d^1 \in \{8, 10\}$, so as seen by the allowable degree sequences, $d^1(v) + d^2(v) = 12$ for every $v$. This violates \cref{lem:upper-bound-d1-d2}, which implies $\sum_v d^1(v) + d^2(v) \le \floor{\frac{8}{9} (13^2)} = 150$.

With the claim proved, it is evident that the part of size $5$ contains only vertices with degree sequence $(8,4)$. But this means those $5$ vertices share a common secondary color % \xh{Did we define this yet?} \ra{No; we can just say non-primary instead of secondary since there are only 2 colors} \xh{I think we should define it, secondary is used later in n=16 when there are 3 colors} 
as well, leading to a monochromatic $K_5$ in that color, which is not tripartite.

Thus, we may assume the vertices do not all have the same primary color. \cref{lem:improved-bound-d1} now implies $\sum_{v} d^1(v) \leq 108$. Thus at least $9$ vertices have degree sequence $(8,4)$ and hence $d^1(v)+d^2(v) = 12$. As seen by the allowable degree sequences, the remaining $4$ vertices have $d^1(v)+d^2(v) \geq 11$, so in total $\sum_v d^1(v)+d^2(v) \geq 152$. But $\floor{\frac{8}{9}(13^2)} = 150$, so this contradicts \cref{lem:upper-bound-d1-d2}.

\paragraph{Case $n=14$.} By \cref{lem:each-vertex-triangles}, each vertex is contained in at least $g_3(14)-g_3(13)+1 = 23$ monochromatic triangles under $\chi$. The degree sequences satisfying \cref{lem:each-vertex-triangles} and \cref{lem:no-two-ones} are $(9,4), (10, 3),$ $(10,2,1),$ $(11,2),$ $(12,1)$, and $(13)$. \cref{lem:upper-bound-d1} implies $\sum_{v} d^1(v) \leq \floor{\frac{2}{3} (14)^2} = 130$, so at least $10$ vertices have degree sequence $(9, 4)$ and hence $d^1(v)+d^2(v) = 13$. As seen by the allowable degree sequences, the remaining $4$ vertices have $d^1(v)+d^2(v) \geq 12$, so in total $\sum_v d^1(v)+d^2(v) \geq 178$. But $\floor{\frac{8}{9}(14^2)} = 174$, so this contradicts \cref{lem:upper-bound-d1-d2}.

\paragraph{Case $n=16$.}
By \cref{lem:each-vertex-triangles}, each vertex is contained in at least $g_3(16) - g_3(15) + 1 = 30$ monochromatic triangles under $\chi$. The only possible degree sequences satisfying \cref{lem:each-vertex-triangles} and \cref{lem:no-two-ones} are $(10, 5), (11, 4),$ $(11, 3, 1),$ $(11, 2, 2),$ $(12, 3),$ $(12, 2, 1),$ $(13, 2),$ $(14, 1),$ and $(15)$. By \cref{lem:upper-bound-d1}, we have $\sum_{v} d^1(v) \leq \floor{\frac{2}{3} (16)^2} = 170$, so at least $6$ vertices have degree sequence $(10, 5)$.

Suppose first that all 16 vertices have the same primary color $1$, say. We claim the tripartition parts for color $1$ have sizes $(6,5,5)$. Indeed, any vertex in a part of size at least 7 satisfies $d^1\leq 9$, contradiction. The only other tripartition of $16$ is $(6,6,4)$, but in this case, since $d^1\geq 10$ for all vertices, the tripartition must be complete in the color $1$. But then 12 vertices have $d^1(v)=10$ and hence $d^2(v)=5$, while the remaining 4 vertices have $d^1(v) = 12$ and hence $d^2(v) \geq 2$ (as seen by the allowable degree sequences). Thus, $\sum_{v} d^1(v) + d^2(v) \geq 12\cdot15+4\cdot 14 = 236$. This violates \cref{lem:upper-bound-d1-d2}, which implies that this sum is at most $\floor{\frac{8}{9} (16^2)} = 227$.

With the claim proved, it is evident that the part of size $6$ contains only vertices with degree sequence $(10,5)$. But this means those $6$ vertices share a common secondary color as well, leading to a monochromatic $K_6$ in that color, which is not tripartite.

Thus, we may assume the vertices do not all have the same primary color. \cref{lem:improved-bound-d1} now implies $\sum_{v} d^1(v) \leq 165$. Thus at least $11$ vertices have degree sequence $(10,5)$ and hence $d^1(v)+d^2(v) = 15$. As seen by the allowable degree sequences, the remaining $5$ vertices have $d^1(v)+d^2(v) \geq 13$, so in total $\sum_v d^1(v)+d^2(v) \geq 230$. But $\floor{\frac{8}{9}(16^2)} = 227$, so this contradicts \cref{lem:upper-bound-d1-d2}.

\paragraph{Case $n=17$.}
By \cref{lem:each-vertex-triangles}, each vertex is contained in at least $g_3(17)-g_3(16)+1 = 35$ monochromatic triangles under $\chi$. The degree sequences satisfying \cref{lem:each-vertex-triangles} and \cref{lem:no-two-ones} are $(11,5), (12,4),$ $(12,3,1),$ $(12,2,2),$ $(13, 3),$ $(13,2,1),$  $(14, 2), (15,1),$ and $(16)$. \cref{lem:upper-bound-d1} implies $\sum_{v} d^1(v) \leq \floor{\frac{2}{3} (17)^2} = 192$, so at least $12$ vertices have degree sequence $(11, 5)$ and hence $d^1(v)+d^2(v) = 16$. As seen by the allowable degree sequences, the remaining $5$ vertices have $d^1(v)+d^2(v) \geq 14$, so in total $\sum_v d^1(v)+d^2(v) \geq 262$. But $\floor{\frac{8}{9}(17^2)} = 256$, so this contradicts \cref{lem:upper-bound-d1-d2}.
\end{proof}

We remark that while the proofs above also go through for other values of $n \leq 17$, they do not extend to large values of $n$. Indeed, one can easily check that $g_3(n) - g_3(n-1) +1 = n^2/8 +o(n^2)$, and as a result, many degree sequences satisfying \cref{lem:each-vertex-triangles} have $d^1(v) = n/2 + o(n)$. On the other hand, the proofs above relied on the allowable degree sequences forcing $d^1(v)$ to be close to $2n/3$. 
Our proof strategy for larger $n$ (which also works for small values of $n$ except $13, 14, 16,$ and $17$) uses the tournament approach, on which we elaborate next.

\section{Induction for large \texorpdfstring{$n$}{n}}\label{sec:tournaments}
%\xh{Rename this section, maybe "Induction for large $n$"?} 
%\xh{I combined what was previously sections 4 and 5, need to go back and update any outline/exposition/crefs that got broken.}

In this section we perform the bulk of the proof of \cref{thm:monochromatic-k-cliques} for $k=3$, reducing it to a handful of technical assumptions, which are verified in Appendix \ref{sec:verifying assumptions}. We mark lemmas that introduce such assumptions with $\mathbf{(\bigstar)}$ and state the assumption clearly at the beginning of the lemma. All of the arguments in this section are elementary; the main difficulty is meticulously tracking every constant, as the bounds are extremely close to tight for certain $n\le 100$. Some lemma statements are written in a more complicated form than absolutely necessary in order to make the computer program as transparent as possible.

We now restate our tournament lemma from \cref{sec:tight components} in the context of the coloring from \cref{asm:contradiction}. Recall the notation from \cref{def:tournament setup}. %\xh{Make some remark here or around the definition of $\pc_2$ and $\pc_3$ that they are implicitly functions of $\chi$.} 
For notational convenience, denote $$\tilde d(n) := \frac{n(n+1)(n-1)}{3} - 8(g_3(n)+1).$$ %\hy{Not sure if I like this notation. We can think about it.} \xh{We can make a remark that $\tilde d(n) = 8d(n)$ for odd $n$, and is slightly larger for even $n$, would that help?} \hy{The factor of $8$ might appear as weird but I don't have a better notation in mind either} \ra{It would be $\tilde d(n) = 8(d(n)+1)$ for odd $n$. I'm not sure it's worth making any remark to this effect.} \xh{Ok, let's keep it as is, we can replace the notation later if we come up with an alternative.}

\begin{lemma}\label{lemma:basic-facts} Let $n$ and $\chi$ be as in \cref{asm:contradiction}. We have:
    \begin{enumerate}[(a)]
        \item $2\pc_2+\pc_3\leq 4(d(n)-1)$.
        \item $\sum_{v\in V}\sum_c\delta_c(v)^2 + 4\pc_2+2\pc_3\leq \tilde d(n).$
    \end{enumerate}    
\end{lemma}
\begin{proof}
This is simply \cref{lemma:basic-facts-general} applied to the coloring $\chi$ from \cref{asm:contradiction}, noting that $e^{(3)}(\chi) \geq g_3(n)+1$.
\end{proof}

\cref{lemma:basic-facts} gives an upper bound on $2\pc_2+\pc_3$ and on $\max_{c,v}\delta_c(v)$.
With this in mind, we define the following. Recall that $n$ is fixed by \cref{asm:contradiction}.

\begin{definition}
For any nonnegative integers $\Delta$ and $P$, let $\cG_{\Delta,P}$ be the set of tripartite colorings $\chi$ of $K_n$ such that there are at least $g_3(n)+1$ monochromatic triangles, each vertex lies in at least $g_3(n)-g_3(n-1)+1$ monochromatic triangles, $2\pc_2+\pc_3\leq P$, and we have $\delta_c(v)\leq \Delta$ for any color $c$ and vertex $v$.
\end{definition}

%\xh{Here is my proposed rewrite of the next two corollaries. The idea is to hide the exact iteration strategy until later (maybe it can be spelled out in section 6, maybe just hidden in the code?)}
\begin{cor}\label{cor:calGempty}
If \cref{asm:contradiction} holds, then there exists nonnegative integers $\Delta, P$ with 
\begin{equation}\label{eq:corcond1} P = \min(P(\Delta), 4(d(n)-1))
\end{equation}
where
\begin{equation}\label{eq:corcond2}P(\Delta) \coloneqq\left\lfloor \frac{\tilde d(n)-\Delta^2}{2}\right\rfloor, \end{equation}
for which $\cG_{\Delta, P}$ is nonempty.
\end{cor}

\begin{proof}
Suppose \cref{asm:contradiction} holds, and let $\chi$ be the tripartite coloring it guarantees with at least $g_3(n)+1$ monochromatic triangles where each vertex is in at least $g_3(n)-g_3(n-1)+1$ monochromatic triangles.
Let $\Delta = \max_{c,v}\delta_c(v)$, and $P=P(\Delta)$.
By \cref{lemma:basic-facts}, we know that the pair $\Delta, P$ satisfies condition \eqref{eq:corcond1}, and that $2\pc_2 + \pc_3 \le P$.
Hence, $\chi\in \cG_{\Delta, P}$, so $\cG_{\Delta, P}$ is nonempty.
\end{proof}

Thus, to complete the contradiction for a given value of $n$, it suffices to prove that for every choice of $\Delta, P$ satisfying the conditions of \cref{cor:calGempty}, the set $\cG_{\Delta,P}$ is empty. % For $n\leq 699$, we check this with the help of a computer program. For $n\geq 700$, the same holds, but we show this by hand. See \cref{sec:verifying assumptions} for details.
We next describe a concrete way of showing that some $\cG_{\Delta, P}$ is empty%, which we use for both $n \leq 699$ and $n\geq 700$
.

\begin{definition}\label{def:v*c*}
Given $\chi\in \cG_{\Delta, P}$, let vertex $v^*\in V$ and color $c^*$ be chosen to maximize $\abs{N_{c^*}(v^*)}$ under $\chi$. Take the coloring $\chi \in \cG_{\Delta, P}$ to be the one which minimizes the total number of edges of $K_n$ not colored $c^*$ under this definition.

Let $U$ be the set of vertices that are in the same part as in $v^*$ with respect to the color $c^*$, and let $X^*, Y^*$ be the other two parts.
Let $X = N_{c^*}(v^*)\cap X^*$ and $\overline{X} = X^*\backslash X$, and define $Y$ and $\overline{Y}$ similarly.
\end{definition}

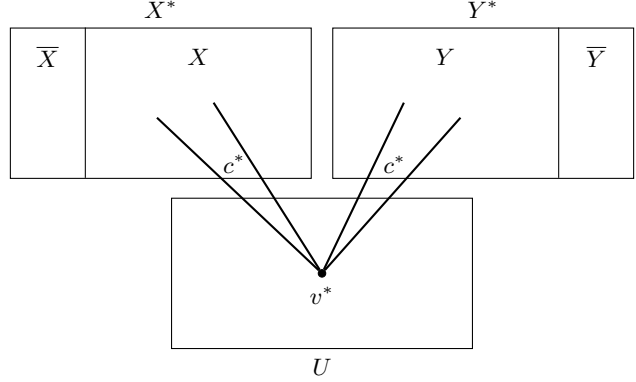
\begin{figure}[htb!]\label{fig:parts}
\centering
\begin{tikzpicture}[scale=1, every node/.style={font=\small}]

% Box dimensions
\def\boxwidth{4cm}
\def\boxheight{2cm}
\def\largefrac{0.75}
\def\smallfrac{0.25}

% X^* box in top-left
\node[draw, minimum width=\boxwidth, minimum height=\boxheight, anchor=north west, label=above:$X^*$] (Xstar) at (0,0) {};

% Partition X^*: X on right, \overline{X} on left
\coordinate (Xsplit) at ($(Xstar.north west)!\smallfrac!(Xstar.north east)$);
\draw (Xsplit) -- ($(Xsplit)+(0,-\boxheight)$);
\node at ($($(Xstar.north west)!\smallfrac/2!(Xstar.north east)$)+(0,-11pt)$) {$\overline{X}$};
\node at ($($(Xstar.north west)!\smallfrac+0.5*\largefrac!(Xstar.north east)$)+(0,-11pt)$) {$X$};

% Y^* box in top-right
\node[draw, minimum width=\boxwidth, minimum height=\boxheight, anchor=north east, label=above:$Y^*$] (Ystar) at (8.3,0) {};

% Partition Y^*: Y on left, \overline{Y} on right
\coordinate (Ysplit) at ($(Ystar.north west)!\largefrac!(Ystar.north east)$);
\draw (Ysplit) -- ($(Ysplit)+(0,-\boxheight)$);
\node at ($($(Ystar.north west)!\largefrac/2!(Ystar.north east)$)+(0,-11pt)$) {$Y$};
\node at ($($(Ystar.north west)!\largefrac+0.5*\smallfrac!(Ystar.north east)$)+(0,-11pt)$) {$\overline{Y}$};

% U box in bottom-center
\node[draw, minimum width=4cm, minimum height=2cm, anchor=north, label=below:$U$] (U) at ($(Xstar.south)!-0.5cm!270:(Xstar.south)!0.5!(Ystar.south)$) {};

% Vertex v^*
\node[circle, draw, fill=black, inner sep=1pt, label=below:$v^*$] (vstar) at (U.center) {};

% Points in X (right 75% of X^*)
\coordinate (Xpt1) at ($($(Xsplit)!\largefrac/3!(Xstar.north east)$)+(0.2cm,-1.2cm)$);
\coordinate (Xpt2) at ($($(Xsplit)!\largefrac*2/3!(Xstar.north east)$)+(0.2cm,-1cm)$);

% Points in Y (left 75% of Y^*)
\coordinate (Ypt1) at ($($(Ystar.north west)!\largefrac/3!(Ysplit)$)+(0.2cm,-1cm)$);
\coordinate (Ypt2) at ($($(Ystar.north west)!\largefrac*2/3!(Ysplit)$)+(0.2cm,-1.2cm)$);

% Undirected edges from v^* to X
\draw[thick] (vstar) -- (Xpt1); 
\node at ($ (vstar)!0.5!(Xpt1) + (-2pt, 12pt) $) {$c^*$};
\draw[thick] (vstar) -- (Xpt2);

% Undirected edges from v^* to Y
\draw[thick] (vstar) -- (Ypt1); 
\node at ($ (vstar)!0.5!(Ypt1) + (12pt, 9pt) $) {$c^*$};
\draw[thick] (vstar) -- (Ypt2);

\end{tikzpicture}
\caption{Schematic of the parts described by \cref{def:v*c*}.}
\end{figure}

\begin{lemma}\label{lemma: recoloring}
    Let $\chi, c^*, X^*, Y^*,$ and $U$ be as in \cref{def:v*c*}, and let $m=\min(\abs{X^*},\abs{Y^*},\abs{U})$. There must be more than $2m$ edges across $X^*, Y^*, U$ not colored $c^*$.
\end{lemma}
\begin{proof}
%\xh{I reworded some of the proof below.}
Let $\chi'$ be the coloring obtained from $\chi$ by recoloring all edges across $X^*, Y^*, U$ with color $c^*$. We claim that if at most $2m$ edges are recolored, then $\chi'$ has $e^{(3)}(\chi') \geq e^{(3)}(\chi)$
To see that this claim implies the lemma, note that $\chi'$ has $|X^*||Y^*||U|$ monochromatic triangles of color $c^*$, and by the minimality of $n$ in \cref{asm:contradiction}, there are at most $g_3(|X^*|)+g_3(|Y^*|)+g_3(|U|)$ monochromatic triangles contained within the three parts. Therefore, by definition of $g_3(n)$, we have $e^{(3)}(\chi) \leq e^{(3)}(\chi') = |X^*||Y^*||U| + g_3(|X^*|)+g_3(|Y^*|)+g_3(|U|) \leq g_3(n)$. But $\chi\in \cG_{\Delta, P}$ is required to have at least $g_3(n)+1$ monochromatic triangles, a contradiction.

We now prove the claim.
    Let $E$ 
    be the set of edges across $X^*,Y^*,U$ not of color $c^*$.
    Recall the notation $1_{S}$ for the indicator of the statement $S$, which is $1$ if $S$ holds and $0$ otherwise. For each $ab\in E$, we define $w_{ab}(x)$ for all $x\in V\backslash \{a,b\}$ as follows: %\xh{Have we define indicator functions yet?}\hy{I don't think so, should we do it here?}
    \begin{itemize}
        \item if $c^*(x)\not\in \{c^*(a),c^*(b)\}$, then $w_{ab}(x) = 1-\frac{1}{2}\left(1_{\chi(ax)\neq c^*}+1_{\chi(bx)\neq c^*}\right)$;
        \item if $c^*(x)=c^*(a)$, then $w_{ab}(x) = -\frac{1}{2}\cdot 1_{\chi(bx)\neq c^*}$;
        \item if $c^*(x)=c^*(b)$, then $w_{ab}(x) = -\frac{1}{2}\cdot 1_{\chi(ax)\neq c^*}.$
    \end{itemize}

    We also define $w_{ab}(x)=0$ if $ab\not\in E$.

    %Now consider recoloring all edges in $E$ with color $c^*$.
    For any three distinct vertices $xuv$, let $W(xuv)$ %\xh{Can we drop the subscript on $\Delta$?}\hy{Actually maybe $\Delta$ is not a good notation... not sure what a better notation would be}
    equal $1$ if $xuv$ is not monochromatic in $\chi$ and monochromatic in $\chi'$, $-1$ if it is monochromatic in $\chi$ but not monochromatic in $\chi'$, and $0$ otherwise.
    We show that
    \[W(xuv)\geq w_{xu}(v)+w_{uv}(x)+w_{vx}(u).\]
    Note that if $c^*(x)=c^*(u)=c^*(v)$ then both sides are 0.
    If $c^*(x),c^*(u),c^*(v)$ are all distinct, let $e_{xuv}$ %\xh{I don't like $N$ for such a small number} 
    be the number of edges $uv,vx,xu$ in $E$.
    Then $W(xuv) \geq 1_{e_{xuv}\in \{1, 2\}}$ and the right hand side equals exactly $1_{e_{xuv} \in \{1, 2\}}$. %$N(1-\frac{1}{2}(N-1))$. \xh{I prefer to write simply that the right hand side equals exactly $1_{N \in \{1,2\}}$.}
    Substituting $e_{xuv}=0,1,2,3$, this verifies the inequality is true in this case. 

    Finally, if $c^*(x)\neq c^*(u)=c^*(v)$, let $e_{xuv}'$ be the number of edges $xu,xv$ in $E$.
    Then $W(xuv)\geq -1_{e_{xuv}'=2}$ and $w_{uv}(x)=0$.
    We also know $w_{xv}(u)+w_{xu}(v)$ is $0$ if $e_{xuv}'\in\{0,1\}$, and it is $-\frac{1}{2}-\frac{1}{2}=-1$ if $e_{xuv}'=2$.
    Therefore the inequality holds in this case as well.

    Adding everything together, we have
    \[e^{(3)}(\chi') - e^{(3)}(\chi) = \sum_{\{x,u,v\}\in \binom{V}{3}}W(xuv) \ge \sum_{\{x,u,v\}\in \binom{V}{3}}\left(w_{xu}(v)+w_{uv}(x)+w_{vx}(u)\right)=\sum_{e\in E}\sum_{x\not\in e}w_e(x).\]
    Now by the definitions of the weights, we know that
    \[\sum_{x\not\in e}w_e(x)\geq m-\frac{1}{2}\abs{E}\geq 0\]
    for any $e\in E$.
    This concludes the proof.
\end{proof}

Thus, to contradict \cref{asm:contradiction}, we just need to show that (using the notation of \cref{def:v*c*}) $\chi$ has at most $2\min(|X^*|, |Y^*|, |U|)$ edges not of color $c^*$ across $X^*$, $Y^*$, and $U$. In words, we show the single color $c^*$ forms a nearly complete tripartite subgraph by repeatedly exploiting the fact that edges of any other color contribute precyclic triangles.

We first prove some bounds on $\abs{X}$ and $\abs{Y}$. Recall that $\bip(x) = \floor{\frac{x}{2}}\ceil{\frac{x}{2}}$ is the maximum number of edges in a bipartite graph on $x$ vertices.
For our convenience later, we list the desired properties that $\abs{X},\abs{Y}$ should satisfy in the following definition.
% \xh{I prefer to put the definition of admissible pair in a def environment separately beforehand.}

\begin{definition}\label{def:admissible-pair}
    A pair of positive integers $(s,t)$ is an \emph{admissible pair} if they satisfy the following four conditions.
    \begin{enumerate}[(a)]
        \item  %If $\lceil\frac{12g_3(n)}{n(n-1)}\rceil\geq \lceil n/2\rceil$, then
        $s+t \geq \lceil n/2\rceil$.
        \item $\abs{s-t}\leq\Delta$.
        \item Let $S$ be the smallest possible sum of positive integers $d_1,d_2$ with $d_2\leq d_1\leq \abs{s}+\abs{t}$ and 
        %\[\bip(d_1) + \left\lfloor\frac{n-1-d_1}{d_2}\right\rfloor\bip(d_2)+\bip(n-1-d_1\pmod {d_2}) \geq g_3(n)-g_3(n-1)+1.\]
        \[\bip(d_1) + q\bip(d_2)+\bip(r) \geq g_3(n)-g_3(n-1)+1,\] where $q$ and $r$ are the unique nonnegative integers such that $n-1-d_1 = qd_2+r$ and $r < d_2$. %\ra{Is this ok? Alternative wordings: ``where $q$ and $r$ are the quotient and remainder when $n-1-d_1$ is divided by $d_2$"; or ``where $q = \floor{\frac{n-1-d_1}{d_2}}$ and $r = n-1-d_1-qd_2$".} \xh{The current one is best.}
        Then, $S\leq \frac{8}{9}n$. %\xh{Is it better to write this condition as (c) There exists positive integers $d_2 \le d_1 \ldots$ such that ... and $d_1 + d_2 \le \frac 89 n$? Can move the definition of the letter $S$ into the proof.} \ra{Someone reading the code would be happier if it's written as it currently is.}%\hy{The display equation is hard to read. Should we just say $n-1-d_1=qd_2+r$ or something like that?} \xh{I prefer that as well.}
        \item $|s||t|+\bip(n-1-|s|-|t|) \geq g_3(n)-g_3(n-1)+1.$
        %\xh{I think throughout 3. and 4. the $\bip$ inputs is off by a factor of $2$, there should not be an $/2$ in the denominators.}
    \end{enumerate}

\end{definition}

\begin{lemma}[$\bigstar$]\label{lemma:possible-X-Y} Assume that $\lceil\frac{12g_3(n)}{n(n-1)}\rceil\geq \lceil n/2\rceil$. Then, $(|X|, |Y|)$ is an admissible pair.
\end{lemma}
\begin{proof}
    By definition, $v^*$ and $c^*$ are chosen to maximize $|X|+|Y|$. So, to show that $\abs{X}+\abs{Y}\geq \lceil n/2\rceil$, it suffices to find a vertex $v'\in V$ and a color $c'$ with $\abs{N_{c'}(v')}\geq n/2$.
    Take $v'$ to be the vertex with the most edges in its link $\cL_{v'}$.
    Then, using the fact that $\cL_{v'}$ is bipartite, %\xh{All the $v$'s in the next two lines should be $v'$?}
    \[\frac{3g_3(n)}{n}\leq e(\cL_{v'})\leq \frac{1}{4}\sum_{c'}\abs{N_{c'}(v')}^2\leq \frac{n-1}{4}\max_{c'}N_{c'}(v').\]
    Therefore
    \[\abs{X}+\abs{Y}\geq \max_{c'}N_{c'}(v')\geq \frac{12g_3(n)}{n(n-1)}\]
    and the inequality holds with ceiling on the right as well.
    By the assumption, it is at least $\lceil n/2\rceil$, as desired.

    Part (b) follows from the fact that $\abs{\abs{X}-\abs{Y}} = \delta_{c^*}(v^*)\leq \Delta$ by definition of $\cG_{\Delta, P}$.

    For part (c), we show that $d^1(w)+d^2(w)\geq S$ for every vertex $w$.
    This immediately implies that $S\leq \frac{8}{9}n$ by \cref{lem:upper-bound-d1-d2}.
    By the definition of $S$, it suffices to show that $d^1(w)$ and $d^2(w)$ satisfy the constraints on $d_1$ and $d_2$ in the definition of $S$.
    By our choice of $X$ and $Y$, we have $d^2(w)\leq d^1(w)\leq \abs{X}+\abs{Y}$.
    We also know that $w$ is in at most
    \[\sum_{c'}\bip(\abs{N_{c'}(w)})  = \sum_{i\geq1}\bip(d^i(w))\]
    monochromatic triangles.
    As $d^3(w)+d^4(w)+\cdots  = n-1-d^1(w)-d^2(w)$ and $d^3(w),d^4(w),\ldots \leq d^2(w)$, by the convexity of $\bip(x) = \lfloor\frac{x}{2}\rfloor\lceil \frac{x}{2}\rceil$, this is maximized when all but at most one nonzero $d^i(w)$ are equal to $d^2(w)$ for $i>2$. So, with $q$ and $r$ the unique nonnegative integers such that $n-1-d^1(w)=qd^2(w)+r$ and $r<d^2(w)$, this shows that
    \[\bip(d^1(w)) + q\bip(d^2(w))+\bip(r) \geq g_3(n)-g_3(n-1)+1\]
    since $w$ is in at least $g_3(n)-g_3(n-1)+1$ monochromatic triangles.
    By the definition of $S$, we thus have $d^1(w)+d^2(w)\geq S$, as desired.
    
    For part (d), recall that $v^*$ must be in at least $g_3(n)-g_3(n-1)+1$ monochromatic triangles. The number of such triangles of color $c^*$ is at most $|X||Y|$, and the number of monochromatic triangles containing $v^*$ not of color $c^*$ is at most $\sum_{i\geq 2} \bip(d^i(v^*))$. %if we denote $d^2(v^*), d^3(v^*), \ldots$ as the degree in other colors of $v^*$ \xh{This is already defined that way right?}, then 
    Since also $\sum_{i\geq 2}d^i(v^*) = n-1-|X|-|Y|$, by convexity of $\bip(x)$, the sum $\sum_{i\geq 2} \bip(d^i(v^*))$ is maximized when $d^2(v^*) = n-1-|X|-|Y|$ and $v^*$ sees no further colors.
    Here, the fact that $\abs{X}+\abs{Y}\geq\frac{n-1}{2}$ ensures that $d^2(v^*)\leq d^1(v^*)$.
\end{proof}

We next define an elementary function $F$ which, roughly speaking, gives a lower bound on the number of precyclic triangles containing a given vertex, given its degree in color $c^*$ to each of the two other parts of $c^*$.% \ra{I replaced bad with precyclic in \cref{def:tournament setup} so I think this works now, although it does hide what $C$ does.}%\ra{We previously defined bad to mean 2-precyclic or 3-precylic, so I think this works, although it does hide what $C$ does.} \xh{Do we need to use "bad" or can we just say "precyclic" means $2$-precyclic or $3$-precyclic?}\

\begin{definition}\label{def:F}
For any $A,B,C> 0$, define a \textit{(real) $(A, B, C)$-composition} to be a pair $(a, b)$ of compositions $a_1+\ldots+a_\ell = A$ and $b_1+\ldots+b_\ell = B$ into nonnegative real numbers, of the same length $\ell$, such that $a_i + b_i \leq C$ for each $i\in[\ell]$. Write $(a, b) \vdash (A, B, C)$ if $(a, b)$ is an $(A, B, C)$-composition. Define $$F(A,B,C) = \inf_{(a, b)\vdash (A, B, C)} \sum_{i\neq j} a_i b_j.$$
%\[F(A,B,C)=\inf_{\substack{t\in \NN\\ a_1,b_1,\ldots, a_t,b_t\in \RR_{\geq 0}\\a_1+\cdots+a_t= A\\ b_1+\cdots+b_t= B\\ a_i+b_i\leq C\;\forall i\in[t]}}\sum_{i\neq j}a_ib_j.\]
We define $F(A,B,C)=0$ if $A$ or $B$ is nonpositive. %\xh{This subscript is nasty, let's come up with some cleaner notation. Maybe: "Define a \textit{(real) $(A,B,C)$-composition} to be a pair $(a,b)$ of compositions $a_1 + \cdots + a_t = A$ and $b_1 + \cdots + b_t = B$ into nonnegative real numbers, of the same length $t$, for which $a_i + b_i \le C$ termwise. We write $(a,b)\vdash (A,B,C)$ if $(a,b)$ is an $(A,B,C)$-composition. Let $F(A,B,C) = \inf_{(a,b)\vdash (A,B,C)} \sum a_i b_j$, where the infimum is over $(A,B,C)$-compositions."}
\end{definition}

We first state three facts about $F$. The proofs are routine and can be found in Appendix \ref{sec:proof of F lemma}.

\begin{restatable}{lemma}{flemma}\label{lemma:elementary} The following are true.
\begin{enumerate}[(a)]
    \item The infimum in the definition of $F(A,B,C)$ is attained by $\ell=\lceil (A+B)/C\rceil$, $a_i+b_i=C$ for all $i<\ell$, and $a_i-b_i=a_j-b_j$ for all $i<j$ with the possible exception of $j=\ell$. If $a_1-b_1 \ne a_\ell-b_\ell$, then $\min(a_\ell,b_\ell)=0$ and $a_\ell-b_\ell$ is between $0$ and $a_1-b_1$.

    Moreover, there is exactly one minimizer satisfying the above.
    \item The function $F(A,B,C)$ is increasing in $A$ and $B$, and decreasing in $C$.
    \item We have
    \[\frac{B'}{B}F(A,B,C)\geq F(A,B',C)\]
    for any $B\geq B'>0$.
\end{enumerate}
\end{restatable}

We now introduce more notation. Recall that $N_{c,1}(v)$ and $N_{c,2}(v)$ are the two parts of the color-$c$ neighborhood of vertex $v$.

\begin{definition}
Define $M:=\max\abs{N_{c,j}(v)}$, where the max is taken over vertices $v$, colors $c$, and $j\in [2]$.
For any set of vertices $S$, any vertex $u\not\in S$ and any color $c$, let $b^{(c)}_S(u)$ be the number of edges from $u$ to $S$ not of color $c$.
We drop the superscript if $c$ is clear from context. Also, for any sets of vertices $S_1, S_2, S_3$, let $\pc_2(S_1, S_2, S_3)$ and $\pc_3(S_1, S_2, S_3)$ denote the number of $2$-precyclic and $3$-precyclic triangles (respectively) with one vertex in each of $S_1, S_2$, and $S_3$. We shorten $\pc_i(\{v\}, S_1, S_2)$ to simply $\pc_i(v, S_1, S_2)$. 
\end{definition}

We now use the function $F$ to argue that certain colored neighborhoods of a vertex are large.

\begin{lemma}\label{lemma:bipartite-to-tripartite}
    Suppose that $S_1,S_2\subseteq V$ are two disjoint sets of vertices and $c$ is a color such that $c(v)$ is constant on each of $S_1$ and $S_2$ with $c(S_1)\neq c(S_2)$.
    Suppose that there are at most $b$ pairs $(s_1,s_2)\in S_1\times S_2$ with $\chi(s_1s_2)\neq c$.
    Let $Q$ be a nonnegative integer. 
    \begin{enumerate}[(a)]
        \item For any $v\in V$ with $c(v)\not\in \{c(S_1),c(S_2)\}$, if $2\pc_2(v, S_1, S_2) + \pc_3(v, S_1, S_2) \leq Q$ and
        \[F\left(\abs{S_1}, \abs{S_2}-\left\lfloor \frac{Q+2b}{2\abs{S_1}}\right\rfloor,M\right)>Q+b,\]
        then $b_{S_2}(v)\leq \frac{1}{2}\abs{S_2}$.
        \item If $b_{S_2}(v)\leq \frac{1}{2}\abs{S_2}$ for some $v\in V$ with $c(v)\not\in \{c(S_1),c(S_2)\}$, then $\pc_2(v, S_1, S_2) \geq\abs{S_1}b_{S_2}(v)-b$. %\xh{Is part (b) ever used separately from part (a)? If not, can we just combine them?}
    \end{enumerate}
\end{lemma}
\begin{proof}
    Let $v$ be a vertex with $c(v)\not\in \{c(S_1),c(S_2)\}$.
    Note that for every $(s_1,s_2)\in S_1\times S_2$, if exactly one of $\chi(vs_1),\chi(vs_2)$ is equal to $c$, then $vs_1s_2$ is a $2$-precyclic triangle unless $\chi(s_1s_2)\neq c$.
    Therefore, we have
    \begin{align}\label{eq:pc2 lower bound}\pc_2(v, S_1, S_2) \geq b_{S_1}(v)(\abs{S_2}-b_{S_2}(v))+(\abs{S_1}-b_{S_1}(v))b_{S_2}(v)-b.\end{align}

    Let us first prove the second item.
    If $b_{S_2}(v)\leq \frac{1}{2}\abs{S_2}$, then we know that $b_{S_2}(v)\leq \abs{S_2}-b_{S_2}(v)$.
    So, from \eqref{eq:pc2 lower bound}, we obtain
    \[\pc_2(v, S_1, S_2) \geq b_{S_1}(v)b_{S_2}(v)+(\abs{S_1}-b_{S_1}(v))b_{S_2}(v)-b = \abs{S_1}b_{S_2}(v)-b,\] as desired.

    We now prove the first item by contradiction.
    Assume that $b_{S_2}(v)>\frac{1}{2}\abs{S_2}$.
    Then we have $b_{S_2}(v) > \abs{S_2}-b_{S_2}(v)$, %\xh{this inequality is flipped?}\ra{fixed}
    so from \eqref{eq:pc2 lower bound} we obtain 
    \begin{align*}
        \pc_2(v, S_1, S_2) &\geq  b_{S_1}(v)(\abs{S_2}-b_{S_2}(v))+(\abs{S_1}-b_{S_1}(v))(\abs{S_2}-b_{S_2}(v))-b\\&=  \abs{S_1}(\abs{S_2}-b_{S_2}(v))-b.
    \end{align*}
    By the assumption, we know that this is at most $Q/2$, which shows that
    \[\abs{S_2}-b_{S_2}(v) \leq \frac{Q+2b}{2\abs{S_1}},\]
    and so 
    \[b_{S_2}(v)\geq \abs{S_2}-\left\lfloor \frac{Q+2b}{2\abs{S_1}}\right\rfloor\]
    using the integrality of $b_{S_2}(v)$.

    Let $S_2'$ be the subset of $S_2$ containing vertices $s_2$ with $\chi(vs_2)\neq c$.
    Then by definition $\abs{S_2'} = b_{S_2}(v)$. 
    
    For any $s\in S_1\cup S_2'$, we denote $\sigma(s) := (c', c'(s))$ where $c'=\chi(vs)$. That is, $\sigma(s)$ is the ordered pair consisting of the color $\chi(vs)$ as well as the label of the part of $s$ in that color's tripartition.
    We show that for any $(s_1,s_2)\in S_1\times S_2'$, if $\sigma(s_1)\neq \sigma(s_2)$ %\xh{This notation is sufficiently confusing to merit some more explanation, for example, we could add a line "define the signature of a vertex $s \in S_1 \cup S_2'$ to be the pair $\sigma(s)\coloneqq (c, c(v))$, where $c=\chi(vs)$"}
    and $\chi(s_1s_2)=c$, then $vs_1s_2$ is a precyclic triangle.
    If $\chi(vs_1)=c$, then since $s_2\in S_2'$ and so $\chi(vs_2)\neq c$, we know that $vs_1s_2$ is a $2$-precyclic triangle.
    We now just need to deal with the case $\chi(vs_1)\neq c$.
    If $\chi(vs_1)\neq \chi(vs_2)$, we see that $vs_1s_2$ is a $3$-precyclic triangle.
    If $\chi(vs_1)=\chi(vs_2)=c'$, then we must have $c'(s_1)\neq c'(s_2)$ and so $vs_1s_2$ is a $2$-precyclic triangle as $\chi(s_1s_2)=c$ by our assumption.

    Let $N$ be the number of pairs $(s_1,s_2)\in S_1\times S_2'$ with $\sigma(s_1)\neq \sigma(s_2)$. We now show that $N > Q+b$. %\xh{We only show $N > Q + b$ later?}
    Let $(c_1,i_1),\ldots, (c_\ell,i_\ell)$ be the possible values of the ordered pair $\sigma(s)$ for $s\in S_1\cup S_2'$.
    %Suppose that the $\ell$ possibilities are $(c_1,i_1),\ldots, (c_\ell,i_\ell)$.
    For any $j\in[\ell]$, let $a_j$ be the number of $s_1\in S_1$ with $\sigma(s_1)=(c_j,i_j)$ and let $b_j$ be the number of $s_2\in S_2'$ with $\sigma(s_2)=(c_j,i_j)$.
    Then we know that $\sum_{j}a_j=\abs{S_1}$ and $\sum_jb_j = \abs{S_2'} = b_{S_2}(v)$.
    Moreover, for each $j\in[\ell]$ we know that there are at least  $a_j+b_j$ vertices $u$ in the graph with $\chi(vu) = c_j$ and $c_j(u)) = i_j$.
    Therefore if $k\in\{1, 2\}$ is such that $i_j\equiv c_j(v)+k\pmod 3$, then there are at least $a_j+b_j$ vertices in $N_{c_j,k}(v)$, showing that $a_j+b_j\leq M$ by the definition of $M$.  We conclude that $(a,b)$ is a $(|S_1|, b_{S_2}(v), M)$-composition.
    Since the number of pairs $(s_1,s_2)\in S_1\times S_2'$ with $\sigma(s_1)\neq \sigma(s_2)$ is simply
    $\sum_{i\neq j} a_ib_j,$
    by the definition of $F$ we have
    \[N\geq F\left(\abs{S_1},b_{S_2}(v),M\right)\geq F\left(\abs{S_1},\abs{S_2}-\left\lfloor\frac{Q+2b}{2\abs{S_1}}\right\rfloor,M\right)>Q+b.\]
    Here we used \cref{lemma:elementary}(b) and the assumption of the current lemma.
    Now among the $N$ pairs, there are at most $b$ pairs $(s_1,s_2)$ with $\chi(s_1s_2)\neq c$.
    Therefore there are more than $Q$ pairs $(s_1, s_2)\in S_1\times S_2'$ that form a precyclic triangle with $v$, which is a contradiction.
\end{proof}

We can use \cref{lemma:bipartite-to-tripartite} to begin bounding the number of edges not colored $c^*$ between the parts $U, X^*,$ and $Y^*$. We need a simple upper bound on $M$.
\begin{lemma}\label{lemma:First-step} We have $M \leq \frac{\abs{X}+\abs{Y}+\Delta}{2}$.
\end{lemma}
\begin{proof}
By definition of $X$ and $Y$, for any vertex $v\in V$ and any color $c$, we have $\abs{N_c(v)}\leq \abs{X}+\abs{Y}$.
We also know that
\[\max\left(\abs{N_{c,1}(v)}, \abs{N_{c,2}(v)}\right) = \frac{\abs{N_c(v)}+\delta_c(v)}{2}\leq \frac{\abs{X}+\abs{Y}+\Delta}{2}\]
since $\delta_c(v)\leq \Delta$ by definition of $\cG_{\Delta,P}$.
Taking the max over $v$ and $c$ of both sides yields $M\leq \frac{\abs{X}+\abs{Y}+\Delta}{2}$.
\end{proof}

Let $b_{XY}$ be the number of edges between $X$ and $Y$ not colored $c^*$, and let $p:=\pc_2(U\sm\{v^*\}, X, Y)$. For the remainder of this section, let $b_S(v)$ refer to $b_S^{(c^*)}(v)$ unless otherwise specified. Recall the definition of admissible pairs from \cref{def:admissible-pair}.

\begin{lemma}[$\bigstar$]\label{lemma:Second-step}
    Suppose the assumption in \cref{lemma:possible-X-Y} holds, and assume that for all admissible pairs $(s, t)$, we have
    \[F\left(s,t-\left\lfloor \frac{P}{2s}\right\rfloor, \left\lfloor\frac{s+t+\Delta}{2}\right\rfloor\right)>P.\]
    Then, $b_X(u)\leq \frac{1}{2}\abs{X}$ and $b_Y(u)\leq \frac{1}{2}\abs{Y}$ for all $u\in U$.
    If this is the case, then
    \begin{equation}\label{eq:bound-UX-UY}
        \sum_{u\in U}b_X(u)\leq \frac{\abs{U}b_{XY}+p}{\abs{Y}} \quad\text{ and }\quad
        \sum_{u\in U}b_Y(u)\leq \frac{\abs{U}b_{XY}+p}{\abs{X}}.
    \end{equation}
\end{lemma}

\begin{proof}
    By \cref{lemma:possible-X-Y} and \cref{lemma:First-step}, we know that $(\abs{X},\abs{Y})$ form an admissible pair and $M\leq \lfloor\frac{\abs{X}+\abs{Y}+\Delta}{2}\rfloor$.
    The assumption implies that
    \[F\left(\abs{X},\abs{Y}-\left\lfloor \frac{P}{2s}\right\rfloor, M\right)\geq F\left(\abs{X},\abs{Y}-\left\lfloor \frac{P}{2s}\right\rfloor, \left\lfloor\frac{\abs{X}+\abs{Y}+\Delta}{2}\right\rfloor\right)>P,\] where the first inequality above used \cref{lemma:elementary}(b).

    If $u = v^*$, then by definition of $X$ and $Y$, $b_X(u) = b_Y(u) = 0$, and the first conclusion of the lemma holds trivially. So, assume $u\in U\sm\{v^*\}$. Since $v^*$ creates $b_{XY}$ $2$-precyclic triangles with $X$ and $Y$ and there are at most $P$ total precyclic triangles with $2$-precyclic triangles being counted twice, we know that $u$ creates at most $P-2b_{XY}$ precyclic triangles with $X$ and $Y$ with $2$-precyclic triangles being counted twice. %\xh{Again I think we want a notation for "number of precyclic triangles across 3 parts U, V, W." Perhaps just define $\textnormal{pc}_2(U,V,W)$, $\textnormal{pc}_3(U,V,W)$, $\textnormal{pc}_2(u,V,W)$, $\textnormal{pc}_3(u,V,W)$}
    Thus, apply \cref{lemma:bipartite-to-tripartite} with $S_1=X,$ $S_2=Y,$ $b=b_{XY}$, and $Q = P-2b_{XY}$. We obtain that $b_Y(u)\leq \frac{1}{2}\abs{Y}$ and $u$ creates at least $\abs{X}b_Y(u)-b_{XY}$ $2$-precyclic triangles with $X$ and $Y$.
    Summing over $u\in U\sm\{v^*\}$, we have
    \[\abs{X}\sum_{u\in U\sm\{v^*\}}b_Y(u)\leq (\abs{U}-1)b_{XY}+p,\] by definition of $p$. Since $b_Y(v^*) = 0$ and $b_{XY}\geq 0$, we conclude that \[\abs{X}\sum_{u\in U}b_Y(u)\leq \abs{U}b_{XY}+p.\]
    Rearranging gives half of \eqref{eq:bound-UX-UY}, and swapping $X$ and $Y$ gives us $b_X(u)\leq \frac{1}{2}\abs{X}$ and the other half of \eqref{eq:bound-UX-UY}. (Note that the assumption of the lemma also holds with $s=|Y|$ and $t=|X|$ by symmetry.)
\end{proof}

So far, we have seen lower bounds on $|X|$ and $|Y|$. The next lemma helps us lower bound $|U|$. Recall that $\pc_2$ and $\pc_3$ are the number of $2$-precyclic and $3$-precyclic triangles under $\chi$, respectively. We use the standard notation $(x)_+ = \max(x, 0)$ and write $(x)_+^2$ for $((x)_+)^2$.
\begin{lemma}\label{lemma:u-lower-bound}
    The part sizes satisfy
    \begin{align*} \abs{X}\left(\abs{Y}-\frac{b_{XY}}{\abs{X}}-\abs{U}\right)_+^2+\abs{Y}\left(\abs{X}-\frac{b_{XY}}{\abs{Y}}-\abs{U}\right)_+^2 + \left(|X|-|Y|\right)^2\leq \tilde d(n)-4\pc_2-2\pc_3.\end{align*} %\xh{I don't think the double parens are necessary}
    %\xh{Make a variable for the expression $\frac{n(n+1)(n-1)}{3}-8(g_3(n)+1)$ that appears repeatedly, will make this inequality fit in one line. It is closely related to $d(n)$ but not exactly the same, so maybe something that looks similar like $D(n)$, $\delta(n)$ or $\bar{d}(n)$?}
\end{lemma}
\begin{proof}
    For every $x\in X$, its $c^*$ neighborhood in $Y$ has size at least $\abs{Y} - b_Y(x)$, and its $c^*$ neighborhood in $U$ has size at most $\abs{U}$, so we have that $\delta_{c^*}(x)\geq (\abs{Y}-b_Y(x)-\abs{U})_+$.
    Therefore
    \[\sum_{x\in X}\delta_{c^*}(x)^2\geq \sum_{x\in X}(\abs{Y}-b_Y(x)-\abs{U})_+^2.\]
    Note that the function $x\mapsto (x)_+^2$ is convex, and so by Jensen's inequality the right hand side is at least
    \[\abs{X}\left(\abs{Y}-\frac{\sum_{x\in X}b_Y(x)}{\abs{X}}-\abs{U}\right)_+^2=\abs{X}\left(\abs{Y}-\frac{b_{XY}}{\abs{X}}-\abs{U}\right)_+^2.\]
    %A similar bound holds for $\sum_{y\in Y}\delta_{c^*}(y)^2$. 
    By the same argument swapping $X$ and $Y$, we also have
    \[
    \sum_{y\in Y}\delta_{c^*}(y)^2\geq \abs{Y}\left(\abs{X}-\frac{b_{XY}}{\abs{Y}}-\abs{U}\right)_+^2.
    \]
    By \cref{lemma:basic-facts}(b), we have
    \begin{align*} \tilde d(n)-4\pc_2-2\pc_3&\geq \sum_{x\in X}\delta_{c^*}(x)^2+\sum_{y\in Y}\delta_{c^*}(y)^2 + \delta_{c^*}(v^*)^2 \\
    &\geq \abs{X}\left(\abs{Y}-\frac{b_{XY}}{\abs{X}}-\abs{U}\right)_+^2+\abs{Y}\left(\abs{X}-\frac{b_{XY}}{\abs{Y}}-\abs{U}\right)_+^2 + \left(|X|-|Y|\right)^2,\end{align*}
    proving the desired inequality.
\end{proof}

Now, we combine several lemmas to prove that, assuming certain conditions on part sizes, there are few edges not colored $c^*$ between $\overline{X},Y$ and between $X, \overline{Y}$.
\begin{lemma}[$\bigstar$]\label{lemma:Third-step}
    Suppose the assumptions of \cref{lemma:possible-X-Y} and \cref{lemma:Second-step} hold, and assume that $\tilde d(n) < n^2/4$ and that for all admissible pairs $(s,t)$ 
   % \xh{This assumption is confusing to parse, I don't know if we have freedom to choose $s$ and $t$ or if they are exactly $s = |X|$, $t = |Y|$?} 
   and nonnegative integers $k$ with $0\leq k\leq \floor{P/2}$, the following holds: If there exists a positive integer $w$ at most $n-s-t$ such that 
    \begin{align}\label{eq:w-required}s\left(t-\frac{k}{s}-w\right)_+^2+t\left(s-\frac{k}{t}-w\right)_+^2+(s-t)^2\leq \tilde d(n)-4k,\end{align}
    then the minimum such $w$ satisfies
    \begin{align}\label{eq:w-inequality1} F\left(s,w-\frac{\lfloor\frac{P}{2}\rfloor-k+\frac{wk}{t}}{s},\left\lfloor\frac{s+t+\Delta}{2}\right\rfloor\right)>P-2k+\frac{wk}{t}\end{align}
    and
    \begin{align}\label{eq:w-inequality2} F\left(w,s-\frac{\lfloor\frac{P}{2}\rfloor-k+\frac{wk}{t}}{w},\left\lfloor\frac{s+t+\Delta}{2}\right\rfloor\right)>P-2k+\frac{wk}{t}.\end{align}

    Then, we can conclude that for all $x\in \overline{X}$ we have $b_Y(x)\leq \frac{1}{2}\abs{Y}$ and $b_U(x)\leq \frac{1}{2}\abs{U}$, and for all $y\in \overline{Y}$ we have $b_X(y)\leq \frac{1}{2}\abs{X}$ and $b_U(y)\leq \frac{1}{2}\abs{U}$.
\end{lemma}

\begin{proof}
    We know that $s=\abs{X}$, $t=\abs{Y}$ and $k=b_{XY}+p$ satisfy all the conditions; in particular, $b_{XY}$ and $p$ count disjoint sets of $2$-precyclic triangles, ensuring that $k \leq \floor{P/2}$. 
    By \cref{lemma:u-lower-bound}, we know that if $|X|=s$ and $|Y|=t$, then $w=|U|$ satisfies \eqref{eq:w-required}. (In particular, replacing $b_{XY}$ with $b_{XY}+p$ in \cref{lemma:u-lower-bound} only makes the left hand side smaller, and replacing $\pc_3$ with $0$ and $\pc_2$ with $k$ only makes the right hand side larger.) As a consequence, if $|X|=s$, $|Y|=t$, and $b_{XY}+p=k$, then a minimum integer $w$ satisfying \eqref{eq:w-required} exists, and we have $|U|\geq w$. %\xh{I am still concerned about this preceding paragraph. First we should spell out why $k=b_{XY} + p$ satisfies the condition of being less than $P/2$.}\ra{Done}\xh{Next, the last two sentences would make more sense if they just read "By Lemma 4.13, we know that $w=|U|$ satisfies the (numbered equation reference), so such a minimum $w$ exists."} %\xh{I am still confused about this preceding paragraph. Are we showing that $w$ always exists, or that something else is trivial if $w$ doesn't exist?}

    %\xh{Reverse the following calculation to add signposting. You should start with: We now show $b_U(y) \le \frac{1}{2} |U|$. By Lemma 4.11, it suffices to check that...}\ra{Done}
    We now show that $b_U(y) \leq \frac{1}{2}|U|$ for all $y\in \overline Y$. To do so, we apply \cref{lemma:bipartite-to-tripartite}(a) with $S_1=X, S_2=U, Q=P-2(b_{XY}+p)$, 
    %\xh{Split these parentheticals out into a full sentence explaining what is being checked} 
    and $b=\frac{\abs{U}(b_{XY}+p)}{\abs{Y}}$. This choice of $Q$ is valid since $b_{XY}$ and $p$ count $2$-precyclic triangles which do not use any $y\in\overline Y$, so for $y\in\overline Y$ we have $2\pc_2(y, X, U) + \pc_3(y, X, U) \leq Q$, as needed. The choice of $b$ is valid as well since \eqref{eq:bound-UX-UY} in \cref{lemma:Second-step} ensures that there are at most $b$ pairs $(x,u)\in X\times U$ with $\chi(xu)\neq c^*$. Thus, setting $\alpha = 1-\frac{b_{XY}+p}{\abs{X}\abs{Y}}$, $\beta = \frac{\lfloor{P/2}\rfloor-(b_{XY}+p)}{\abs{X}}$ and $\gamma =P-2(b_{XY}+p)$, to apply the lemma it remains to show that %\xh{I think define the shorthands $\alpha, \beta, \gamma$ before writing the equation to make it more readable. }
    \begin{align}\label{eq:needtoshow1}F\left(\abs{X},\alpha|U|-\beta, M\right)>\gamma+\frac{\abs{U}(b_{XY}+p)}{\abs{Y}}.\end{align}
    We claim that $\alpha  > 0$. Indeed, $k \leq P/2 \leq \frac{\tilde d(n)-\Delta^2}{4}$ by \eqref{eq:corcond1}. 
    Furthermore, by \cref{lemma:possible-X-Y} we have that $(\abs{X},\abs{Y})$ is an admissible pair, which shows that $\abs{X}+\abs{Y} \geq n/2$ and $\abs{\abs{X}-\abs{Y}} \leq \Delta$, and so $\abs{X}\abs{Y} \geq n^2/16 - \Delta^2/4$. Hence the assumption $\tilde d(n) < n^2/4$ gives us the claim. Recalling that $F(A,B,C) = 0$ if $A\leq 0$ or $B\leq 0$, the assumption \eqref{eq:w-inequality1} implies that $\alpha w - \beta > 0$, and then since $\alpha>0$ and $|U|\geq w$, we have $\alpha |U| - \beta > 0$ as well.
    
    Since $M\leq \lfloor \frac{\abs{X}+\abs{Y}+\Delta}{2}\rfloor$ by \cref{lemma:First-step} and $F(A, B, C)$ is decreasing in $C$, \eqref{eq:w-inequality1} tells us that
    \[\frac{b_{XY}+p}{\abs{Y}}w+\gamma<F(\abs{X},\alpha w-\beta, M)\leq \frac{\alpha w-\beta}{\alpha\abs{U}-\beta}F(\abs{X},\alpha\abs{U}-\beta,M),\]
    where for the second inequality we used \cref{lemma:elementary}(c).
    Hence
    \[F(\abs{X},\alpha\abs{U}-\beta,M)> \frac{\alpha\abs{U}-\beta}{\alpha w-\beta}\left(\frac{b_{XY}+p}{\abs{Y}}w+\gamma\right)\geq \frac{\abs{U}}{w}\left(\frac{b_{XY}+p}{\abs{Y}}w+\gamma\right)\geq \frac{b_{XY}+p}{\abs{Y}}\abs{U}+\gamma,\]
    %\xh{It is opaque why $\alpha |U| - \beta$ and $\alpha w - \beta$ are positive, which is needed to make these inequalities true/sensible.}\hy{So I think we should say that if $P-2k+wk/t < 0$ then it does not matter; if $P-2k+wk/t\geq 0$ then the assumption forces $\alpha w-\beta > 0$ and so $\alpha |U|-\beta > 0$ as well (okay but we also need to prove $\alpha>0$ which is $st < k$, whose proof is also absent}
    establishing \eqref{eq:needtoshow1}. This verifies all assumptions for \cref{lemma:bipartite-to-tripartite}(a), so we conclude that $b_U(y) \le \frac{1}{2} |U|$ for all $y \in \overline{Y}$.
    
    %Set $\alpha = 1-\frac{k}{ts}$, $\beta = \frac{\lfloor{P/2}\rfloor-k}{s}$ and $\gamma =P-2k$.
    %Since $M\leq \lfloor \frac{\abs{X}+\abs{Y}+\Delta}{2}\rfloor$ and $F(A, B, C)$ is decreasing in $C$, the assumption of the lemma tells us that
    %\[\frac{b_{XY}+p}{\abs{Y}}w+\gamma<F(\abs{X},\alpha w-\beta, M)\leq \frac{\alpha w-\beta}{\alpha\abs{U}-\beta}F(\abs{X},\alpha\abs{U}-\beta,M),\]
    %where for the second inequality we used \cref{lemma:elementary}(c).
    %Hence
    %\[F(\abs{X},\alpha\abs{U}-\beta,M)> \frac{\alpha\abs{U}-\beta}{\alpha w-\beta}\left(\frac{b_{XY}+p}{\abs{Y}}w+\gamma\right)\geq \frac{\abs{U}}{w}\left(\frac{b_{XY}+p}{\abs{Y}}w+\gamma\right)\geq \frac{b_{XY}+p}{\abs{Y}}\abs{U}+\gamma.\]
    %\xh{It is opaque why $\alpha |U| - \beta$ and $\alpha w - \beta$ are positive, which is needed to make these inequalities true/sensible.}
    %Plugging in the definitions of $\alpha,$ $\beta,$ and $\gamma$, we get
    %\[F\left(\abs{X},\abs{U}-\frac{\lfloor\frac{P}{2}\rfloor-(b_{XY}+p)+\frac{\abs{U}(b_{XY}+p)}{\abs{Y}}}{\abs{X}}, M\right)>P-2(b_{XY}+p)+\frac{\abs{U}(b_{XY}+p)}{\abs{Y}}.\]
    %Apply \cref{lemma:bipartite-to-tripartite} with $S_1=X, S_2=U, Q=P-2(b_{XY}+p)$ and $b=\frac{\abs{U}(b_{XY}+p)}{\abs{Y}}$ (which is valid by \eqref{eq:bound-UX-UY} in \cref{lemma:Second-step}), and we see that for each $y\in \overline{Y}$ we have $b_U(y)\leq \frac{1}{2}\abs{U}$.

    Now we similarly prove $b_X(y)\leq \frac{1}{2}\abs{X}$ for all $y\in \overline{Y}$. %\xh{similarly here, start by explaining what it suffices to check using lemma 4.11} 
    We apply
    \cref{lemma:bipartite-to-tripartite}(a) with $S_1=U$, $S_2=X$, $Q=P-2(b_{XY}+p)$ and $b = \frac{\abs{U}(b_{XY}+p)}{\abs{Y}}$. To do so, it remains to show that
    \begin{align}\label{eq:needtoshow2}
        F\left(\abs{U}, \abs{X}-\frac{b_{XY}+p}{\abs{Y}}-\frac{\floor{P/2}-(b_{XY}+p)}{\abs{U}},M\right)
        %> \frac{\abs{U}}{w}\left(\frac{b_{XY}+p}{\abs{Y}}w+\gamma\right)
        \geq \gamma+\frac{\abs{U}(b_{XY}+p)}{\abs{Y}}.
    \end{align}
    From the assumption \eqref{eq:w-inequality2} and the fact that $M\leq \frac{\abs{X}+\abs{Y}+\Delta}{2}$, we have
    \begin{align*}\frac{b_{XY}+p}{\abs{Y}}w+\gamma &< F\left(w,\abs{X}-\frac{b_{XY}+p}{\abs{Y}}-\frac{\floor{P/2}-(b_{XY}+p)}{w},M\right)\\&\leq \frac{w}{\abs{U}}F\left(\abs{U}, \abs{X}-\frac{b_{XY}+p}{\abs{Y}}-\frac{\floor{P/2}-(b_{XY}+p)}{\abs{U}},M\right),\end{align*}
    %\hy{I guess we have a similar issue here}
    where in the second line we used the monotonicity of $F(A,B,C)$ with respect to $B$, the symmetry between $A$ and $B$, and \cref{lemma:elementary}(c). Thus, we have
    \[F\left(\abs{U}, \abs{X}-\frac{b_{XY}+p}{\abs{Y}}-\frac{\floor{P/2}-(b_{XY}+p)}{\abs{U}},M\right)
    > \frac{\abs{U}}{w}\left(\frac{b_{XY}+p}{\abs{Y}}w+\gamma\right),\] which establishes \eqref{eq:needtoshow2} since $|U|\geq w$. Thus we conclude that $b_X(y) \leq \frac{1}{2}|X|$ for all $y\in \overline Y$, as desired.

    % From the second inequality given by the assumption, we have
    % \[\frac{b_{XY}+p}{\abs{Y}}w+\gamma < F\left(w,\abs{X}-\frac{b_{XY}+p}{\abs{Y}}-\frac{\floor{P/2}-(b_{XY}+p)}{w},M\right)\]
    % where we again use that $M\leq \frac{\abs{X}+\abs{Y}+\Delta}{2}$.
    % By the monotonicity of $F(A,B,C)$ with respect to $B$, the symmetry between $A$ and $B$, and \cref{lemma:elementary}(c), we have
    % \begin{align*}
    %    F\left(w,\abs{X}-\frac{b_{XY}+p}{\abs{Y}}-\frac{\floor{P/2}-(b_{XY}+p)}{w},M\right)\leq \frac{w}{\abs{U}}F\left(\abs{U}, \abs{X}-\frac{b_{XY}+p}{\abs{Y}}-\frac{\floor{P/2}-(b_{XY}+p)}{\abs{U}},M\right),
    % \end{align*}
    % and so
    % \begin{align*}
    %     F\left(\abs{U}, \abs{X}-\frac{b_{XY}+p}{\abs{Y}}-\frac{\floor{P/2}-(b_{XY}+p)}{\abs{U}},M\right)
    %     > \frac{\abs{U}}{w}\left(\frac{b_{XY}+p}{\abs{Y}}w+\gamma\right)\geq P-2(b_{XY}+p)+\frac{\abs{U}(b_{XY}+p)}{\abs{Y}}.
    % \end{align*}
    % Apply \cref{lemma:bipartite-to-tripartite} with $S_1=U$, $S_2=X$, $Q=P-2(b_{XY}+p)$ and $b = \frac{\abs{U}(b_{XY}+p)}{\abs{Y}}$. We obtain that $b_X(y)\leq \frac{1}{2}\abs{X}$ for each $y\in \overline{Y}$.

    By symmetry between $X$ and $Y$, we also have $b_U(x)\leq \frac{1}{2}\abs{U}$ and $b_Y(x)\leq \frac{1}{2}\abs{Y}$, completing the proof.
\end{proof}

The next lemma provides an upper bound on the number of certain kinds of edges not colored $c^*$.

\begin{lemma}\label{lem:sum-Xy-bound}
Suppose the assumptions of \cref{lemma:possible-X-Y}, \cref{lemma:Second-step}, and \cref{lemma:Third-step} hold. Then,
\begin{equation*}
    \sum_{y\in \overline{Y}}b_X(y)\leq \frac{\abs{Y}\pc_2(U, X, \overline{Y})+\left(\abs{U}b_{XY}+p\right)\abs{\overline{Y}}}{\abs{U}\abs{Y}} \ \  \text{ and } \ \  \sum_{y\in \overline{Y}}b_U(y)\leq \frac{\abs{Y}\pc_2(U, X, \overline{Y})+\left(\abs{U}b_{XY}+p\right)\abs{\overline{Y}}}{\abs{X}\abs{Y}}.
\end{equation*}
\end{lemma}
\begin{proof}
    % For every $y\in \overline{Y}$, we know that each $x\in X$ and $u\in U$ with $\chi(xu)=c^*$ and exactly one of $\chi(xy), \chi(uy)$ being $c^*$ form a $2$-precyclic triangle with $y$.
    % As a consequence, we have %the total number of $2$-precyclic triangles formed this way %\xh{A little ambiguous, are we just saying $\pc_2(\overline{Y}, X, U) \ge \ldots$?} is lower bounded by \xh{Generally, write "is at least" or "is greater than" (if strict) instead of "is lower bounded by", similarly for upper bounded by. Among other things this removes the amibiguity about strictness.}
    % \[\pc_2(U, X, \overline Y) \geq \sum_{y\in \overline{Y}}\left(b_X(y)(\abs{U}-b_U(y))+b_U(y)(\abs{X}-b_X(y))-\frac{\abs{U}b_{XY}+p}{\abs{Y}}\right)_+,\]
    % where we used \cref{lemma:Second-step} to obtain that the number of edges between $X$ and $U$ not colored $c^*$ is at most $\frac{\abs{U}b_{XY}+p}{\abs{Y}}$.
    % Using \cref{lemma:Third-step}, we have $$\sum_{y\in \overline{Y}} (-2b_X(y)b_U(y) + b_U(y)|X|) \geq \sum_{y\in \overline{Y}} \left(-2\left(\frac{1}{2}|X|\right)b_U(y) + b_U(y)|X|\right) \geq 0.$$ So, we obtain
    % \[\pc_2(U, X, \overline Y)\geq -\frac{(\abs{U}b_{XY}+p)\abs{\overline{Y}}}{\abs{Y}}+\sum_{y\in \overline{Y}}b_X(y)\abs{U}.\]

    For each $y\in\overline{Y}$, by part (b) of \cref{lemma:bipartite-to-tripartite} with $S_1=U$, $S_2= X$, $v=y$ and $b=\frac{\abs{U}b_{XY}+p}{\abs{Y}}$ (which is valid by \cref{lemma:Second-step}), we have
    \[\pc_2(y,U,X)\geq b_X(y)\abs{U}-\frac{\abs{U}b_{XY}+p}{\abs{Y}}.\]
    Summing over $y\in\overline{Y}$, we get
    \[\pc_2(U,X,\overline{Y})= \sum_{y\in\overline{Y}}\pc_2(y,U,X)\geq -\frac{(\abs{U}b_{XY}+p)\abs{\overline{Y}}}{\abs{Y}}+\sum_{y\in\overline{Y}}b_X(y)\abs{U}.\]
    By rearranging the above inequality, the first part of the lemma follows. The same exact argument holds with the roles of $X$ and $U$ swapped, so we also obtain the second part of the lemma. 
\end{proof}

For the next lemma, we introduce some more notation for conciseness. Let $\eta$ be any upper bound on the quantity $\frac{\floor{P/2}}{\abs{X}\abs{Y}}.$ %\xh{This is strange, it reads as if the rest of the proof takes $\eta$ as a variable. Will we take $\eta = 1/8$ everywhere? If so, maybe not worth defining, just write $7/8$ everywhere (use a macro for $7/8$ if we're not confident about it and might need to change later).} (If \eqref{eta<=1/8} holds, we can take $\eta = 1/8$.)
Set 
    \begin{equation}\label{eq:overlineM}\overline{M}=\min\left(\left\lfloor \frac{\floor{P/2}-b_{XY}}{(1-\eta)\left(\frac{n}{4}-\frac{\Delta}{2}\right)}\left(1+\frac{1}{\abs{U}}\right)\right\rfloor, \abs{X^*}+\abs{Y^*}-\left\lceil\frac{n}{2}\right\rceil\right).\end{equation}
    Furthermore, using the standard notation $(x)_- = \min(x, 0)$, set $\overline{X}_{\max} = \overline{M}+(\abs{X^*}-\abs{Y^*}+\Delta)_-$ and $\overline{Y}_{\max} = \overline{M}+(\abs{Y^*}-\abs{X^*}+\Delta)_-$.
    Set $X_{\min} = \abs{X^*}-\overline{X}_{\max}$ and $Y_{\min} = \abs{Y^*}-\overline{Y}_{\max}$.

%The next lemma states several more restrictions on the sizes of $X^*, X, \overline{X}$ and $Y,' Y, \overline{Y}$.

\begin{lemma}\label{lem:overlineM}
Suppose the assumptions of \cref{lemma:possible-X-Y}, \cref{lemma:Second-step}, and \cref{lemma:Third-step} hold. Then, we have the following inequalities. %The following inequalities hold. %\xh{Highlight whether these are unconditional or conditional on some assumptions like previous lemmas.} 
\begin{enumerate}[(a)]
    \item $\abs{\abs{X^*}-\abs{Y^*}}\leq \overline{M}+\Delta$.
    \item $\abs{\overline{X}} \leq \overline{X}_{\max}$;\ \  $\abs{\overline{Y}} \leq \overline{Y}_{\max}$;\ \  $|X| \geq X_{\min}$;\ \  and $|Y|\geq Y_{\min}$.
    \item $X_{\min}\left((1-\eta)Y_{\min}-\abs{U}\right)_+^2+Y_{\min}\left((1-\eta)X_{\min}-\abs{U}\right)_+^2\leq \tilde d(n)-4b_{XY}$. %eq:u-lower-bound}
    \item %\label{eq:u-upper-bound}
    $X_{\min}\left((1-\eta)\abs{U}-\abs{Y^*}\right)_+^2+Y_{\min}\left((1-\eta)\abs{U}-\abs{X^*}\right)_+^2\leq \tilde d(n)-4b_{XY}$. %\xh{Can we introduce and use the relatively standard notation $(x)_+$ for $\max(x, 0)$?}
\end{enumerate}
\end{lemma}
\begin{proof}
    We first show that $\abs{\overline{X}}, \abs{\overline{Y}}\leq \overline{M}$.
    %\xh{I don't like that we have this numbered equation being referred to outside the context window.}
    Note that for every $(x,y)\in X\times \overline{Y}$ with $\chi(xy)= c^*$, the triangle $v^*xy$ is $2$-precyclic.
    Therefore we have 
    \[\pc_2(U, X, \overline{Y})\geq \abs{X}\abs{\overline{Y}}-\sum_{y\in \overline{Y}}b_X(y)\geq \abs{X}\abs{\overline{Y}}-\frac{\abs{Y}\pc_2(U, X, \overline{Y})+\lfloor\frac{P}{2}\rfloor\abs{U}\abs{\overline{Y}}}{\abs{U}\abs{Y}},\]
    where we used \cref{lem:sum-Xy-bound} and the fact that $b_{XY}+p\leq \lfloor\frac{P}{2}\rfloor$. (Indeed, the quantities $b_{XY} = \pc_2(v^*, X, Y)$, $p = \pc_2(U\setminus\{v^*\}, X, Y),$ and $\pc_2(U, X, \overline Y)$ count disjoint sets of $2$-precyclic triangles.)
    Rearranging, we get
    \[\pc_2(U, X, \overline{Y})\left(1+\frac{1}{\abs{U}}\right)\geq \abs{\overline{Y}}\left(\abs{X}-\frac{\floor{P/2}}{\abs{Y}}\right)\geq (1-\eta)\abs{X}\abs{\overline{Y}}\geq (1-\eta)\left(\frac{n}{4}-\frac{\Delta}{2}\right)\abs{\overline{Y}},\]
    where we used the fact that $(\abs{X},\abs{Y})$ is an admissible pair (parts (a) and (b) from \cref{def:admissible-pair}) to obtain
    \[\abs{X}\geq \frac{\abs{X}+\abs{Y}-\Delta}{2}\geq \frac{n}{4}-\frac{\Delta}{2}.\]
    This gives 
    \[\abs{\overline{Y}}\leq \left\lfloor \frac{\lfloor\frac{P}{2}\rfloor-b_{XY}}{(1-\eta)\left(\frac{n}{4}-\frac{\Delta}{2}\right)}\left(1+\frac{1}{\abs{U}}\right)\right\rfloor.\]
    To show that $\abs{\overline{Y}}\leq \overline{M}$, it remains to show that $\abs{\overline{Y}}\leq \abs{X^*}+\abs{Y^*}-\lceil \frac{n}{2}\rceil$.
    This is clear as
    \[\left\lceil\frac{n}{2}\right\rceil \leq \abs{X}+\abs{Y}\leq \abs{X^*}+\abs{Y^*}-\abs{\overline{Y}}.\]
    Rearranging, we get the desired inequality, which gives $\abs{\overline{Y}}\leq \overline{M}$.
    An analogous argument gives $\abs{\overline{X}}\leq \overline{M}$.

    Since $\abs{\overline{X}},\abs{\overline{Y}}\leq \overline{M}$, we have
    \[\Delta\geq \abs{\abs{X}-\abs{Y}}\geq \abs{\abs{X^*}-\abs{Y^*}}-\overline{M},\]
    so $\abs{\abs{X^*}-\abs{Y^*}}\leq \overline{M}+\Delta,$ proving (a).

    Next, let us show (b).
    To show $\abs{\overline{X}}\leq \overline{X}_{\max}$, note that
    \[-\Delta\leq \abs{X}-\abs{Y}\leq (\abs{X^*}-\abs{\overline{X}})-(\abs{Y^*}-\overline{M}),\]
    and rearranging gives
    \[\abs{\overline{X}}\leq \overline{M}+\abs{X^*}-\abs{Y^*}+\Delta.\]
    This shows $\abs{\overline{X}}\leq \overline{X}_{\max}$.
    Similarly $\abs{\overline{Y}}\leq \overline{Y}_{\max}$.
    Since $X$ and $\overline{X}$ partition $X^*$, we immediately obtain
    \[\abs{X} = \abs{X^*}-\abs{\overline{X}}\geq \abs{X^*}-\overline{X}_{\max} = X_{\min}\]
    and similarly $\abs{Y}\geq Y_{\min}$.

    Next, for part
    (c), from the definition of $\eta$ we observe that 
    \[\left(\abs{X}-\frac{b_{XY}}{\abs{Y}}-\abs{U}\right)_+\geq \left(\left(1-\frac{\floor{P/2}}{\abs{X}\abs{Y}}\right)\abs{X}-\abs{U}\right)_+\geq \left((1-\eta)\abs{X}-\abs{U}\right)_+\]
    and similarly
    \[\left(\abs{Y}-\frac{b_{XY}}{\abs{x}}-\abs{U}\right)_+\geq \left((1-\eta)\abs{Y}-\abs{U}\right)_+.\]
    Combined with $\abs{X}\geq X_{\min}$ and $\abs{Y}\geq Y_{\min}$, part (c) follows directly from \cref{lemma:u-lower-bound}.

    To prove (d), we use a similar idea as in the proof of \cref{lemma:u-lower-bound}. %\xh{I am confused on where the $b_{XY}$ in this inequality comes from, I understand why it was $4\pc_2 + 2\pc_3$ in the analogous bound Lemma 4.13. Can we deduce (d) as a direct corollary of 4.13?} \ra{Well, $b_{XY} \leq \pc_2$ since if $(x, y)\in X\times Y$ has $\chi(xy)\neq c$, then $v^*xy$ is $2$-precyclic. So this inequality is weaker but still useful for our purposes. And I checked and couldn't make this follow from 4.13.} \xh{Ok, can we spell out the moment we use 4.1(b) then?}
    For any $x\in X$, we have $\delta_{c^*}(x)\geq (\abs{U}-b_U(x)-\abs{Y^*})_+$.
    Therefore
    \[\sum_{x\in X}\delta_{c^*}(x)^2\geq \sum_{x\in X}(\abs{U}-b_U(x)-\abs{Y^*})_+^2.\]
    Using Jensen's inequality and \cref{lemma:Second-step}, the above is at least 
    \[\abs{X}\left(\abs{U}-\frac{\abs{U}b_{XY}+p}{\abs{X}\abs{Y}}-\abs{Y^*}\right)_+^2.\]
    Since
    \[\frac{\abs{U}b_{XY}+p}{\abs{X}\abs{Y}}\leq \frac{\floor{P/2}}{\abs{X}\abs{Y}}\abs{U}\leq \eta\abs{U},\]
    we know that 
    \[\abs{X}\left(\abs{U}-\frac{\abs{U}b_{XY}+p}{\abs{X}\abs{Y}}-\abs{Y^*}\right)_+^2\geq X_{\min}\left((1-\eta)\abs{U}-\abs{Y^*}\right)_+^2.\]
    Analogous inequalities hold with $X$ and $Y$ flipped. Combining this with \cref{lemma:basic-facts}(b) and the fact that $b_{XY} \leq \pc_2$, we have
    \begin{align*}\tilde d(n)-4b_{XY}&\geq \sum_{x\in X}\delta_{c^*}(x)^2+\sum_{y\in Y}\delta_{c^*}(y)^2 \\
    &\geq X_{\min}\left((1-\eta)\abs{U}-\abs{Y^*}\right)_+^2+Y_{\min}\left((1-\eta)\abs{U}-\abs{X^*}\right)_+^2,\end{align*}
    proving (d).
\end{proof}

%\xh{I really don't like the notations $\textup{Coeff}_{b_{XY}}$ and $b_{\max from x}$. For the former, it is common in generatingfunctionology to use square brackets to denote coefficients, perhaps we could write $C[b_{XY}]$ or $\textup{Coeff}[b_{XY}]$? A strict upgrade to the latter would just be to write $\max_X b_U$, it is then transparent what this is. }\hy{Yes let's do $\max_Xb_U$ for $b_{\max from x}$. For $\textup{Coeff}_{b_XY}$ I guess it is slightly awakard that we also have $\textup{Coeff}_{\lfloor P/2\rfloor -b_{XY}}$ so the standard bracket wouldn't be technically correct. We should check but if we haven't used $C$ for a while I think $C[b_{XY}]$ and $C[\lfloor P/2\rfloor -b_{XY}]$ can be good, with the only caveat being that $C[\lfloor P/2\rfloor -b_{XY}]$ will still be really long to fit :(.}

To state the next lemma concisely, we introduce three more notations. Define: \begin{align}\label{eq:coeffs and tau} C[b_{XY}] &:= \left(1+\frac{\overline{X}_{\max}}{X_{\min}}+\frac{\overline{Y}_{\max}}{Y_{\min}}+\frac{\abs{U}(\abs{X^*}+\abs{Y^*})}{X_{\min}Y_{\min}}\right);\nonumber\\
C[\floor{P/2}-b_{XY}]&:=\left(1+\frac{\overline{X}_{\max}}{X_{\min}}+\frac{\overline{Y}_{\max}}{Y_{\min}}\right)\left(\frac{1}{X_{\min}}+\frac{1}{Y_{\min}}+\frac{1}{\abs{U}}\right);\nonumber\\
\tau &:= \left\lfloor C[b_{XY}] \cdot b_{XY}+C[\floor{P/2}-b_{XY}]\cdot \left(\floor{P/2}-b_{XY}\right)+\overline{X}_{\max}\overline{Y}_{\max}\right\rfloor.\end{align}

The final lemma gives a concrete upper bound on the number of edges not colored $c^*$ across the tripartition of color $c^*$. This bound depends on $b_{XY}$, so the second part of the lemma serves to bound $b_{XY}$.

\begin{lemma}
\label{lem:tau-bad}
Suppose the assumptions of \cref{lemma:possible-X-Y}, \cref{lemma:Second-step}, and \cref{lemma:Third-step} hold. Then the number of edges across $U, X^*, Y^*$ not colored $c^*$ is at most $\tau$. Furthermore, with $\max_{X^*} b_U := \max_{x\in X^*}b_U(x)$ and similarly for $\max_{Y^*} b_U$, we have %$b_{\textup{max from }x}$ and $b_{\textup{max from }y}$ be integers such that $b_{U}(x)\leq b_{\textup{max from }x}$ for every $x\in X^*$ and $b_{U}(y)\leq b_{\textup{max from }y}$ for every $y\in Y^*$. Then, we have
\begin{align} (\abs{U}-\max_{X^*}b_U-\max_{Y^*}b_U)b_{XY}&\leq \floor{P/2}\label{eq:bounding bad 1}\end{align}
and
\begin{equation} \label{eq:bounding bad 2}(\abs{U}-\max_{X^*}b_U-\max_{Y^*}b_U)b_{XY} \leq \binom{b_{XY}}{2}+b_{XY}\left(\overline{X}_{\max}+\overline{Y}_{\max}+\min(\max_{X^*}b_U,\max_{Y^*}b_U)\right).\end{equation}
%\\
%(\abs{U}-b_{\textup{max from }x}-b_{\textup{max from }y})b_{XY} &\leq \binom{b_{XY}}{2}+b_{XY}\left(\overline{X}_{\max}+\overline{Y}_{\max}+\min(b_{\textup{max from }x},b_{\textup{max from }y})\right).\label{eq:bounding bad 2}\end{align} \ra{Save this for next section}
\end{lemma}
\begin{proof}
The number of edges not colored $c^*$ across $X^*,Y^*,U$ is at most
    \begin{align}\label{eq:edges not c^*} b_{XY}+\sum_{u\in U}\left(b_X(u)+b_Y(u)\right)+\sum_{x\in \overline{X}}\left(b_U(x)+b_Y(u)\right)+\sum_{y\in \overline{Y}}\left(b_U(y)+b_X(y)\right)+\abs{\overline{X}}\abs{\overline{Y}}.\end{align}
    Using \cref{lemma:Second-step}, the first sum is at most
    \[\left(\abs{U}b_{XY}+p\right)\left(\frac{1}{\abs{X}}+\frac{1}{\abs{Y}}\right).\]
    
    For the last sum, we apply \cref{lem:sum-Xy-bound} to obtain an upper bound on each of $\sum_{y\in \overline{Y}}b_X(y)$ and $\sum_{y\in \overline{Y}}b_U(y)$.
    In total, the last sum is at most 
    \[\left(\pc_2(U, X, \overline{Y})+\frac{(\abs{U}b_{XY}+p)\abs{\overline{Y}}}{\abs{Y}}\right)\left(\frac{1}{\abs{X}}+\frac{1}{\abs{U}}\right).\]
    Symmetrically, the middle sum in \eqref{eq:edges not c^*} is at most
    \[\left(\pc_2(U, \overline{X}, Y)+\frac{(\abs{U}b_{XY}+p)\abs{\overline{X}}}{\abs{X}}\right)\left(\frac{1}{\abs{Y}}+\frac{1}{\abs{U}}\right).\]
    Since there are at most $\floor{P/2}$ $2$-precyclic triangles and $b_{XY}, p, \pc_2(U, \overline{X}, Y),$ and $\pc_2(U, X, \overline{Y})$ all count disjoint sets of $2$-precyclic triangles, we have
    \[b_{XY}+p+\pc_2(U, \overline{X}, Y)+\pc_2(U, X, \overline{Y})\leq \floor{P/2}.\]
    Therefore
    \[p+\pc_2(U, \overline{X}, Y)+\pc_2(U, X, \overline{Y})\leq \floor{P/2}-b_{XY}.\]
    Let us now compute the coefficient of $b_{XY}$ in the sum of the five terms in \eqref{eq:edges not c^*}.
    It is
    \begin{align*} &1+\abs{U}\left(\frac{1}{\abs{X}}+\frac{1}{\abs{Y}}+\frac{\abs{\overline{X}}}{\abs{X}}\left(\frac{1}{\abs{Y}}+\frac{1}{\abs{U}}\right)+\frac{\abs{\overline{Y}}}{\abs{Y}}\left(\frac{1}{\abs{X}}+\frac{1}{\abs{U}}\right)\right)\\
    &=1+\frac{\abs{\overline{X}}}{\abs{X}}+\frac{\abs{\overline{Y}}}{\abs{Y}}+\frac{\abs{U}(\abs{X}+\abs{\overline{X}}+\abs{Y}+\abs{\overline{Y}})}{\abs{X}\abs{Y}}\leq C[b_{XY}].\end{align*}
    We also observe that the coefficients of $p,\pc_2(U, \overline{X}, Y)$ and $\pc_2(U, X, \overline{Y})$ are all at most $C[\floor{P/2}-b_{XY}]$.
    Therefore the number of edges across the parts not colored $c^*$ is at most
    \[\left\lfloor C[b_{XY}]\cdot b_{XY}+C[\floor{P/2}-b_{XY}]\cdot(\floor{P/2}-b_{XY})+\abs{\overline{X}}\abs{\overline{Y}}\right\rfloor\leq \tau.\]

    We now prove \eqref{eq:bounding bad 1}.
    For any edge $(x,y)\in X\times Y$ that is not of color $c^*$, we see that there are at least $\abs{U}-b_U(x)-b_U(y)$ choices of $u$ so that $\chi(ux)=\chi(uy)=c^*$.
    For each of these choices, $uxy$ is a $2$-precyclic triangle. Since $b_U(x)\leq \max_{X^*}b_U$ and $b_U(y) \leq \max_{Y^*}b_U$, this proves \eqref{eq:bounding bad 1}.
    
    To prove \eqref{eq:bounding bad 2}, 
    we claim that if it doesn't hold, then recoloring all the $b_{XY}$ edges with color $c^*$ cannot decrease the number of monochromatic triangles, which contradicts that $\chi$ is the coloring in $\cG_{\Delta, P}$ with the fewest edges not colored $c^*$ (see \cref{def:v*c*}). Indeed, the left hand side of \eqref{eq:bounding bad 2} is the minimum number of monochromatic triangles we create.
    To bound the number of monochromatic triangles we destroy, note that there are trivially at most $\binom{b_{XY}}{2}$ such triangles that use two edges from the $b_{XY}$ edges. To use just one $b_{XY}$ edge, the third vertex of the triangle cannot be in $X$ or $Y$ because this triangle would have at least one edge colored $c^*$ and at least one not colored $c^*$. So the third vertex can be in $\overline{X}$, in $\overline{Y}$, or in $U$, but there are at most $\min(\max_{X^*}b_U,\max_{Y^*}b_U)$ vertices in $U$ that could create a monochromatic triangle with an edge not colored $c^*$.
\end{proof}

This last lemma finally gives us a concrete upper bound on the number of edges not colored $c^*$ across $X^*, Y^*,$ and $U$. We use it to finish the proof.

\begin{proof}[Proof of \cref{thm:monochromatic-k-cliques} for $k=3$] Recall that by \cref{cor:calGempty} and \cref{lemma: recoloring}, to contradict \cref{asm:contradiction}, it suffices to show that whenever $\Delta, P$ satisfy the conditions of \cref{cor:calGempty} and $\chi\in \cG_{\Delta, P}$, there are at most $2\min(|X^*|, |Y^*|, |U|)$ edges not colored $c^*$ across the parts $X^*, Y^*$, and $U$. By \cref{lem:tau-bad}, it thus suffices to verify that the assumptions of \cref{lemma:possible-X-Y}, \cref{lemma:Second-step}, and \cref{lemma:Third-step} all hold and that $\tau\leq 2\min(|X^*|, |Y^*|, |U|)$. %\xh{This needs to be more self-contained. Maybe something like "Assuming Asm 3.1, and using [set of all lemmas used], it suffices to check the assumptions of [set of all starred lemmas] and to demonstrate that $\tau \le $ [put exact bound needed] to obtain that $c^*$ is spanning complete tripartite, by Lemma 4.5. This is a contradiction because blah blah induction so we are done.} By \cref{lem:tau-bad} and \cref{lemma: recoloring}, it remains to demonstrate that the quantity $\tau$ is small. 
This is indeed the case in all allowable configurations for $n\not\in\{13, 14, 16, 17\}$ (while these four values of $n$ were handled in \cref{prop: small n}). We describe how to verify this in Appendix \ref{sec:verifying assumptions}, using the lemmas in this section together with a computer-assisted calculation.
\end{proof}

\section{Higher Uniformity}\label{sec:higher uniformity}
To prove \cref{thm:monochromatic-k-cliques} for $k\geq 7$, we use an elementary approach involving convexity. Specifically, the Loomis-Whitney inequality gives an upper bound on the number of $k$-cliques in a $k$-partite graph with a prescribed number of edges. This leads to a lower bound on the number of edges of the most common color required to achieve at least $g_k(n)$ monochromatic copies of $K_k$. Then, we show that once we already have many edges of the most common color, recoloring the remaining edges crossing that color's $k$-partition can only improve the coloring. By induction, this proves the extremal coloring is iterated $k$-partite.

Note that \cref{thm:monochromatic-k-cliques} is obvious for $n < k$ since there are no copies of $K_k$ in $K_n$. So, for the base case for the induction, we show that the theorem holds for all $k\geq 3$ and $n\in[k, k(k-1)]$.

\begin{lemma}\label{lem:k to k(k-1)}
    For any integers $k\geq 3$ and $n\in[k,k(k-1)]$, any $k$-partite coloring of $K_n$ has at most $g_k(n)$ monochromatic copies of $K_k$.
\end{lemma}
\begin{proof}
Recall that $g_k(n)$ is the number of edges in the $n$-vertex balanced iterated blowup of an edge. For $n \leq k(k-1)$, this is simply the balanced complete $k$-partite $k$-graph, and hence 
    \[g_k(n) = \left(\frac{n}{k}-p_{n,k}\right)^{k(1-p_{n,k})}\left(\frac{n}{k}+1-p_{n,k}\right)^{kp_{n,k}},\]
    where $p_{n,k} = \frac{n}{k}-\lfloor\frac{n}{k}\rfloor$.
    %This is the maximum of $n_1\cdots n_k$ subject to $n_1+\cdots +n_k=n$ and $n_i$ are nonnegative integers.
We proceed by induction on $n$. The statement is obvious when $n=k$; we have $g_n(k)=1$ and we achieve one copy of $K_k$ in $K_n$ by coloring all edges the same color.
    Now suppose that the statement holds for $n-1$ for some $k+1\leq n\leq k(k-1)$, and suppose for the sake of contradiction that there is a $k$-partite coloring of $K_n$ with more than $g_k(n)$ monochromatic copies of $K_k$.
    We know that each vertex $v$ is in contained in at least $g_k(n)-g_k(n-1)+1$ monochromatic copies of $K_k$, or else by removing $v$ we would have a coloring of $K_{n-1}$ with at least $g_k(n-1)+1$ monochromatic copies of $K_k$, violating the induction hypothesis.
    
    We can think of the balanced $k$-partite $k$-graph on $n$ vertices as being obtained from that on $n-1$ vertices by adding a vertex to a smallest part. We thus see that
    \[g_k(n)-g_k(n-1) = g_{k-1}\left(n-\left\lceil\frac{n}{k}\right\rceil\right).\]
    Fix an arbitrary vertex $v$. For each color $c$, let $n_c$ be the number of edges incident to $v$ of color $c$.
    The number of monochromatic copies of $K_k$ containing $v$ is at most
    \[\sum_{c}g_{k-1}(n_c)\]
    by the $k$-partiteness of each color.
    Since we also have
    \[g_{k-1}(t)-g_{k-1}(t-1) = g_{k-2}\left(t-\left\lceil\frac{t}{k-1}\right\rceil\right),\]
    we know that $g_{k-1}(t)$ is convex.
    Now if $\max_c n_c\leq n-k+1$, then by convexity we know that
    \[\sum_{c}g_{k-1}(n_c)\leq g_{k-1}(n-k+1)\]
    as $(n-1)-(n-k+1)=k-2<k-1$.
    Combining this with the lower bound on the number of monochromatic copies of $K_k$ containing $v$, we have
    \[g_{k-1}(n-k+1)\geq g_{k-1}\left(n-\left\lceil\frac{n}{k}\right\rceil\right)+1,\]
    which shows that
    \[n-k+1\geq n-\left\lceil \frac{n}{k}\right\rceil+1,\]
    or equivalently $\lceil \frac{n}{k}\rceil \geq k$.
    This is a contradiction to $n\leq k(k-1)$, which shows that there must be some color $c$ with $n_c\geq n-k+2$.
    As an immediate corollary, any other $c'$ must have $n_{c'}\leq k-2$, and so for each vertex $v$, all monochromatic copies of $K_k$ containing $v$ must have the same color.
    Call this the \emph{dominating color} of $v$.

    Now group the vertices based on their dominating color.
    Let $n_1,n_2,\ldots, n_s$ be the number of vertices with their dominating color being $c_1,\ldots, c_s$.
    Then there can be at most
    $\sum_{i=1}^{s}g_k(n_i)$
    monochromatic copies of $K_k$ as no monochromatic $K_k$ can cross the groups, and the monochromatic $K_k$'s within a group must all have the same color.
    Using convexity once again, this is maximized when $s=1$ and $n_1=n$, which contradicts the assumption that there are more than $g_k(n)$ monochromatic $K_k$'s.
\end{proof}

We next state the lemma which completes the proof by induction for large $k$. 
\begin{lemma}\label{lem:big n k}
    Let $k\geq 4$ and $n\geq k(k-1)+1$. Suppose
    \begin{equation}\label{eq:alphabeta}
    \begin{split}
        \alpha &= 1 - \left(1-\frac{1}{k}\right)^{\frac{k}{k-2}}\left(1-\frac{k^3}{2n^2}\right)^{\frac{2}{k-2}}, \\
        \beta &= \frac{k^3}{2n^2} + \left(\frac{\alpha k}{k-1}\right)^{k/2}
    \end{split}
    \end{equation}
    are both in $(0,1)$, $\gamma>0$ is the unique positive solution to $e^{-2\gamma}(1+\gamma)^2 = 1-\beta$, and
    \begin{align}\label{eq:gammarequirement}
    \alpha^{\frac{k-2}{2}} \left(\frac{k}{k-3}\right)^{k-2} < e^{-2\gamma}.
    \end{align}
    Then, in any $k$-partite coloring of $K_n$ which maximizes the number of monochromatic copies of $K_k$, the most commonly occurring color must form a complete spanning $k$-partite subgraph.
\end{lemma}
The main theorem follows quickly from this lemma and the induction hypothesis.
\begin{proof}[Proof of \cref{thm:monochromatic-k-cliques} for $k\geq 7$.]
By induction, we may assume the theorem holds for smaller values of $n$. If the conclusion of \cref{lem:big n k} holds, then the number of copies of $K_k$ in the most common color $c$ is $\prod_{j=1}^k s_j$, where $s_j$ are the part sizes of color class $c$. All remaining monochromatic copies of $K_k$ must lie within those parts, so by the induction hypothesis, there are at most $\sum_{j=1}^k g_k(s_j)$ monochromatic copies of $K_k$ not of color $c$, for a total of $\prod_{j=1}^k s_j + \sum_{j=1}^k g_k(s_j) \leq g_k(n)$ monochromatic copies of $K_k$.

Since \cref{lem:k to k(k-1)} proved the theorem for $n\leq k(k-1)$, to complete the proof by induction it suffices to show that any $n\geq k(k-1)+1$ satisfies the conditions of \cref{lem:big n k}. In fact, we claim that it suffices to check only $n=k(k-1)+1$. To see this, note that for $k$ fixed, if $n$ increases, then $\alpha$ decreases (and stays in $(0,1)$), so also $\beta$ decreases (and stays in $(0,1)$). Then $\gamma$ also decreases; indeed, the function $f(\gamma) = e^{-2\gamma}(1+\gamma)^2$ is decreasing for $\gamma>0$, and $f(0)=1$. Thus, if $\alpha^{(k-2)/2} \left(\frac{k}{k-3}\right)^{k-2} < e^{-2\gamma}$ holds for $n=k(k-1)+1$, it must also hold for all larger $n$ as well since the left hand side decreases and the right hand side increases.

Substituting $n=k(k-1)+1$, we obtain 
\begin{align*}
    \alpha &= 1-\left(1-\frac{1}{k}\right)^{\frac{k}{k-2}}\left(1-\frac{k^3}{2(k(k-1)+1)^2}\right)^{\frac{2}{k-2}}, \\
    \beta &= \frac{k^3}{2(k(k-1)+1)^2} + \left(\frac{\alpha k}{k-1}\right)^{k/2}.
\end{align*}
We claim that if $k \geq 7$ increases, both $\alpha$ and $\beta$ decrease. First note that $1-1/k$ increases and $k/(k-2)$ decreases, so $(1-1/k)^{k/(k-2)}$ increases. Then, we compute
$$\frac{d}{dk}\left(\frac{k^3}{2(k(k-1)+1)^2}\right) = \frac{-k^6 + 3k^4 - 4k^3 + 3k^2}{2(k(k-1)+1)^4},$$ which is clearly negative when $k\geq 7$. Thus, $1-\frac{k^3}{2(k(k-1)+1)^2}$ increases; also, $2/(k-2)$ is decreasing, so this concludes the proof that %$(1-1/k)^{k/(k-2)}$ and $(1-k^3/(2(k(k-1)+1)^2))^{2/(k-2)}$
$\alpha$ decreases. For $\beta$, the derivative computation above shows that the first term is decreasing. Since $\alpha$ decreases, the factor $\alpha^{k/2}$ in the second term decreases, and as it is a standard fact that $(1+1/x)^{x+1}$ decreases to $e$ for positive $x$, we have that $(k/(k-1))^{k/2}$ decreases as well, proving the claim. %\ra{Xiaoyu - to what extent should we justify these? Hans's comments suggest some justification might be helpful. If full rigor is desired I can write out a couple sentences.} \hy{I think we need to be a bit more careful here. The way I see this is that we first factor out $1-1/k$, which is increasing. The rest is something like $(\cdots)^{2/(k-2)}$, where the $\cdots$ is increasing but smaller than $1$ and $2/(k-2)$ is decreasing, and so the entire thing is increasing as well. } %\xh{These inline formulas are egregious, either reformat them to fit better (e.g. $(k/(k-1))^{k/2}$ and don't use left right on inline parens) or put them on new lines} 
%Also, when $k$ increases, %$\frac{k^3}{2(k(k-1)+1)^2}$ decreases, as do $\alpha^{k/2}$ and $\left(\frac{k}{k-1}\right)^{k/2}$
%both terms in the expression for $\beta$ decrease \hy{This time I don't see why this is true. The first term is believable but probably requires some explicit computation. The second term comes terribly close as we are basically asking whether $\alpha^k(1+1/(k-1))^{k}$ is decreasing in $k$ or not, whereas the expression $(1+1/(k-1))^{k-1}$ is increasing in $k$. I don't know if I believe in the claim actually.}; 

With the claim proved, note that since $\beta$ decreases, $\gamma$ decreases as well. %, and both $\alpha$ and $\beta$ remain in $(0,1)$. 
In \eqref{eq:gammarequirement}, the right hand side then increases, and $\alpha^{(k-2)/2}$ decreases. Although the factor $(k/(k-3))^{k-2}$ is increasing in $k\geq 5$, it is at most $e^3$ for all $k\geq5$. So, it suffices for the stronger inequality $\alpha^{(k-2)/2}e^3 < e^{-2\gamma}$ to hold; if it does hold for some $k$, then it holds for all larger $k$ as well.

Thus, to prove the result, we just need to demonstrate that the conditions hold for $k = 7$ and $n=k(k-1)+1=43$. Solving numerically (and rounding to 4 decimal places), we obtain $\alpha =0.2249, \beta=0.102,\gamma =0.3648,$ and for the final inequality, $\alpha^{(k-2)/2} e^3 =0.4817$ while $e^{-2\gamma} =0.4821$, meeting the conditions of the lemma.
\end{proof}

It remains to prove \cref{lem:big n k}, for which we need two helper lemmas.
The first is a short computation that helps control constants later on.

\begin{lemma}\label{lem:computation}
    For any $k\geq 3$ and $n\ge k(k-1)+1$, we have $g_k(n) \ge \left(1-\frac{k^3}{2n^2}\right)\frac{n^k}{k^k}$.
\end{lemma}
\begin{proof}
    We have that \[
    g_k(n) \ge \left(\frac{n}{k} - p_{n,k}\right)^{k(1-p_{n,k})}\left(\frac{n}{k} + 1 - p_{n,k}\right)^{kp_{n,k}}
    \]
    where $p_{n,k} = \frac{n}{k} - \floor{\frac{n}{k}}$. Let the right hand product be $R$ and consider the ``pushing" operation where we replace two factors $(n/k-x)(n/k+y)$ in the product with $(n/k)(n/k+y-x)$. The value of $R$ increases by $xy \le \frac{1}{4}$ times the product of the remaining terms, which one can use the AM-GM inequality to check that it is at most $((n+1)/k)^{k-2} \le 2\frac{n^{k-2}}{k^{k-2}}$. (For the last inequality, note that $((n+1)/n)^{k-2} \leq e^{(k-2)/n} \leq 2$ since $n$ is large enough.) After at most $k-1$ pushing operations, all terms are equal, so we obtain that the original $R$ satisfies
    \[
    R \ge \frac{n^k}{k^k} - \frac{k}{4} \cdot 2 \frac{n^{k-2}}{k^{k-2}} = \left(1-\frac{k^3}{2n^2}\right)\frac{n^k}{k^k}
    \]
    as desired.
\end{proof}
We now use the Loomis-Whitney inequality to lower bound the number of edges and copies of $K_k$ in the most common color.
\begin{lemma}\label{lem:LW ineq time}
    For $k\geq 3$ and $n\ge k(k-1)+1$, in an edge coloring of $K_n$ such that all color classes are $k$-partite which maximizes the number of monochromatic copies of $K_k$, the most common color must have at least $(1-\alpha)\binom{n}{2}$ edges and at least $(1-\beta)\frac{n^k}{k^k}$ monochromatic copies of $K_k$, where $\alpha$ and $\beta$ are given by \eqref{eq:alphabeta}. %\hy{Do we need to write the definition out again?}
    % \begin{align*}
    %     1-\alpha &= \left(1-\frac{1}{k}\right)^{\frac{k}{k-2}}\left(1-\frac{k^3}{2n^2}\right)^{\frac{2}{k-2}}, \\
    %     \beta &= \frac{k^3}{2n^2} + \left(\frac{\alpha k}{k-1}\right)^{k/2}.
    % \end{align*}
\end{lemma}
\begin{proof}
    Let $e_i$ be the number of edges of color $i$, and let $t_i$ be the number of copies of $K_k$ in color $i$. Each color class is $k$-partite, so by the Loomis-Whitney inequality (see e.g. \cite[Corollary 15.7.5]{ProbMethod}), we have 
    \begin{equation}\label{eq:LW bound}
    t_i \le \left(\frac{e_i}{\binom{k}{2}}\right)^{k/2}.
    \end{equation}
    %\ra{How standard is using Shearer's lemma like this? Should we state it as a separate lemma and elaborate on how it's being applied here?}\hy{I feel like we should be more specific here. Also this is really just Loomis--Whitney with convexity rather than Shearer's.} \xh{I agree that this is just Loomis-Whitney, we should say Loomis-Whitney everywhere.}
    
    On the other hand, the maximizing coloring must have at least $g_k(n) \ge \left(1-\frac{k^3}{2n^2}\right)\frac{n^k}{k^k}$ total monochromatic copies of $K_k$ (here we applied \cref{lem:computation}), so we obtain
    \[
    \sum_i \left(\frac{e_i}{\binom{k}{2}}\right)^{k/2} \ge \left(1-\frac{k^3}{2n^2}\right)\frac{n^k}{k^k}.
    \]
    We have also that $\sum e_i = \binom{n}{2}$. Letting $m = \max e_i$, we obtain
    \[
    \frac{m^{k/2-1}}{\binom{k}{2}^{k/2}} \binom{n}{2}= \frac{m^{k/2-1}}{\binom{k}{2}^{k/2}} \sum_i e_i \ge \sum_i \left(\frac{e_i}{\binom{k}{2}}\right)^{k/2} \ge \left(1-\frac{k^3}{2n^2}\right)\frac{n^k}{k^k}.
    \]
    Solving for $m$ and using the bounds $\binom{n}{2} \leq n^2/2$ and $\binom{k}{2} \geq (k-1)^2/2$, we obtain
    \[
    m \ge \left(1-\frac{1}{k}\right)^{\frac{k}{k-2}}\left(1-\frac{k^3}{2n^2}\right)^{\frac{2}{k-2}} \frac{n^2}{2} \ge (1-\alpha)\binom{n}{2}.
    \]
    This proves the first half of the lemma. For the second half, observe that the number of edges not of the most common color must then be at most $\alpha \binom{n}{2}$. Using the bound \eqref{eq:LW bound}, the number of monochromatic copies of $K_k$ not in the most common color is at most $\sum\left(e_i/\binom{k}{2}\right)^{k/2}$, where the sum is over colors other than the most common one. By convexity, this sum is maximized by letting one $e_i$ equal $\alpha\binom{n}{2}$ and the others equal $0$, so we obtain that there are at most
    \[
    \left(\frac{\alpha \binom{n}{2}}{\binom{k}{2}}\right)^{k/2} \le \left(\frac{\alpha k }{k-1}\right)^{k/2} \frac{n^k}{k^k}
    \]
    monochromatic copies of $K_k$ that are not of the most common color. Subtracting from 
    \[
    g_k(n) \ge \left(1-\frac{k^3}{2n^2}\right)\frac{n^k}{k^k}
    \]
    gives a lower bound on the number of $K_k$'s in the most common color, proving the second part.
\end{proof}
Finally, we use this result to prove \cref{lem:big n k}.
\begin{proof}[Proof of \cref{lem:big n k}]
 Recolor every edge crossing the $k$-partition of the most common color $c^*$ by $c^*$. To finish the proof, we show that any recolored edge $uv$ ends up in more monochromatic copies of $K_k$ than it starts with.

    The number $A$ of monochromatic copies of $K_k$ that $uv$ lies in before recoloring is at most the number of monochromatic copies of $K_{k-2}$ in all the colors except $c^*$. For $k\ge 4$, \cref{lem:LW ineq time} together with the Loomis-Whitney inequality gives
    \[
    A \le \left(\frac{\alpha\binom{n}{2}}{\binom{k-2}{2}}\right)^{(k-2)/2} \le \alpha^{(k-2)/2} \left(\frac{k}{k-3}\right)^{k-2} \frac{n^{k-2}}{k^{k-2}}.
    \]

    %\hy{I think the guiding text in this paragraph is wrong. We want to estimate $\gamma_1$ and $\gamma_2$, so we upper bound the maximum possible number of monochromatic copies \emph{before} recoloring. By Lemma 5.4 it has to be somewhat large, showing that $\gamma_1$ and $\gamma_2$ have to be somewhat small. In particular, it has nothing to do with $B$.}
    Suppose $u,v$ lie in two parts $P_1, P_2$ of size $(1+\gamma_1)\frac{n}{k}, (1+\gamma_2)\frac{n}{k}$ of the $k$-partition of $c^*$.
    We show that $P_1$ and $P_2$ cannot be far larger than average.
    Indeed, the total number of copies of $K_k$ colored $c^*$ before recoloring is at most
    \[
    |P_1||P_2|\left(\frac{n-|P_1|-|P_2|}{k-2}\right)^{k-2}=(1+\gamma_1)(1+\gamma_2)\left(1-\frac{\gamma_1 + \gamma_2}{k-2}\right)^{k-2}\frac{n^k}{k^k}\le (1+\bar{\gamma})^2e^{-2\bar{\gamma}} \frac{n^k}{k^k},
    \]
    where $\bar{\gamma} = \frac{\gamma_1+\gamma_2}{2}$. By \cref{lem:LW ineq time},  this quantity must be at least $(1-\beta)\frac{n^k}{k^k}$, so we obtain
    \[(1+\bar{\gamma})^2e^{-2\bar{\gamma}} \ge 1-\beta.
    \]
    Since the left hand side is a decreasing function in $\bar{\gamma}>0$, it follows that $\bar{\gamma} \le \gamma$ as defined in the lemma statement.

    Now we lower bound $B$, the number of monochromatic copies of $K_k$ containing $uv$ after recoloring. Observe that $B$ is exactly the number of $K_{k-2}$'s crossing the $k-2$ parts other than $P_1$ and $P_2$. Therefore, we obtain
    \[
    B(1+\gamma)^2\frac{n^2}{k^2} \ge B(1+\bar{\gamma})^2\frac{n^2}{k^2} \ge B|P_1||P_2| \ge (1-\beta)\frac{n^k}{k^k}
    \]
    as the right hand side is a lower bound for the number of copies of $K_k$ colored $c^*$ before recoloring. We conclude that
    \[
    B \ge (1+\gamma)^{-2}(1-\beta)\frac{n^{k-2}}{k^{k-2}} = e^{-2\gamma} \frac{n^{k-2}}{k^{k-2}},
    \]
    by our definition of $\gamma$. It suffices to combine our bounds for $A$ and $B$ together with \eqref{eq:gammarequirement} to conclude that $A < B$, as desired.
\end{proof}

%\ra{Decide whether we want a Concluding Remarks section. Things that might go in there: difficulty of proving \cref{thm:iter-k-partite} for $k\in\{4, 5, 6\}$; and the exponential-to-doubly-exponential transition.}\hy{I think we do}

\section{Concluding Remarks}\label{sec:conclusion}
The original conjecture of Erd\H os and Hajnal~\cite{EH72} is simply that $r_k(s, g_k(s)+1; t)$ is exponential in some positive power of $t$. While \cref{thm:main} settles that conjecture, the power of $t$ in the exponent remains in question. The previous partial results by Conlon, Fox, and Sudakov~\cite{CFS} (for $k=3$ and some values of $s$) and Mubayi and Razborov~\cite{MR21} (for $s>k\geq 4$) established $r_k(s, g_k(s)+1; t) = 2^{\Omega(t)}$. We believe that this should hold for $k=3$ and all $s$ as well.
\begin{conj}
    For all $s > 3$, we have $r_3(s, g_3(s)+1; t) = 2^{\Omega(t)}$.
\end{conj}
This exponent of $\Omega(t^1)$ would be best possible.
Our result has $t^{2/3}$ inside the exponent, which is a consequence of its dependence on \cref{thm:blackbox}. Any improvement in the exponent there would yield an improvement in the exponent in \cref{thm:main}.

An adjacent open problem in this direction is to resolve the main conjecture from \cite{CFG+Y}, which can be thought of as a strengthening of our main theorem. %\ra{Check this wording}

\begin{conj}[\cite{CFG+Y}]
For any $3$-graph $H$, there exists a constant $c$ such that $r(H, K_n^3) \leq n^c$ for all $n$ if and only if $H$ is iterated tripartite.
\end{conj}

%\xh{I think we should also restate the main conjecture of https://arxiv.org/pdf/2411.13812 here and state that it is somehow morally a strengthing of our theorem.}

%Our main results, \cref{thm:main} and the corresponding statement for $k\geq 7$, follow quickly from the Tur\'an type result \cref{thm:iter-k-partite}, which in turn follows from the extremal coloring result \cref{thm:monochromatic-k-cliques}. For $k=3$, much of our method is based on counting cyclic triangles in tournaments, which does not extend to capture the recursive structure of iterated $k$-partite graphs for $k > 3$. On the other hand, our techniques for $k\geq 7$ depend on the Loomis-Whitney inequality and convexity, and the bounds we obtain from these arguments become stronger as $k$ grows. We believe that the proof could be pushed down to $k=6$ with more work, but the calculations appear to fall apart for $k \leq 5$. \xh{Motion to cut this paragraph, I think we've said most of it already earlier.}

%Although \cref{conj:main} is now fully proven, 
We believe \cref{thm:iter-k-partite} and the corresponding \cref{thm:monochromatic-k-cliques} to be interesting in their own right, but did not find a simple way to extend them to uniformities $4$ through $6$. %and while our proof for $k\geq 7$ could plausibly extend to $k=6$ with more work, the calculations appear to fall apart for $k\geq 5$. \ra{Worth keeping the stuff after ``and while"?} 
Hence, we pose the following:

\begin{conj}
    For $k\in\{4, 5, 6\}$, any $k$-partite coloring of the complete graph $K_n$ has at most $g_k(n)$ monochromatic copies of $K_k$.
\end{conj}

We remark that the idea to pass from edge colorings to cyclic triangles in tournaments is partially inspired by the recent breakthrough of Balogh and Luo determining the Tur\'an densities of long tight cycles minus an edge \cite{BL23}. %\cite[Claim 5.13]{BL23}. 
Recall that a tight cycle is a $3$-graph with vertex set $[\ell]$ and edge set $\{\{1, 2, 3\}, \{2, 3, 4\}, \ldots, \{\ell-2, \ell-1, \ell\}, \{\ell-1, \ell, 1\}, \{\ell, 1, 2\}\}$. Denote by $C_\ell^-$ the tight cycle minus one edge.
Notice that any $3$-graph with each tight component being tripartite automatically forbids $C_\ell^-$ for any $\ell$ not divisible by $3$, so their results imply \cref{thm:iter-k-partite} holds asymptotically. 
It is possible to directly adapt their techniques to show that the extremal colorings in \cref{thm:iter-k-partite} have spanning monochromatic complete tripartite subgraphs for very large $n$, thus giving an alternative proof of our induction step in this regime. Because of the significant technical difficulties at small $n$, we did not pursue a complete proof along these lines. %\xh{I edited this discussion.} %, while to show \cref{thm:monochromatic-k-cliques} even just for large $n$, we need to prove the structural statement for all values of $n$. %, and it seems difficult to extend Balogh and Luo's argument to all smaller $n$.
%We nonetheless borrow some idea from \cite{BL23}: their structural result comes from a stability result about $C_5^-$-free orientable hypergraphs (\cite[Proposition 4.1]{BL23}) which they prove by examining the corresponding tournament, and we considered some tournaments that can be constructed from the edge coloring. \xh{The above two-paragraph discussion is kind of wishy-washy, shorten it somehow.} 
%\ra{Attempted to shorten above discussion}

As a final remark, we note that Erd\H os and Hajnal \cite{EH72} conjectured more about the growth of $r_k(s, e; t)$: that after the polynomial-to-exponential transition at $e = g_k(t)+1$, there is another transition from exponential to doubly exponential in a power of $t$ at some well-defined $e = h_k^{(2)}(s)$, followed by yet another transition to triply exponential at $e=h_k^{(3)}(s)$, and so forth until finally $r_k(s, e; t)$ transitions from $\operatorname{twr}_{k-2}$ to $\operatorname{twr}_{k-1}$ in a power of $t$ at some $e = h_k^{(k-2)}(s) < \binom{s}{k}$. Neither Erd\H os and Hajnal nor subsequent authors wagered a guess as to what $h_k^{(i)}(s)$ might be for any $i \geq 2$, except for the specific case $s=k+1$; the conjecture for that case was verified by Mubayi and Suk~\cite{MS20} when $k\geq 6$. Perhaps the next step in this area of research is to determine (or at least bound) the transition point from exponential to doubly exponential. 

\begin{problem}
    For $k\ge 4$, determine the function $h_k^{(2)}(s)$ such that $r_k(s, h_k^{(2)}(s)-1; t) = 2^{t^{O(1)}}$ while $r_k(s, h_k^{(2)}(s); t) = 2^{2^{t^{\Omega(1)}}}$, if one exists.
\end{problem}

\paragraph{Acknowledgments.} We are grateful to David Conlon, Jacob Fox, Dhruv Mubayi, Jiaxi Nie, Maya Sankar and Yuval Wigderson for stimulating conversations and helpful comments on an earlier version of this manuscript. The first author was supported in part by a Georgia Tech ARC-ACO fellowship. The third author was supported by NSF grant DMS - 2246682.

\newpage %Delete if too much white space? 

\begin{appendices}
\section{Proof of \texorpdfstring{\cref{lemma:elementary}}{F Lemma}}\label{sec:proof of F lemma}
%\ra{Todo: play around with appendix format, see weak saturation paper}
In this section, we prove the following about the function $F$ from \cref{def:F}.
\flemma*

\begin{proof}%[Proof of \cref{lemma:elementary}]
    We first notice that for any $a_1,b_1,\ldots, a_\ell,b_\ell$, we have
    \[\sum_{i\neq j}a_ib_j = \sum_ia_i\sum_jb_j - \sum_ia_ib_i.\]
    Therefore
    \[F(A,B,C) = AB - \sup_{(a,b)\vdash(A,B,C)}\sum_{i}a_ib_i.\]
    
    By compactness, the supremum is achieved for each fixed length $\ell\in\NN$ of the composition with $\ell C\geq A+B$.
    We may also assume that among the maximizers, the composition maximizes the lexicographical order of $(a_1+b_1,a_2+b_2,\ldots, a_\ell+b_\ell)$ as well.
    Note that if $a_\ell=b_\ell=0$ then we may simply discard them and decrease $\ell$.
    Hence we may assume without loss of generality that $a_\ell+b_\ell>0$.

    Let us first show that $a_i+b_i=C$ if $i<\ell$.
    Assume for the sake of contradiction that $a_i+b_i<C$ for some $i<\ell$.
    By the choice of the composition, we have $0<a_\ell+b_\ell\leq a_i+b_i<C$.
    We may assume, without loss of generality, that $a_\ell\leq a_i$ and $b_\ell>0$.
    Note that for sufficiently small $\varepsilon$, we have $b_\ell-\varepsilon>0$, $a_i+b_i+\varepsilon < C$ and 
    \[a_i(b_i+\varepsilon)+a_\ell(b_\ell-\varepsilon) \geq a_ib_i+a_\ell b_\ell.\]
    Since our sequence is a maximizer, we see that the equality must hold and so the sequence after modifications at $b_i$ and $b_\ell$ must remain a maximizer.
    However, $a_i+b_i+\varepsilon>a_i+b_i$ and $a_j+b_j$ remains unchanged for all $j<i$, which is a contradiction to the lexicographical ordering of $(a_1 + b_1, \ldots, a_\ell+b_\ell)$.
    Therefore we must have $a_i+b_i=C$ for all $i<\ell$.
    As a consequence, we must have $\ell = \lceil (A+B)/C\rceil$ as $(\ell-1)C < A+B$ and $\ell C\geq A+B$.

    Next, for any $i\neq j$, if $a_i-b_i<a_j-b_j$, we show that $a_j=0$ or $b_i=0$.
    Otherwise, choose $\varepsilon>0$ small enough so that $a_j-\varepsilon, b_i-\varepsilon>0$.
    Then we have
    \[(a_i+\varepsilon)(b_i-\varepsilon)+(a_j-\varepsilon)(b_j+\varepsilon) = a_ib_i+a_jb_j+\left((a_j-b_j)-(a_i-b_i)\right)\varepsilon-2\varepsilon^2,\]
    which is strictly greater than $a_ib_i+a_jb_j$ for sufficiently small $\varepsilon$.
    This is a contradiction.
    Therefore we must either have $a_j=0$ or $b_i=0$.

    To show that $a_i-b_i=a_j-b_j$ for all $i<j$ unless $j=\ell$, it suffices (by the above) to show that $a_i,b_i\neq 0$ for all $i<\ell$.
    Assume for the sake of contradiction that $a_i=0$ for some $i<\ell$.
    Then $b_i=C$.
    Moreover, for any $j\neq i$, we have that either $a_i-b_i<a_j-b_j$ or $a_j=0$ and $b_j=C$ as well.
    The former case forces $a_j=0$ or $b_i=0$, which shows that $a_j=0$ must hold as $b_i=C$.
    However, this shows that $A = \sum_{j\in[\ell]}a_i = 0$, which is a contradiction. A similar contradiction occurs if $b_i=0$ for some $i<\ell$.

    Now if $a_i-b_i \neq a_\ell-b_\ell$ for some $i<\ell$, then we know that one of $a_\ell,b_\ell$ must be zero.
    Moreover, $a_\ell-b_\ell$ must be between $0$ and $a_i-b_i$: to see this, assume that $b_\ell=0$ (the other case is analogous).
    If $a_\ell > a_i-b_i$, then we know that $a_\ell-b_\ell>a_i-b_i$ and so either $a_\ell=0$ or $b_i=0$, both of which lead to a contradiction.

    Finally, to show the first item, it remains to show that those conditions already allow us to compute a minimizer.
    Let $D=a_i-b_i$ for any $i<\ell$, and $D' = a_\ell-b_\ell$.
    Then we know that $D'=D$ or $D'$ lies strictly between $0$ and $D$, where the latter only happen when $\min(a_\ell,b_\ell)=0$.
    We also know that $a_i+b_i=C$ for any $i<\ell$ and $a_\ell+b_\ell = (A+B)-(\ell-1)C$.
    Therefore, $D$ and $D'$ together already determine at most one possible $a_1,b_1,\ldots, a_\ell,b_\ell$.
    It remains to show that there is exactly one possible pair $(D,D')$ that leads to a valid $a_1,b_1,\ldots, a_\ell,b_\ell$.

    Observe that $(\ell-1)D+D' = A-B$.
    If $\abs{(A-B)/\ell} > (A+B)-(\ell-1)C$, then we must have $D\neq D'$: otherwise, we have $\abs{a_\ell-b_\ell} = \abs{D'} = \abs{(A-B)/\ell} > (A+B)-(\ell-1)C = a_\ell+b_\ell$, which is a contradiction.
    In the case where $D\neq D'$, we must have $\abs{D'} = (A+B)-(\ell-1)C$ and $\textup{sgn}(D') = \textup{sgn}(D) =\textup{sgn}(A-B)$.
    This uniquely determines $D'$ and so $D$.

    Now if $\abs{(A-B)/\ell}\leq (A+B)-(\ell-1)C$, we show that $D=D'$.
    Otherwise, we have $D'$ lies strictly between $0$ and $D$, and $\abs{D'} = (A+B)-(\ell-1)C$ as $\min(a_\ell,b_\ell)=0$.
    However, this forces $\abs{D}\leq \abs{D'}$, which is a contradiction.
    Therefore we have $D=D'=(A-B)/\ell$, showing that $(D,D')$ can be uniquely determined in this case as well.
    This concludes the proof of the first item.

    To show the second item, simply note that if $A$ is decreased to $A'$, then for every $(A,B,C)$-composition $(a,b)$, we can simply replace $a_i$ by $\frac{A'}{A}a_i$ to get an $(A',B,C)$-composition.
    This shows that $F(A,B,C)\geq F(A',B,C)$ whenever $A>A'$.
    A similar argument works for $B$.
    Lastly, for any $C<C'$, we know that any $(A,B,C)$-composition $(a,b)$ is also an $(A,B,C')$-composition, showing that $F(A,B,C)\geq F(A,B,C')$, as desired.

    To show the last item, we first observe that we may assume $A>0$ as otherwise the inequality trivially holds.
    Now note that for any $(A,B,C)$-composition $(a,b)$, we may replace $b_i$ with $\frac{B'}{B}b_i$ so that the it becomes an $(A,B',C)$-composition.
    Since $\sum_{i\neq j}a_ib_j$ becomes $\frac{B'}{B}\sum_{i\neq j}a_ib_j$, we have the desired inequality.
\end{proof}

\section{Verifying assumptions of the lemmas}\label{sec:verifying assumptions}
%In this section, we finish the proof that \cref{asm:contradiction} cannot hold for any value of $n$. 
Recall that in \cref{sec:degree sequences}, we showed that \cref{asm:contradiction} cannot hold with $n=13, 14, 16,$ or $17$. For all other values of $n$ up to $699$, we use a computer program to verify the assumptions of \cref{lemma:possible-X-Y}, \cref{lemma:Second-step}, and \cref{lemma:Third-step} and then check that all colorings $\chi$ satisfying the lemmas in \cref{sec:tournaments} have $\tau \leq 2\min(|X^*|, |Y^*|, |U|)$. Then, from \cref{lem:tau-bad} and \cref{lemma: recoloring} we conclude that \cref{asm:contradiction} cannot hold. Our program 
% \xh{not sure if this matters but I find the word "code" unprofessional, perhaps we can use "program" or "algorithm" everywhere?} 
can be found at \href{https://github.com/rubenascoli01/Polynomial-to-exponential-transition}{\texttt{github.com/rubenascoli01/Polynomial-to-exponential-transition}}%\xh{Can we make this clickable?}%\ra{Decide whether to include the code in the appendix}
, and we elaborate on how it works in the first subsection below. Then, the second subsection is dedicated to showing (without computer assistance) that for $n\geq 700$, the assumptions of the lemmas in \cref{sec:tournaments} hold and, once again, we have $\tau\leq 2\min(|X^*|, |Y^*|, |U|)$. %This allows us to use \cref{lemma: recoloring} to conclude that \cref{asm:contradiction} cannot hold for any $n\geq 700$.

\subsection{Explanation of Python program}
%\xh{I am not convinced most of this section needs to be in the paper, versus spread between the readme and the comments of the python code itself.}

Recall the notation $\tilde d(n) = \frac{n(n+1)(n-1)}{3} - 8(g_3(n)+1)$. Define $\Delta_{\max}$ to be the largest nonnegative integer satisfying
\begin{equation}\label{Deltamax} \Delta_{\max}^2 \leq \tilde d(n).\end{equation}
We now describe how our computer program algorithmically arrives at a contradiction to \cref{asm:contradiction}, given the tools in this paper to help prove that $\cG_{\Delta, P}$ is empty for some integers $\Delta$ and $P$.
%This will be most useful for designing an algorithm that proves a contradiction with \cref{asm:contradiction} for small $n$.
\begin{lemma}\label{lem:algorithm}
To contradict \cref{asm:contradiction}, it suffices to show that the following procedure succeeds.
Initialize $\Delta^{(0)}=-1$ and $P^{(0)} = 4(d(n)-1)$.
At time $t\geq 1$, set $\Delta^{(t)}$ to be the maximum integer in $\{\Delta^{(t-1)}+1,\ldots, \Delta_{\max}\}$ so that we can prove $\cG_{\Delta^{(t)}, P^{(t-1)}}$ is empty.
The procedure succeeds if $\Delta^{(t)} = \Delta_{\max}$, and it fails if $\Delta^{(t)}$ does not exist. If neither of these holds, then set
\[P^{(t)}=\left\lfloor \frac{\tilde d(n)-\left(\Delta^{(t)}+1\right)^2}{2}\right\rfloor.\] 
If $P^{(t)} \geq P^{(t-1)}$, then the procedure fails.  Otherwise, $t$ increases and the procedure continues to the next iteration.
\end{lemma}
\begin{proof}
Suppose that the procedure succeeds at time $T$.
We have that $\Delta^{(T)} = \Delta_{\max}$.

We claim it suffices to show that for any nonnegative integers $\Delta, P$ satisfying the conditions in \cref{cor:calGempty}, we can find $t\in[T]$ so that $\Delta\leq \Delta^{(t)}$ and $P\leq P^{(t-1)}.$ Indeed, if this holds, then since $\cG_{\Delta, P} \subseteq \cG_{\Delta^{(t)}, P^{(t-1)}}$ and our procedure showed at time $t$ that $\cG_{\Delta^{(t)}, P^{(t-1)}}$ is empty, we conclude that $\cG_{\Delta, P}$ is empty. \cref{cor:calGempty} tells us that if \cref{asm:contradiction} holds, then for some $\Delta, P$ satisfying the conditions of the corollary, $\cG_{\Delta, P}$ is nonempty, yielding a contradiction.

To find this, given $\Delta$ and $P$, let $t$ be the smallest positive integer with $\Delta^{(t)} \geq \Delta$.
This positive integer exists as we must have $\Delta \leq \Delta_{\max} = \Delta^{(T)}$.
It remains to show that $P\leq P^{(t-1)}.$
This is clear if $t=1$.
If $t>1$, then by the minimality of $t$, we know that $\Delta\geq \Delta^{(t-1)}+1$.
Therefore
\[\Delta^2+2P\geq \left(\Delta^{(t-1)}+1\right)^2+2P,\]
and so if $P > P^{(t-1)}$ then the condition $P \leq P(\Delta)$ of the \cref{cor:calGempty} is violated.
\end{proof}
%At this point, it is worth noting that for $n\leq 699$, we use computer code to go through the algorithm outlined by \cref{cor:algorithm}.  For $50 \leq n \leq 699$, the procedure succeeds in the first iteration, i.e. $\cG_{\Delta_{\max}, 4(d(n)-1)}$ is empty. For $n\geq 700$, the same holds, but we show this by hand. See \cref{sec:verifying assumptions} for details.

Our computer program executes the algorithm outlined in \cref{lem:algorithm}. Comments in the code explain how we use the lemmas in \cref{sec:tournaments} to argue that a particular $\cG_{\Delta, P}$ is empty. %The call to the function \texttt{no\_counterexample} is where we check if we can prove that $\cG_{\Delta, P}$ is empty; this is where the bulk of the computational work happens. We proceed from there as outlined in \cref{lem:algorithm} and as explained by the comments in that section of code.

We now elaborate on a part of the program inside the function \texttt{max\_bad\_edges}. To use \cref{lem:tau-bad}, it suffices to obtain an upper bound on $\max_{X^*}b_U$ and $\max_{Y^*}b_U$.
Our program uses slightly different strategies for $n \geq 70$ and for $n < 70$. For $n \geq 70$, we use the following. Recall the definition of $\tau$ from \eqref{eq:coeffs and tau}.
\begin{lemma}\label{lem:bmaxfromx} Suppose that
\begin{equation}\label{eq:not_enough_link}
\bip(\abs{X^*})+\bip(\abs{Y^*})+\bip(\abs{U})+\tau\leq g_3(n)-g_3(n-1).\end{equation}
Let $b_{\max}$ be the largest nonnegative integer $b$ such that $b\leq \abs{U}$ and 
\begin{align}\label{eq: b max equation}\left(\abs{U}-b\right)\abs{Y^*}+\bip(\abs{X^*}-1)+\bip(b)+\tau-b\geq g_3(n)-g_3(n-1)+1.\end{align} Then, $b_{U}(x)\leq b_{\max}$ for every $x\in X^*$.
\end{lemma}
\begin{proof}
Fix a vertex $x\in X^*$. 
For any monochromatic triangle $xpq$ that contains $x$, we have three cases:
    \begin{itemize}
        \item $\chi(xp)=\chi(xq) = \chi(pq) = c$.
        \item $\chi(xp)=\chi(xq)=\chi(pq)\neq c$ and $c(p)=c(q)$.
        \item $\chi(xp)=\chi(xq)=\chi(pq)\neq c$ and $c(p)\neq c(q)$.
    \end{itemize}
    There are at most $(\abs{U}-b_U(x))(\abs{Y^*}-b_{Y^*}(x))$ triangles of the first type.
    For the second one, by using the fact that the link of the triangles is bipartite, we know that there are at most
    \[\bip(b_U(x))+\bip(b_{Y^*}(x))+\bip(\abs{X^*}-1)\]
    such triangles.
    Lastly, any triangle of the third type corresponds to an edge $pq$ across $U,X^*,Y^*$ not incident to $x$ with color not equal to $c$.
    There are at most $\tau-b_U(x)-b_{Y^*}(x)$ such triangles.
    As a consequence,
    \[\bip(b_U(x))+\bip(b_{Y^*}(x))+\bip(\abs{X^*}-1)+(\abs{U}-b_U(x))(\abs{Y^*}-b_{Y^*}(x))+\tau-b_U(x)-b_{Y^*}(x)\geq g_3(n)-g_3(n-1)+1.\]
    The left hand side is convex in $b_{Y^*}(x)$ as $\bip(b_{Y^*}(x))$ is convex in $b_{Y^*}(x)$ and the rest is linear in $b_{Y^*}(x)$.
    Therefore it is maximized when $b_{Y^*}(x)=0$ or $b_{Y^*}(x) = \abs{Y^*}$.
    We now show that it cannot be maximized at $b_{Y^*}(x)=\abs{Y^*}$.
    Otherwise, we must have
    \[\bip(b_U(x))+\bip(\abs{Y^*})+\bip(\abs{X^*}-1)+\tau-b_U(x)\geq g_3(n)-g_3(n-1)+1.\]
    As $\bip(t)$ is increasing, we know that
    \[\bip(b_U(x))+\bip(\abs{Y^*})+\bip(\abs{X^*}-1)+\tau-b_U(x)\leq \bip(\abs{U})+\bip(\abs{X^*})+\bip(\abs{Y^*})+\tau,\]
    which contradicts our assumption \eqref{eq:not_enough_link}.
    Thus, plugging in $b_{Y^*}(x)=0$, we have
    \begin{align*} &\bip(b_U(x))+\bip(b_{Y^*}(x))+\bip(\abs{X^*}-1)+(\abs{U}-b_U(x))(\abs{Y^*}-b_{Y^*}(x))+\tau-b_U(x)-b_{Y^*}(x)\\&\leq \bip(b_U(x))+\bip(\abs{X^*}-1)+(\abs{U}-b_U(x))\abs{Y^*}+\tau-b_U(x).\end{align*}
    By our definition of $b_{\max}$, we know that $b_U(x)\leq b_{\max}$ for every $x\in X^*$.
\end{proof}
Note that as a consequence of this proof, if \eqref{eq:not_enough_link} holds but there is no $0 \leq b \leq |U|$ satisfying \eqref{eq: b max equation}, then we can discard this case as no $x\in X^*$ can be in enough monochromatic triangles. 

%We remark that for $n\leq 699$, the code checks the assumption \eqref{eq:not_enough_link} in the \texttt{max\_bad\_edges} function, and if it does not hold, we simply allow this case to be considered for the final value returned by \texttt{max\_bad\_edges}. 

%\ra{Probably better to format next part as a lemma, but I'm not sure how as there are several steps...}

For $n < 70$, we need a slightly more complicated method to find a good bound on $\max_{X^*}b_U$. 
This will be a two-step process.
We first try to bound $\max_{X^*}b_U$ in a similar fashion as what we did for $n\geq 70$, but this time with a more refined approach.

\begin{lemma}\label{lem:bmaxnlessthan70}
    Let $b_{\max}$ be the largest nonnegative integer $b$ such that $b\leq \abs{U}$ and there exists
    \[b'\in \left\{(\abs{Y^*}-(\abs{U}-b+\Delta))_+, \min(\abs{Y^*}, \abs{Y^*}-(\abs{U}-b-\Delta))\right\}\] so that
\[\left(\abs{U}-b\right)\left(\abs{Y^*}-b'\right)+\bip(\abs{X^*}-1)+\bip(b)+\bip(b')+\tau-b-b'\geq g_3(n)-g_3(n-1)+1.\]
Then $b_U(x)\leq b_{\max}$ for all $x\in X^*$.
\end{lemma}
\begin{proof}
    From the previous proof, we know that
    \begin{align}\label{eq:upperbound-triangle}&\left(\abs{U}-b_U(x)\right)\left(\abs{Y^*}-b_{Y^*}(x)\right)+\bip(\abs{X^*}-1)+\bip(b_U(x))+\bip(b_{Y^*}(x))+\tau-b_U(x)-b_{Y^*}(x)\end{align}
    is at least $g_3(n)-g_3(n-1)+1$.
    We now replace $b_{Y^*}(x)$ with some $b'$ so that the inequality still holds and either $b'=(\abs{Y^*}-(\abs{U}-b_U(x)+\Delta))_+$ or $b'=\min(\abs{Y^*}, \abs{Y^*}-(\abs{U}-b_U(x)-\Delta))$.
    To do so, denote by $I_{b_U(x)}$ the interval with those two values as endpoints.
    Note that 
\[\abs{(\abs{Y^*}-b_{Y^*}(x))-(\abs{U}-b_U(x))}=\delta_{c^*}(x)\leq \Delta\]
and so $b_{Y^*}(x)\geq \abs{Y^*}-(\abs{U}-b_U(x)+\Delta)$ and $b_{Y^*}(x)\leq \abs{Y^*}-(\abs{U}-b_U(x)-\Delta)$.
In other words, $b_{Y^*}(x)\in I_{b_U(x)}$.
Since the expression \eqref{eq:upperbound-triangle} is convex in $b_{Y^*}(x)$, we may replace $b_{Y^*}(x)$ with one of the endpoints of $I_{b_U(x)}$ while keeping it at least $g_3(n)-g_3(n-1)+1$.
By the definition of $b_{\max}$, we then see that $b_U(x)\leq b_{\max}$.
\end{proof}

The second step for $n<70$ is to bootstrap our bound on $b_U(x)$ using part (b) of \cref{lemma:bipartite-to-tripartite}. 
To do so, we need to have $b_U(x)\leq \abs{U}/2$.
The next lemma lists two criteria that imply this.

\begin{lemma}\label{lem:bootstrap-bU}
    Set 
    \[b_{U,Y^*}= \left\lfloor\frac{\left(\abs{U}b_{XY}+\left(\floor{P/2}-b_{XY}\right)\right)\abs{Y^*}}{X_{\min}{Y_{\min}}}\right\rfloor.\]
    Suppose that at least one of the following two statements holds.
    \begin{enumerate}
        \item The $b_{\max}$ from \cref{lem:bmaxnlessthan70} is at most $\abs{U}/2$.
        \item We have the inequality
        \[F\left(\abs{Y^*}, \abs{U}-\left\lfloor\frac{P - 2(b_{XY}-\abs{Y^*})_++2b_{U,Y^*}}{2\abs{Y^*}}\right\rfloor, \frac{\abs{X^*}+\abs{Y^*}+\Delta}{2}\right)>P - 2(b_{XY}-\abs{Y^*})_++b_{U,Y^*}.\]
    \end{enumerate}
    Then for every $x\in X^*$,
    \[b_U(x)\leq \left\lfloor\frac{\left\lfloor\frac{P}{2}\right\rfloor-(b_{XY}-\abs{Y^*})_++b_{U,Y^*}}{|Y^*|}\right\rfloor.\]
\end{lemma}
\begin{proof}
    Let us first show that either criterion implies $b_U(x)\leq \abs{U}/2$ for all $x\in X^*$.
    By \cref{lem:bmaxnlessthan70}, the first criterion clearly implies $b_U(x)\leq \abs{U}/2$.
    Now suppose that the second criterion holds, and we fix some $x\in X^*$.
    Apply \cref{lemma:bipartite-to-tripartite} with $v=x$, $S_1=Y^*$, $S_2 = U$, $b=b_{U,Y^*}$, $Q = P - 2(b_{XY}-\abs{Y^*})_+$ and $M = (|X^*|+|Y^*|+\Delta)/2$, which is an obvious upper bound on $(\abs{X}+\abs{Y}+\Delta)/2$.
    In order to apply it in this way, we need to first upper bound the number of edges not colored $c^*$ between $U$ and $Y^*$.
    Using the notation introduced right before \cref{lem:overlineM}, we can bound this by 
\[\sum_{u\in U}b_Y(u)+\sum_{y\in \overline{Y}}b_U(y) \leq \frac{\abs{U}b_{XY}+p+\pc_2(U, X, \overline{Y})+\frac{(\abs{U}b_{XY}+p)\abs{\overline{Y}}}{\abs{Y}}}{\abs{X}} \leq \frac{(\abs{U}b_{XY}+p)\abs{Y^*}+\pc_2(U, X, \overline{Y})\abs{Y}}{X_{\min} Y_{\min}},\]
and we just need to use the fact that $p+\pc_2(U, X, \overline{Y})\leq \lfloor\frac{P}{2}\rfloor-b_{XY}$ to upper bound this expression by $b_{U,Y^*}$.
We also need to show that there can be at most $P - 2(b_{XY}-\abs{Y})_+$ precyclic triangles that $x$ form with $U$ and $Y^*$, where $2$-precyclic triangles are counted twice.
Note that $v^*$ already forms $b_{XY}$ $2$-precyclic triangles with $X$ and $Y$, and there can be at most $\abs{Y}$ that use $x$.
Therefore there are at least $(b_{XY}-\abs{Y})_+$ $2$-precyclic triangles that do not use $x$ at all.
This shows the desired bound, and part (a) of \cref{lemma:bipartite-to-tripartite} now gives that $b_U(x)\leq \abs{U}/2$, as desired.

Now by \cref{lemma:bipartite-to-tripartite}(b), we have 
\[\abs{Y^*}b_U(x) - b_{U,Y^*}\leq \left\lfloor\frac{P}{2}\right\rfloor-(b_{XY}-\abs{Y^*})_+.\]
Rearranging gives the desired inequality.
\end{proof}
% This can be checked in two ways.
% First, if the integer $b_{\max}$ computed above satisfies $b_{\max}\leq \abs{U}/2$, then the desired bound holds.
% Otherwise, we try to use \cref{lemma:bipartite-to-tripartite} with $S_1=Y^*$, $S_2 = U$, 
% \[b=b_{U,Y^*}\stackrel{\textup{def}}{=} \left\lfloor\frac{\left(\abs{U}b_{XY}+\left(\floor{P/2}-b_{XY}\right)\right)\abs{Y^*}}{X_{\min}{Y_{\min}}}\right\rfloor,\]
% \[Q = P - 2(b_{XY}-\abs{Y^*})_+,\] and $M = (|X^*|+|Y^*|+\Delta)/2$, which is an obvious upper bound on $(|X|+|Y|+\Delta)/2$.
% To apply \cref{lemma:bipartite-to-tripartite} with this set of parameters, we need to first upper bound the number of edges not colored $c^*$ between $U$ and $Y^*$.

% \]
% we thus update $b_{\max}$ accordingly.

\subsection{Verifying the assumptions for \texorpdfstring{$n\geq 700$}{n>=700}}
%As mentioned in \cref{sec:tournaments}, 
For $n\geq 700$, we show by hand via \cref{lemma: recoloring} that $\cG_{\Delta_{\max}, 4(d(n)-1)}$ is empty. This mostly amounts to determining the orders of magnitude of the various quantities mentioned in the lemmas in \cref{sec:tournaments} and confirming with direct computations that the lower order terms are insignificant. 

Throughout this section, all logarithms are base $2$ and we do not round any numbers unless specified. We begin with the following estimates on $d(n)$ and $\Delta_{\max}$. 
\begin{lemma}\label{lem:d(n) bound}
    For $n\geq 700$, we have $d(n) \leq 0.06n\log n$ and $\Delta_{\max} \leq 0.76\sqrt{n\log n}$.
\end{lemma}
\begin{proof}
Our program checks that $d(n) \leq 0.05891n\log n$ holds for $200\leq n < 600$; see Section 3 of the program. To prove the first part of the lemma, we %to show that $d(n) \leq 0.05891n\log n$ for all $n\geq 600$, 
show that for any $n$, if $d(n') \leq c_0 n'\log n'$ holds for all $\floor{n/3} \leq n' < n$, then $d(n) \leq c_0 n\log n$ holds as well. Indeed, this is apparent from \cref{lem:d(n) recursion}. For example, if $d(2x-1)\leq c_0(2x-1)\log(2x-1)$ and $d(2x) \leq c_0(2x)\log(2x)$, then $d(6x-2) = 2d(2x-1)+d(2x) \leq 2c_0(2x-1)\log(2x-1)+c_0(2x)\log(2x) \leq c_0(6x-2)\log(6x-2)$. Similar statements hold with the other five recursive equations defining $d(n)$.

Next, from \eqref{Deltamax}, %\hy{maybe also refer to the definition of $\Tilde{d}(n)$?} \ra{I put a reminder near \eqref{Deltamax}} 
we have
\begin{align}\label{eq:Deltamaxtilded}\Delta_{\max}^2 \leq \tilde d(n) = \frac{1}{3}n(n+1)(n-1)-8(T(n)-d(n)+1) \leq n+8d(n),\end{align} so by our estimate on $d(n)$, we have $\Delta_{\max} \leq \sqrt{n+8\cdot 0.05891n\log(n)}\leq 0.76\sqrt{n\log n}$ for $n\geq 700$. 
\end{proof}

We remark that a more careful analysis of the recursion in \cref{lem:d(n) recursion} shows that $$\limsup_{n\to\infty} \frac{d(n)}{n\log n} = \frac{1}{12\log 3} \approx 0.0526,$$ so the above lemma is close to optimal.

%We next estimate $\Delta_{\max}$. \ra{Include this in above lemma?}
%\xh{Here and throughout, if we only need it for $n\ge 700$ then just state it for $n\ge 700$.}

Now, we verify that the assumptions of \cref{lemma:possible-X-Y},
%\cref{lemma:First-step},
\cref{lemma:Second-step}, and \cref{lemma:Third-step} hold for $n\geq 700$, and that we can take $\eta = 1/8$.

\begin{lemma}\label{lem:assumptions>=700}
Let $n\geq 700$, $\Delta = \Delta_{\max},$ and $P=4(d(n)-1)$. The following hold:
\begin{enumerate}[(a)]
    \item $\lceil\frac{12g_3(n)}{n(n-1)}\rceil\geq \lceil n/2\rceil$.
    %\item $\frac{n^2}{16}-\frac{\Delta^2}{4}\geq 16(d(n)-1)$.
    \item For all admissible pairs $(s,t)$ (as defined in \cref{def:admissible-pair}), we have
    \[F\left(s,t-\left\lfloor \frac{P}{2s}\right\rfloor, \left\lfloor\frac{s+t+\Delta}{2}\right\rfloor\right)>P.\]
    \item We have $\tilde d(n) < n^2/4$, and for all admissible pairs $(s,t)$ %\hy{probably should replace it by all $s,t$ satisfying the previous item, or just define the set of possible $s,t$'s.}
    and nonnegative integers $k$ with $0\leq k\leq \floor{P/2}$, the following holds: If there exists a minimum positive integer $w$ at most $n-s-t$ such that 
    \begin{align}\label{eq:w-req}s\max\left(t-\frac{k}{s}-w,0\right)^2+t\max\left(s-\frac{k}{t}-w,0\right)^2+(s-t)^2\leq \tilde d(n)-4k,\end{align}
    then the minimum such $w$ satisfies
    \begin{align}\label{eq:w-ineq1}F\left(s,w-\frac{\lfloor\frac{P}{2}\rfloor-k+\frac{wk}{t}}{s},\left\lfloor\frac{s+t+\Delta}{2}\right\rfloor\right)>P-2k+\frac{wk}{t}\end{align}
    and
    \begin{align}\label{eq:w-ineq2}F\left(w,s-\frac{\lfloor\frac{P}{2}\rfloor-k+\frac{wk}{t}}{w},\left\lfloor\frac{s+t+\Delta}{2}\right\rfloor\right)>P-2k+\frac{wk}{t}.\end{align}
    \item $\frac{\floor{P/2}}{\abs{X}\abs{Y}} \leq 1/8$. %\ra{adjust proof of this item (incorporate old proof of lemma:First-step); order has swapped so this is last instead of second.}
\end{enumerate}
\end{lemma}
\begin{proof}
% Let's examine (a) first. Recall that $g_3(n) = T(n) - d(n)$. If $n$ is even, the left hand side is at least $\frac{1}{2}\frac{n(n^2-4) - 24d(n)}{n(n-1)}$ and the right hand side is $n/2$. Note that $24d(n) \leq 1.44n\log n \leq 0.03n(n-1)$ for $n\geq 700$, so it suffices to show that $\frac{n^2-4}{n-1} - 0.03\geq n$; indeed, this is a linear inequality in $n$ that holds for $n\geq 700$. If instead $n$ is odd, the left hand side is $\ceil{\frac{1}{2}\frac{n(n^2-1) - 24d(n)}{n(n-1)}}$ and the right hand side is $(n+1)/2$. Similarly to the even case, because of the ceiling, it suffices to show that $\frac{1}{2}\left(\frac{n^2-1}{n-1} - 0.03\right) > \frac{1}{2}(n +1)-1$, or $n+1 - 0.03 > n-1$, which obviously holds.

%\xh{Edited part (a) below:}

We examine (a) first. Recall that $g_3(n) = T(n) - d(n)$. For $n\ge 700$, we have $24d(n) \le 1.44n \log n \le 0.03 n (n-1)$ by \cref{lem:d(n) bound}, and so for $n$ even we have
\[
\left\lceil\frac{12g_3(n)}{n(n-1)}\right\rceil\geq \left\lceil0.5 \left(\frac{n^2-4}{n-1}\right) - 0.03\right\rceil \ge 0.5n,
\]
as desired. A similar calculation holds for $n$ odd.

For (b), %we shall use $A, \B, C$ when describing arguments to the function $F$ to avoid confusion with $B = 4(d(n)-1)$. 
since $F(A,B,C)$ is increasing in $A$ and $B$ and decreasing in $C$, it suffices to verify the inequality where we lower the first two arguments or raise the third argument. We know that  $s+t \geq 0.5n$ and $|s-t| \leq \Delta$. By \cref{lem:d(n) bound}, we have $s \geq 0.25n-0.38\sqrt{n\log n}$ and so $$\left\lfloor\frac{4(d(n)-1)}{2s}\right\rfloor \leq \frac{0.24n\log n}{0.5n-0.76\sqrt{n\log n}} \leq 0.6\log n$$% \xh{This looks awful inline}
for $n\geq 700$. With this in mind, set $A = s$ and $B = t-0.6\log n$. Also, the third argument is at most $C=0.5(s+t)+0.38\sqrt{n\log n}$. It thus suffices to show that $F(A, B, C) > 0.24n\log n$. 

By \cref{lemma:elementary}, the infimum in the definition of $F$ is achieved by having $\ell=\ceil{(A+B)/C}$ terms for each of $a_i$ and $b_j$. Clearly $(A+B)/C \leq 2$ since $A+B \leq s+t$ and $C \geq 0.5(s+t)$; also $(A+B)/C > 1$ since $s+t \geq 0.5n$ and $(0.5n - 0.6\log n)/(0.25n + 0.38\sqrt{n\log n}) > 1$ for $n\geq 700$.
%Furthermore, $|(A-B)/2| \leq |s-t|/2+0.305\log n \leq 0.385\sqrt{n\log n} + 0.305\log n$, which is less than $(A+B)-(2-1)C = (s+t)/2 - 0.385\sqrt{n\log n} - 0.61\log n\geq n/4 - 0.385\sqrt{n\log n} - 0.61\log n$ as long as $n\geq 108$.

Hence, the infimum is obtained by letting %$a_1-b_1=a_2-b_2 = (A-B)/2 = (s-t)/2+0.305\log n$, giving 
$a_1 = 0.5s + 0.19\sqrt{n\log n} + 0.15\log n$, $b_1 = 0.5t + 0.19\sqrt{n\log n} - 0.15\log n$, $a_2 = 0.5s - 0.19\sqrt{n\log n} - 0.15\log n$, and $b_2 = 0.5t - 0.19\sqrt{n\log n} - 0.45\log n$. Thus $$F(A,B,C) = a_1b_2 + a_2b_1 =0.5s(t-0.6\log n) - 0.0722n\log n - 0.114\sqrt{n\log^3 n} -0.045\log^2n.$$ 
To minimize this subject to the constraints $s+t\geq 0.5n$ and $|s-t|\leq 0.76\sqrt{n\log n}$, we choose $s =0.25n + 0.38\sqrt{n\log n}$ and $t=0.25n - 0.38\sqrt{n\log n}$. Substituting these into the above, the resulting expression is greater than $0.24n\log n$ when $n \geq 700$.

For the first statement of (c), we compute $$\tilde d(n) = \frac{1}{3}n(n+1)(n-1)-8(T(n)-d(n)+1) \leq n+8d(n) \leq n+0.48n\log n,$$ which is less than $n^2/4$ for $n\geq 700$, as required.

For the second statement of (c), we first obtain a lower bound for $w$. 
As seen in the previous paragraph, the right hand side of \eqref{eq:w-req} is at most $n+0.48n\log n$% \hy{why? I think this is not right?}
. The constraints $|s-t|\leq 0.76\sqrt{n\log n}$ and $s+t\geq 0.5n$ imply $s, t\geq 0.2 n$ when $n\geq 700$. We have $k\leq \floor{P/2} \leq 2d(n) \leq 0.12n\log n$, so $k/s, k/t \leq 0.6\log n$. Thus, $w \geq t-1.2\log n$ and $w\geq s-1.2\log n$. Indeed, otherwise, the left hand side of \eqref{eq:w-req} would be at least $0.2n(0.6\log n)^2> n+0.48n\log n$ since $n \geq 700$, causing the left hand side to exceed the right hand side.

Now we lower bound $F$ to verify \eqref{eq:w-ineq1} and \eqref{eq:w-ineq2}, in a similar manner to what we did for part (b).

For both \eqref{eq:w-ineq1} and \eqref{eq:w-ineq2}, the left hand side is increasing with $w$. (In particular, the coefficient of $w$ in the second argument in \eqref{eq:w-ineq1} is $1-\frac{k}{st} \geq 1-\frac{2d(n)}{st}\geq 1-\frac{0.12n\log n}{0.06n^2} > 0$ for $n \geq 700$.) So, we may assume $w$ is as small as possible when lower bounding the left hand side, i.e. $w=t-1.2\log n$. For \eqref{eq:w-ineq1}, this means $wk/t \leq k$, so $w-(2d(n)-k+wk/t)/s \geq w-0.12n\log n / (0.2 n) = t-1.8\log n$. With this in mind, set $A_1 = s$ and $B_1 = t-1.8\log n$. For \eqref{eq:w-ineq2}, the bounds $w\geq s-1.2\log n$, $w\geq t-1.2\log n$, and $s+t\geq 0.5n$ imply that $w\geq 0.25n - 1.2\log n \geq 0.2n$ when $n\geq 700$, so the second argument is at least $s-0.6\log n$ (following the same logic as before). So, set $A_2 = t-1.2\log n$ and $B_2=s-0.6\log n$. For both inequalities, set $C = 0.5(s+t) + 0.38\sqrt{n\log n}$. 

We bound the right hand side, $P-2k+wk/t$, independently of our assumption above that $w$ is the smallest allowed value. Indeed, we just use that $w\leq n-(s+t)\leq n-0.5n = 0.5n$, so since $t\geq 0.2 n$, we have $wk/t \leq 2.5k$. Hence the right hand side is at most $4d(n) + 0.5k \leq 5d(n) \leq 0.3n\log n$.

We are ready to verify \eqref{eq:w-ineq1}. In the notation of \cref{lemma:elementary}, we have $\ell=2$, and we obtain %$|(A_1-B_1)/2| \leq |s-t|/2 + 1.3\log n/2 \leq 0.38\sqrt{n\log n} + 0.65\log n$, which is less than $(A_1+B_1)-(2-1)C = (s+t)/2-1.3\log n - 0.385\sqrt{n\log n} \geq n/4-1.3\log n - 0.385\sqrt{n\log n}$ as long as $n\geq 165$. Hence, we are in the case $a_1-b_1=a_2-b_2 = (A_1-B_1)/2 = (s-t)/2+1.3\log n/2$, giving 
$a_1 = 0.5s + 0.19\sqrt{n\log n} + 0.45\log n$, $b_1 = 0.5t + 0.19\sqrt{n\log n} - 0.45\log n$, $a_2 = 0.5s - 0.19\sqrt{n\log n} - 0.45\log n$, and $b_2 = 0.5t - 0.19\sqrt{n\log n} - 1.35\log n$. Thus, $$F(A_1,B_1,C) = a_1b_2+a_2b_1 = 0.5s(t-1.8\log n) - 0.0722n\log n - 0.342\sqrt{n\log^3 n} - 0.405\log^2(n)
.$$ 
To minimize this subject to the constraints $s+t\geq 0.5n$ and $|s-t|\leq 0.76\sqrt{n\log n}$, we choose $s = 0.25n+0.38\sqrt{n\log n}$ and $t=0.25n - 0.38\sqrt{n\log n}$. Substituting these into the above, the resulting expression is greater than $0.3n\log n$ when $n\geq 700$.

Verifying \eqref{eq:w-ineq2} is similar. We have $\ell=2$, %and $|(A_2-B_2)/2| \leq 0.385\sqrt{n\log n} + 0.05\log n$, which is less than $(A_2+B_2)-(2-1)C$, so we are in the case $D=D' = (A_2-B_2)/2 = (t-s)/2 - 0.05\log n$, giving 
$a_1 = 0.5t + 0.19\sqrt{n\log n} - 0.15\log n$, $b_1 = 0.5s + 0.19\sqrt{n\log n} + 0.15\log n$, $a_2 = 0.5t - 0.19\sqrt{n\log n} - 1.05\log n$, $b_2 = 0.5s - 0.19\sqrt{n\log n} - 0.75\log n$. Thus, $$F(A_2, B_2, C) = a_1b_2 + a_2b_1 = 0.5st - 0.3t\log n - 0.6s\log n - 0.0722n\log n-0.342\sqrt{n\log^3 n} - 0.045\log^2n.$$ Again, this is minimized when $s = 0.25n +0.38\sqrt{n\log n}$ and $t = 0.25n - 0.38\sqrt{n\log n}$, and then the quantity is more than $0.3n\log n$ for $n\geq 700$.

For (d), since $P = 4(d(n)-1)$, it suffices to show $d(n)-1\leq \frac{1}{16}\abs{X}\abs{Y}$.
\cref{lem:d(n) bound} implies that the inequality 
\begin{equation*}\frac{n^2}{16}-\frac{\Delta^2}{4}\geq 16(d(n)-1)\end{equation*} holds as long as $\frac{n^2}{16} - \frac{(0.76)^2n\log n}{4} \geq 0.96n\log n$, which is true for $n\geq 700$.
It then suffices to show that
\[\abs{X}\abs{Y}\geq \frac{n^2}{16}-\frac{\Delta^2}{4}.\]
The assumption of \cref{lemma:possible-X-Y} (which is (a) of this lemma) is met, so we have
\[\abs{X}\abs{Y} = \frac{1}{4}\left((\abs{X}+\abs{Y})^2-(\abs{X}-\abs{Y})^2\right)\geq \frac{1}{4}\left(\left(\frac{n}{2}\right)^2-\Delta^2\right),\]
which proves the desired inequality.
\end{proof}

Finally, we use \cref{lem:overlineM}, \cref{lem:tau-bad}, and \cref{lem:bmaxfromx} to show that in all allowable situations, there are at most $2\min(|X^*|, |Y^*|, |U|)$ edges across $X^*, Y^*, U$ not colored $c^*$, which contradicts \cref{lemma: recoloring} and finishes the proof.

\begin{lemma}
Let $n\geq 700, \Delta = \Delta_{\max},$ and $P = 4(d(n)-1)$. We always have $\tau \leq 2\min(|X^*|, |Y^*|, |U|)$.
\end{lemma}
%\ra{Do we need to say anything else in that lemma to make it self-contained? Like, ``if $X, Y,$ etc satisfy the conditions of \cref{lem:overlineM}" or something?}
\begin{proof}
Recall the notation of \cref{lem:tau-bad}. Our final goal is to obtain an upper bound on $\tau$, and we do this in three steps. First, we bound all the relevant quantities except $b_{XY}$, yielding a weak upper bound on $\tau$. Second, we use this weak bound with \cref{lem:bmaxfromx} to bound $\max_{X^*}b_U$ and $\max_{Y^*}b_U$. Third, those bounds imply a strong upper bound on $b_{XY}$ via \cref{lem:tau-bad}, yielding a stronger upper bound on $\tau$.

We first obtain a crude lower bound on $|U|$. By \cref{lemma:u-lower-bound} and the bounds on $w$ in the proof of \cref{lem:assumptions>=700}(d), we have that $|U| \geq |X|-1.2\log n$ and $|U| \geq |Y|-1.2\log n$. Since $|X|+|Y| \geq 0.5n$, at least one of them is at least $0.25n$, so $|U| \geq 0.25n - 1.2\log n \geq 0.23 n$ for $n\geq 700$. 

We now upper bound the quantity $\overline{M}$, defined in \eqref{eq:overlineM}. By %\cref{lemma:First-step} and 
\cref{lem:assumptions>=700}(d), we can take $\eta = 1/8$. Using the bounds $\floor{P/2} \leq 2d(n) \leq 0.12n\log n$, $b_{XY} \geq 0$, $\Delta \leq 0.76\sqrt{n\log n}$, and $|U| \geq 0.23n$, we obtain:
$$\overline{M} \leq \frac{0.12n\log n}{0.875\left(0.25n-0.38\sqrt{n\log n}\right)}\left(1+\frac{1}{0.23n}\right) \leq 0.7\log(n) 
\leq 0.01n$$ for $n\geq 700$. 

Next, we bound $|X^*|, |Y^*|, X_{\min}$, and $Y_{\min}$. By \cref{lem:overlineM}(a), we have $||X^*|-|Y^*|| \leq \overline{M}+\Delta \leq 0.01n + 0.76\sqrt{n\log n} \leq 0.1 n$ since $n\geq 700$. So, since also $|X^*|+|Y^*|\geq 0.5n$, we have $|X^*|, |Y^*| \geq 0.2n$. Also, $\overline{X}_{\max}, \overline{Y}_{\max}\leq \overline{M} \leq 0.01n$, so $X_{\min}, Y_{\min} \geq 0.19n$.

We claim that $|U|\leq 0.4n$. Suppose not; then $\frac{7}{8}|U|\geq 0.35n$ and at least one of $|X^*|$ and $|Y^*|$ is at most $0.3n$. Without loss of generality, say $|X^*|\leq 0.3n$; now with $\eta=1/8$ (which is valid by \cref{lem:assumptions>=700}(d)) the second term in the left hand side of \cref{lem:overlineM}(d) is at least $0.19n \cdot (0.05n)^2$, which is greater than the right hand side of at most $\tilde d(n) \leq n+0.48n\log n$ when $n\geq 700$, a contradiction.

We now use our bounds so far to tighten these estimates further. %Next, note that $|U|(|X^*|+|Y^*|) \leq 0.25n^2$ since $|U|+|X^*|+|Y^*|=n$. 
We claim $|X^*| \geq 0.27n$. If not, then since $||X^*|-|Y^*||\leq 0.1n$, we have $|Y^*| < 0.37n$, so $|U|> 0.36n$. But then the second term of the left hand side of \cref{lem:overlineM}(d) is at least $0.19n\cdot(0.045n)^2$, which is bigger than the right hand side of at most $n+0.48n\log n$ for $n\geq 700$. Thus $X_{\min} \geq 0.27n - \overline{M} \geq 0.26n$. Symmetrically, $Y_{\min}\geq 0.26$. Furthermore, since $|U| \leq 0.4n$, we know $|X^*|$ and $|Y^*|$ must average at least $0.3n$, so $X_{\min}$ and $Y_{\min}$ must average at least $0.3n-0.01n=0.29n$, which shows that $X_{\min}Y_{\min} \geq 0.26n \cdot 0.32n \geq 0.083n^2$. Putting these bounds together yields 
\[C[b_{XY}] \leq 1+\frac{0.01n}{0.26n}+\frac{0.01n}{0.32n} +\frac{0.4n\cdot 0.6n}{0.083n^2} \leq 3.962.\]

Since $|U| \geq |X|-1.2\log n$ and $|U|\geq |Y|-1.2\log n$, and we also have $|X^*|\leq X_{\min}+0.01n$ and $|Y^*|\leq Y_{\min}+0.01n$, we have that $n = |U|+|X^*|+|Y^*| \leq 3|U|+0.02n+2.4\log n$.
This gives $|U| \geq 0.31n$.

We now bound $C[\floor{P/2}-b_{XY}]$. The first factor is at most $1+\frac{0.01}{0.26}+\frac{0.01}{0.32} \leq 1.07$. The second factor is at most $\frac{1}{n}(\frac{1}{0.26}+\frac{1}{0.32}+\frac{1}{0.31})\leq \frac{10.2}{n}.$ Hence $C[\floor{P/2}-b_{XY}] \leq 10.92/n$.

We now obtain a weak upper bound on $\tau$. Since $10.92/n < 3.962$, we have $$C[b_{XY}]\cdot b_{XY} + C[\floor{P/2}-b_{XY}]\cdot(\floor{P/2}-b_{XY}) \leq 3.962 \floor{P/2} \leq 3.962 \cdot 2d(n).$$
Thus, $\tau \leq 3.962\cdot 0.12n\log n+ (0.01 n)^2 \leq 0.01 n^2$ for $n\geq 700$. 

To improve this, we need to bound $b_{XY}$ using \eqref{eq:bounding bad 1} from \cref{lem:tau-bad}, so we first use \cref{lem:bmaxfromx} to get a good bound on $\max_{X^*}b_U$ and $\max_{Y^*} b_U$. We first verify \eqref{eq:not_enough_link}. By convexity of $\bip(x)$, to maximize the left hand side of \eqref{eq:not_enough_link} subject to $0.31n \leq |U|\leq 0.4n$, $||X^*|-|Y^*||\leq0.1 n$, and $|X^*|+|Y^*|+|U| = n$, we set $|U| = 0.4n, |X^*| = 0.35n,$ and $|Y^*|=0.25n$, yielding $$\bip(|X^*|)+\bip(|Y^*|)+\bip(|U|)+\tau \leq 0.25(0.35^2+0.25^2+0.4^2)n^2 + 0.01n^2 \leq 0.1 n^2.$$ On the other hand, the right hand side of \eqref{eq:not_enough_link} is $T(n)-d(n)-T(n-1)+d(n-1) \geq (n^2/8-n/4)-0.06n\log n\geq 0.123 n^2$ for $n\geq 700$. This verifies \eqref{eq:not_enough_link}.

Thus, by \cref{lem:bmaxfromx}, $\max_{X^*}b_U$ is at most $b_{\max}$, the largest nonnegative integer $b$ such that $b\leq \abs{U}$ and \begin{equation}\label{eq:bmax} \left(\abs{U}-b\right)\abs{Y^*}+\bip(\abs{X^*}-1)+\bip(b)+\tau-b\geq g_3(n)-g_3(n-1)+1.\end{equation} We claim that $b_{\max} \leq 0.45|U|$.

Suppose $b = c|U|$ for some $c\in[0,1]$. Then the left hand side of \eqref{eq:bmax} is at most $(1-c)|U||Y^*| + 0.25|X^*|^2 + 0.25c^2|U|^2 + 0.01n^2$, while the right hand side is at least $0.123n^2$. For fixed $|U|, |X^*|,$ and $|Y^*|$, our bound on the left hand side is quadratic in $c$ and the leading coefficient is positive, meaning that it is convex in $c$. Consider first $c=1$, which makes the expression simplify to $0.25(|X^*|^2+|U|^2)+0.01n^2$. With all our constraints, this quantity is maximized by letting $|U| = 0.4n$ and $|X^*| = 0.35n$, but then the expression is still less than $0.1 n^2$, so it cannot exceed the right hand side. Thus for fixed $|X^*|, |Y^*|$, and $|U|$, %we assume our bound on the left hand side is maximized by $c=0$. In particular %\hy{Hmmm I don't see why this is an in particular. I think the following is a consequence that the inequality holds when $c=1$ as verified before?}
if we show that the left hand side cannot exceed the right hand side when $c=0.45$, then the left hand side cannot exceed the right hand side for any $c\in [0.45, 1]$ by convexity. This is what we do next. 

When $c=0.45$, our bound on the left hand side of \eqref{eq:bmax} is $0.55|U||Y^*|+0.25|X^*|^2+0.25(0.45)^2|U|^2+0.01n^2$. For fixed $|U|$, $|X^*|+|Y^*|$ is fixed, so the expression is quadratic in $|X^*|$ with positive leading coefficient and is thus maximized either when $|X^*|$ is as large as possible or as small as possible.
Let $\abs{U} = \alpha n$.
Since $||X^*|-|Y^*||\leq 0.1n$, we thus only need to check $|X^*| = (0.55-0.5\alpha )n, |Y^*| = (0.45-0.5\alpha )n$ and vice versa. 
When  $|X^*| = (0.55-0.5\alpha)n$ and $|Y^*| = (0.45-0.5\alpha)n$, we have
\[0.55|U||Y^*|+0.25|X^*|^2+0.25(0.45)^2|U|^2+0.01n^2 =(- 0.161875 \alpha^2+0.11\alpha+0.085625)n^2\leq 0.105n^2\]
by direct computation.
When $|X^*| = (0.45-0.5\alpha)n$ and $|Y^*| = (0.55-0.5\alpha)n$, we similarly have
\[0.55|U||Y^*|+0.25|X^*|^2+0.25(0.45)^2|U|^2+0.01n^2 =(- 0.161875 \alpha^2+0.19\alpha+0.060625)n^2\leq 0.121n^2.\]
Since the right hand side is at least $0.123n^2$, this completes the proof that $\max_{X^*}b_U\leq 0.45|U|$. Symmetrically, $\max_{Y^*}b_U \leq 0.45|U|$ as well. 

Now, \eqref{eq:bounding bad 1} tells us that $(1-0.45-0.45)|U| b_{XY} \leq 2d(n)$, so since $|U|\geq 0.31n$, we have $0.031b_{XY}\leq 2d(n)/n \leq 0.12\log n$, so $b_{XY} \leq 3.871\log n$. 

Putting all of our bounds together and using $\overline{X}_{\max}, \overline{Y}_{\max} \leq \overline{M} \leq 0.7\log n$, we obtain
$$\tau \leq 3.962\cdot 3.871\log n + \frac{10.92}{n}\cdot 0.12n\log n + (0.7\log n)^2,$$ which is at most $0.3n$ when $n\geq 700$. Since $m = \min(|X^*|, |Y^*|, |U|) \geq 0.27n$, we have $\tau \leq 2m$, which completes the proof.
\end{proof}
\end{appendices}
\end{document}